\documentclass[12pt]{article}
\title{Empirical estimation of entropy\\ functionals with confidence}
\pdfoutput=1
\author{\textbf{Kumar Sricharan}, Department of EECS, University of Michigan
\\ \textbf{Raviv Raich}, School of EECS, Oregon State University
\\ \textbf{Alfred O. Hero III}, Department of EECS, University of Michigan
}

\date{February 25, 2011}

\usepackage[hmargin=2.5cm,vmargin=3cm]{geometry}
\usepackage[british]{babel}
\usepackage{latexsym,amssymb,amsmath}
\usepackage[mathcal]{euscript}
\usepackage{amsmath, amsthm, amssymb}
\usepackage{graphicx}
\usepackage[absolute,overlay]{textpos}
\usepackage[mathcal]{euscript}
\usepackage{algorithm}
\usepackage{algorithmic}
\usepackage{appendix}
\usepackage{color}
\usepackage{subfigure}

\newcommand{\expect}{{{\mathbb{E}}}}
\newcommand{\mb}{\mathbf}
\newcommand{\var}{{{\mathbb{V}}}}

\newcommand{\es}{\epsilon}
\newcommand{\gleq}{\,\raisebox{-.6ex}{$\stackrel{>}{\scriptstyle <}$}\,}

\newtheorem{theorem}{Theorem}[section]
\newtheorem{lemma}[theorem]{Lemma}

\newtheorem{corollary}[theorem]{Corollary}

\begin{document}
\maketitle



\begin{abstract}
This paper introduces a class of k-nearest neighbor ($k$-NN) estimators called bipartite plug-in (BPI) estimators for estimating integrals of non-linear functions of a probability density, such as Shannon entropy and R\'enyi entropy. The density is assumed to be smooth, have bounded support, and be uniformly bounded from below on this set. Unlike previous $k$-NN estimators of non-linear density functionals, the proposed estimator uses data-splitting and boundary correction to achieve lower mean square error. Specifically, we assume that $T$ i.i.d. samples $\mb{X}_i \in \mathbb{R}^d$ from the density are split into two pieces of cardinality $M$ and $N$ respectively, with $M$ samples used for computing a k-nearest-neighbor density estimate and the remaining $N$ samples used for empirical estimation of the integral of the density functional.   By studying the statistical properties of k-NN balls, explicit rates for the bias and variance of the BPI estimator are derived in terms of the sample size, the dimension of the samples and the underlying probability distribution. Based on these results, it is possible to specify optimal choice of tuning parameters $M/T$, $k$ for maximizing the rate of decrease of the mean square error (MSE). The resultant optimized BPI estimator  converges faster and achieves lower mean squared error than previous $k$-NN entropy estimators. In addition, a central limit theorem is established for the BPI estimator that allows us to specify tight asymptotic confidence intervals.
\end{abstract}

\section{Introduction}
Non-linear functionals of a multivariate density $f$ of the form $\int g(f(x),x) f(x) dx$ arise in applications including machine learning, signal processing, mathematical statistics, and statistical communication theory. Important examples of such functionals include Shannon and R\'enyi entropy. Entropy based applications for image matching, image registration and texture classification are developed in~\cite{heroapp, neem}. Entropy functional estimation is fundamental to independent component analysis in signal processing~\cite{milfish}. Entropy has also been used in Internet anomaly detection~\cite{lakh} and data and image compression applications~\cite{jainak}. Several entropy based nonparametric statistical tests have been developed for testing statistical models including uniformity and normality~\cite{vasi,dude}. Parameter estimation methods based on entropy have been developed in~\cite{amin,ranneby}.  For further applications, see, for example, Leonenko~{\it etal}~\cite{leo2}. 

In these applications, the functional of interest must be estimated empirically from sample realizations of the underlying densities. Several estimators of entropy measures have been proposed for general multivariate densities $f$. These include consistent estimators based on entropic graphs \cite{hero, pal}, gap estimators \cite{vanes}, nearest neighbor distances \cite{leo, leo2, litt, wang}, kernel density plug-in estimators ~\cite{ahmad,egg,bickel,hallmarr,birge,gini}, Edgeworth approximations~\cite{edge}, convex risk minimization~\cite{long2} and orthogonal projections~\cite{laurent}.

The class of density-plug-in estimators considered in this paper are based on $k$-nearest neighbor ($k$-NN) distances and, more specifically, bipartite k-nearest neighbor graphs over the random sample. The basic construction of the proposed bipartite plug-in (BPI) estimator is as follows (see Sec. II.A for a precise definition). Given a total of $T$ data samples we split the data into two parts of size $N$ and size $M$, $N+M=T$. On the part of size $M$ a $k$-NN density estimate is constructed. The density functional is then estimated by plugging the $k$-NN density estimate into the functional and approximating the integral by an empirical average  over the remaining $N$ samples. This can be thought of as computing the estimator over a bipartite graph with the $M$ density estimation nodes connected to the $N$ integral approximating nodes. The BPI estimator exploits a close relation between density estimation and the geometry of proximity neighborhoods in the data sample. The BPI estimator is designed to automatically incorporate boundary correction, {\emph{without}} requiring prior knowledge of the support of the density. Boundary correction compensates for bias due to distorted $k$-NN neighborhoods that occur for points  near the boundary of the density support set. Furthermore, this boundary correction is {\emph{adaptive}} in that we achieve the same MSE rate of convergence that can be attained using an oracle BPI estimator having knowledge of boundary of the support. Since the rate of convergence relates the number of samples $T =N+M$  to the performance of the estimator, convergence rates have great practical utility. A statistical analysis of the bias and variance, including rates of convergence, is presented for this class of boundary compensated BPI estimators. In addition, results on weak convergence (CLT) of BPI estimators are established. These results are applied to optimally select estimator tuning parameters $M/T, k$ and to derive confidence intervals. {For arbitrary smooth functions $g$, we show that by choosing $k$ increasing in $T$ with order $O(T^{-2/(2+d)})$, an optimal MSE rate of order $O(T^{-4/(2+d)})$ is attained by the BPI estimator. For certain specific functions $g$ including Shannon entropy ($g(u) = \log(u)$) and R\'enyi entropy ($g(u) = u^{\alpha-1}$), a faster MSE rate of order $O(((\log T)^{6}/T)^{4/d})$ is achieved by BPI estimators by correcting for bias.} 

\subsection{Previous work on $k$-NN functional estimation}

The authors of~\cite{sing, leo, leo2, litt} propose $k$-NN estimators for Shannon entropy ($g(u) = \log(u)$) and R\'enyi entropy($g(u) = u^{\alpha-1}$). Evans~{\it etal}~\cite{evjo} consider positive moments of the $k$-NN distances ($g(u) = u^k, k \in \mathbb{N}$). Recently, Baryshnikov~{\it etal}~\cite{bar} proposed $k$-NN estimators for estimating $f$-divergence $\int \phi(f_0(x)/f(x)) f(x) dx $ between an unknown density $f$, from which sample realizations are available, and a known density $f_0$. Because $f_0$ is known, the $f$-divergence $\int \phi(f_0(x)/f(x)) f(x) dx $ is equivalent to a entropy functional $\int g(f(x),x) dx$ for a suitable choice of $g$. Wang~{\it etal}~\cite{wang} developed a $k$-NN based estimator of $\int g(f_1(x)/f_2(x),x) f_2(x) dx$ when both $f_1$ and $f_2$ are unknown. The authors of these works ~\cite{sing,leo,evjo,wang} sestablish that the estimators they propose are asymptotically unbiased and consistent. The authors of ~\cite{litt} analyze estimator bias for $k$-NN estimation of Shannon and R\'enyi entropy. For smooth functions $g(.)$, Evans~{\it etal}~\cite{evans} show that the variance of the sums of these functionals of $k$-NN distances is bounded by the rate $O(k^5/T)$. Baryshnikov~{\it etal}~\cite{bar} improved on the results of Evans~{\it etal} by determining the exact variance up to the leading term ($c_k/T$ for some constant $c_k$ which is a function of $k$). Furthermore, Baryshnikov~{\it etal} show that the entropy estimator they propose converges weakly to a normal distribution. However, Baryshnikov~{\it etal} do not analyze the bias of the estimators, nor do they show that the estimators they propose are consistent. Using the results obtained in this paper, we provide an expression for this bias in Section~\ref{sec:compareb} and show that the optimal MSE for Baryshnikov's estimators is $O(T^{-2/(1+d)})$.

In contrast, the main contribution of this paper is the analysis of a general class of BPI estimators of smooth density functionals.  We provide asymptotic bias and variance expressions and a central limit theorem. The bipartite nature of the BPI estimator enables us to correct for bias due to truncation of $k$-NN neighborhoods near the boundary of the support set; a correction that does not appear straightforward for previous $k$-NN based entropy estimators. We show that the BPI estimator is MSE consistent and that the MSE is guaranteed to converge to zero as $T \rightarrow \infty$ and $k \rightarrow \infty$ with a rate that is minimized for a specific choice of $k$, $M$ and $N$ as a function of $T$. Therefore, the thus optimized BPI estimator can be implemented without any tuning parameters. In addition a CLT is established that can be used to construct confidence intervals to empirically assess the quality of the BPI estimator. Finally, our method of proof is very general and it is likely that it can be extended to kernel density plug-in estimators, $f$-divergence estimation and mutual information estimation.

{Another important distinction between the BPI estimator and the $k$-NN estimators of Shannon and R\'enyi entropy proposed by the authors of ~\cite{sing, leo, leo2} is that these latter estimators are  consistent for finite $k$, while the proposed BPI estimator requires the condition that $k \to \infty$ for MSE convergence. {By allowing $k \to \infty$, the BPI estimators of Shannon and R\'enyi entropy achieve MSE rate of order $O(((\log T)^{6}/T)^{4/d})$}.  This asymptotic rate is faster than the $O(T^{-2/d})$ MSE convergence rate~\cite{litt} of the previous $k$-NN estimators~\cite{sing, leo, leo2} that use a fixed value of $k$. It is shown by simulation that BPI's asymptotic performance advantages, predicted by our theory, also hold for small sample regimes. 

\subsection{Organization}

The remainder of the paper is organized as follows. Section~\ref{sec:prel} formulates the entropy estimation problem and introduces the BPI estimator. The main results concerning the bias, variance and asymptotic distribution of these estimators are stated in Section~\ref{sec:main} and the consequences of these results are discussed. The proofs are given in the Appendix. The MSE is analyzed in Section 4. We discuss bias correction of the BPI estimator for the case of Shannon and R\'enyi entropy estimation in Section~\ref{sec:compareRS}. Estimation of Shannon MI is briefly discussed in Section 6. We numerically validate our theory by simulation in Section~\ref{sec:exp}. Applications to structure discovery and dimension estimation are discussed in Sections 8 and 9 respectively. A conclusion is given in Section~\ref{sec:conc}.

\subsubsection*{Notation}
Bold face type will indicate random variables and random vectors and regular type face will be used for non-random quantities. Denote the expectation operator by the symbol $\expect$ and conditional expectation given $\mb{Z}$ by  $\expect_{\mb{Z}}$. Also define the variance operator as $\var[\mb{X}] = \expect[(\mb{X}-\expect[\mb{X}])^2] $ and the covariance operator as $Cov[\mb{X},\mb{Y}] = \expect[(\mb{X}-\expect[\mb{X}])(\mb{Y}-\expect[\mb{Y}])] $. Denote the bias of an estimator by $\mathbb{B}$.

\section{Preliminaries}
\label{sec:prel}
We are interested in estimating non-linear functionals $G(f)$ of $d$-dimensional multivariate densities $f$ with support ${\cal S}$, where $G(f)$ has the form
\begin{equation}
\label{eq:oracle}
G(f) = \int g(f(x),x) f(x) d\mu(x) = \expect[g(f(x),x)], \nonumber
\end{equation}
for some smooth function $g(f(x),x)$. Let ${\cal B}$ denote the boundary of ${\cal S}$. Here, $\mu$ denotes the Lebesgue measure and $\expect$ denotes statistical expectation w.r.t density $f$. We assume that i.i.d realizations  $\{\mb{X}_1, \ldots, \mb{X}_N, \mb{X}_{N+1}, \ldots, \mb{X}_{N+M}\}$ are available from the density $f$.  Neither $f$ nor its support set are known.

The plug-in estimator is constructed using a data splitting approach as follows. The data is randomly subdivided into two parts ${\cal X}_N = \{\mb{X}_1, \ldots, \mb{X}_N\}$ and ${\cal X}_M = \{\mb{X}_{N+1}, \ldots, \mb{X}_{N+M}\}$ of $N$ and $M$ points respectively.
In the first stage, a boundary compensated $k$-NN density estimator ${\tilde{\mb{f}}_k}$ is estimated at the $N$ points $\{\mb{X}_1, \ldots, \mb{X}_N\}$ using the $M$ realizations $\{\mb{X}_{N+1}, \ldots, \mb{X}_{N+M}\}$. Subsequently, the $N$ samples $\{\mb{X}_1, \ldots, \mb{X}_N\}$ are used to approximate the functional $G(f)$ to obtain the basic Bipartite Plug-In (BPI) estimator:
\begin{eqnarray}
\label{eq:plugin}
  \hat{\mb{G}}_N(\mb{\tilde{f}}_k) &=& \frac{1}{N} \sum_{i=1}^N g({\tilde{\mb{f}}{_k}(\mb{X}_i)},\mb{X}_i). 
\end{eqnarray}
As the above estimator performs an average over the $N$ variables $X_i$ of the function $g(\tilde{f}(X_i),X_i)$, which is estimated from the other $M$ variables, this estimator can be viewed as averaging over the edges of a bipartite graph with $N$ and $M$ nodes on its left and right parts.  
\subsection{Boundary compensated {$k$}-NN density estimator}

\label{sec:knn}
Since the probability density $f$ is bounded above, the observations will lie strictly on the interior of the support set ${\mathcal S}$. However, some observations that occur close to the boundary of ${\mathcal S}$ will have $k$-NN balls that intersect the boundary. This leads to significant bias in the $k$-NN density estimator. In this section we describe a method that compensates for this bias. The method can be interpreted as extrapolating the location of the boundary from extreme points in the sample and suitably reducing the volumes of their $k$-NN balls.  

Let $d(X,Y)$ denote the Euclidean distance between points $X$ and $Y$ and $\mb{d}_k(X)$ denote the Euclidean distance between a point X and its $k$-th nearest neighbor amongst the $M$ realizations $\mb{X}_{N+1},..,\mb{X}_{N+M}$. Define a ball with radius $r$ centered at $X$: $S_r(X) =  \{Y :d(X,Y) \leq r\}$. The $k$-NN region is $\mb{S}_k(X) = \{Y:d(X,Y)\leq \mb{d}_k(X)\}$ and the volume of the $k$-NN region is $\mb{\mb{V}}_k(X) = \int_{\mb{S}{_k}(X)}{dZ}$. The standard {$k$-NN} density estimator~\cite{quu} is defined as $$\hat{\mb{f}}_{k}(X) = \frac{k-1}{M\mb{\mb{V}}_k(X)}.$$ 

If a probability density function has bounded support, the $k$-NN balls $\mb{S}_k(X)$ centered at points $X$ close to the boundary may intersect with the boundary ${\cal B}$, or equivalently $\mb{S}_k(X) \cap {\cal S}^c \neq \phi$, where ${\cal S}^c$ is the complement of ${\cal S}$. As a consequence, the $k$-NN ball volume $\mb{\mb{V}}_k(X)$ will tend to be higher for points $X$ close to the boundary leading to significant bias of the $k$-NN density estimator.

{ Let $R_k(X)$ correspond to the coverage value $(1+p_k)k/M$, i.~e.~, $R_k(X)  = \inf \{r: \int_{S_r(X)} f(Z)dZ = (1+p_k)k/M\}$, where $p_k = {\sqrt{6}}/(k^{\delta/2})$ for some fixed $\delta \in (2/3,1)$.  Define $$\epsilon_{BC} = N \exp(-3k^{(1-\delta)}).$$ Define $N_k(X)$ as the region corresponding to the coverage value $(1+p_k)k/M$, i.e. ${N}_k(X) = \{Y:d(X,Y) \leq R_k(X)\}$. Finally, define the interior region ${\cal S}_I$ \begin{equation} {\cal S}_I = \{X \in {\cal S}: {N}_k(X) \cap {\cal S}^c = \phi\}. \label {sbdefine} \end{equation} We show in Appendix B that the bias of the standard $k$-NN density estimate is of order $O((k/M)^{(2/d)})$ for points $X \in {\cal S}_I$ and is of order $O(1)$ at points $X \in {\cal S-S}_I$. } { This motivates the following method for compensating for this bias. This compensation is done in two stages: (i) the set of interior points ${\cal I}_N \subset {\cal X}_N$ are identified using variation in $k$-nearest neighbor distances in Algorithm~\ref{alg1H} (see Appendix B for details) and it is show that ${\cal I}_N \notin {\cal S-S}_I$ with probability $1-O(\epsilon_{BC})$; and (ii) the density estimator at points in ${\cal B}_N = {\cal X}_N - {\cal I}_N$ are corrected by extrapolating to the density estimates at interior points ${\cal I}_N$ that are close to the boundary points. We emphasize that this nonparametric correction strategy { {does not}} assume knowledge about the support of the density $f$.}

For each boundary point $\mb{X}_{i} \in {\cal B}_N$, let $\mb{X}_{n(i)} \in \cal{I}_N$ be the interior sample point that is closest to $\mb{X}_{i}$. The corrected density estimator $\tilde{\mb{{f}}}_k$ is defined as follows. 
\begin{eqnarray}
\label{correcdec3}
\mb{\tilde{f}}_k(\mb{X}_i) = \left\{ \begin{array}{lll}
 \mb{\hat{f}}_k(\mb{X}_i) & \mbox{ $\{\mb{X}_i \in \cal{I}_N\}$} \\
 \mb{\hat{f}}_k(\mb{X}_{n(i)}) & \mbox{ $\{\mb{X}_i \in {\cal B}_N\} $}
       \end{array} \right. 
\end{eqnarray}

\begin{figure}[!t]
\centering
\includegraphics[scale=.28]{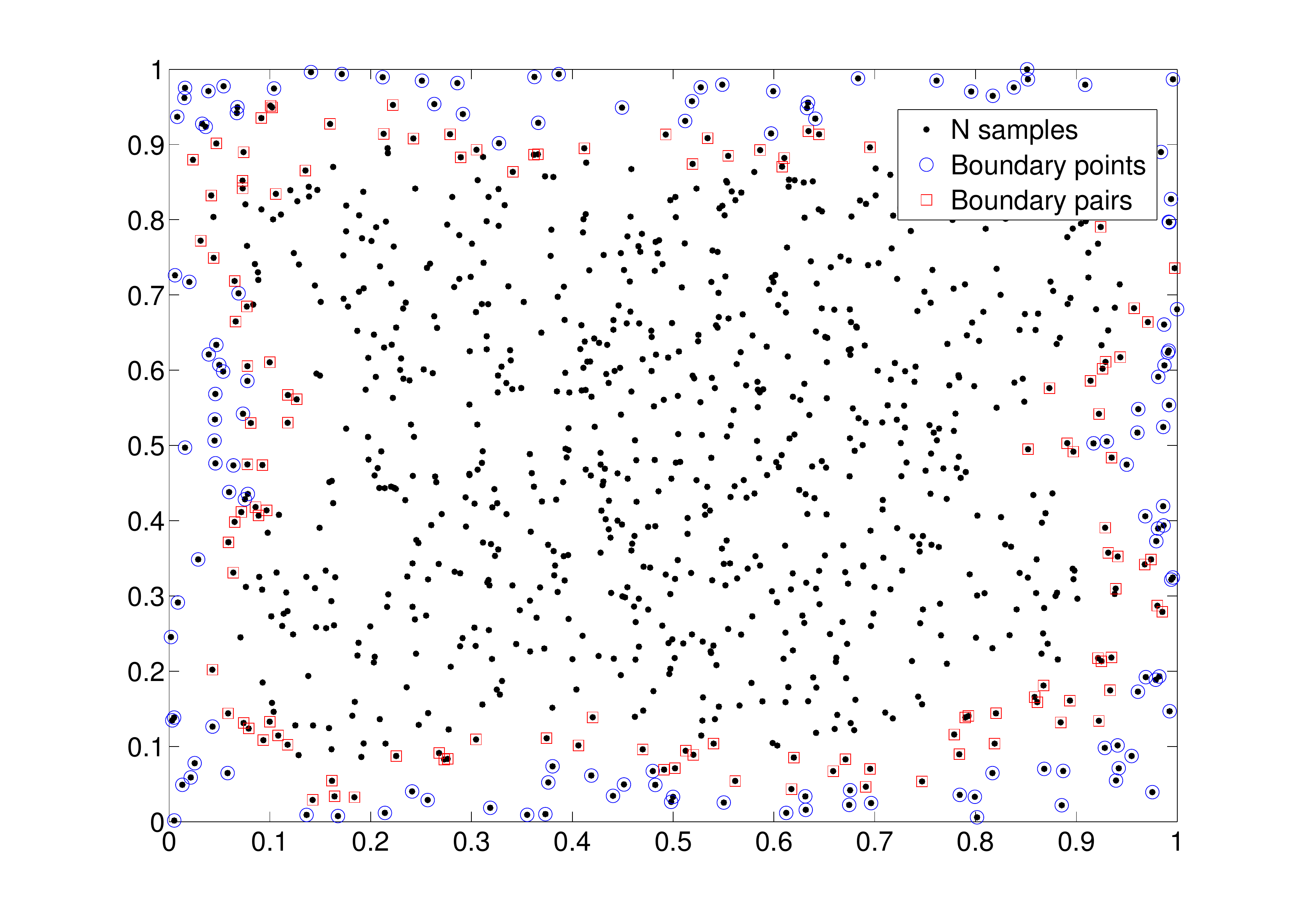}
\label{fig-labelcorrecpairnot}


\caption{Detection of boundary points using Algorithm~\ref{alg1H} for 2d beta distribution.}
\end{figure}

\begin{figure}[!t]
  \begin{center}
    \includegraphics[width=5in]{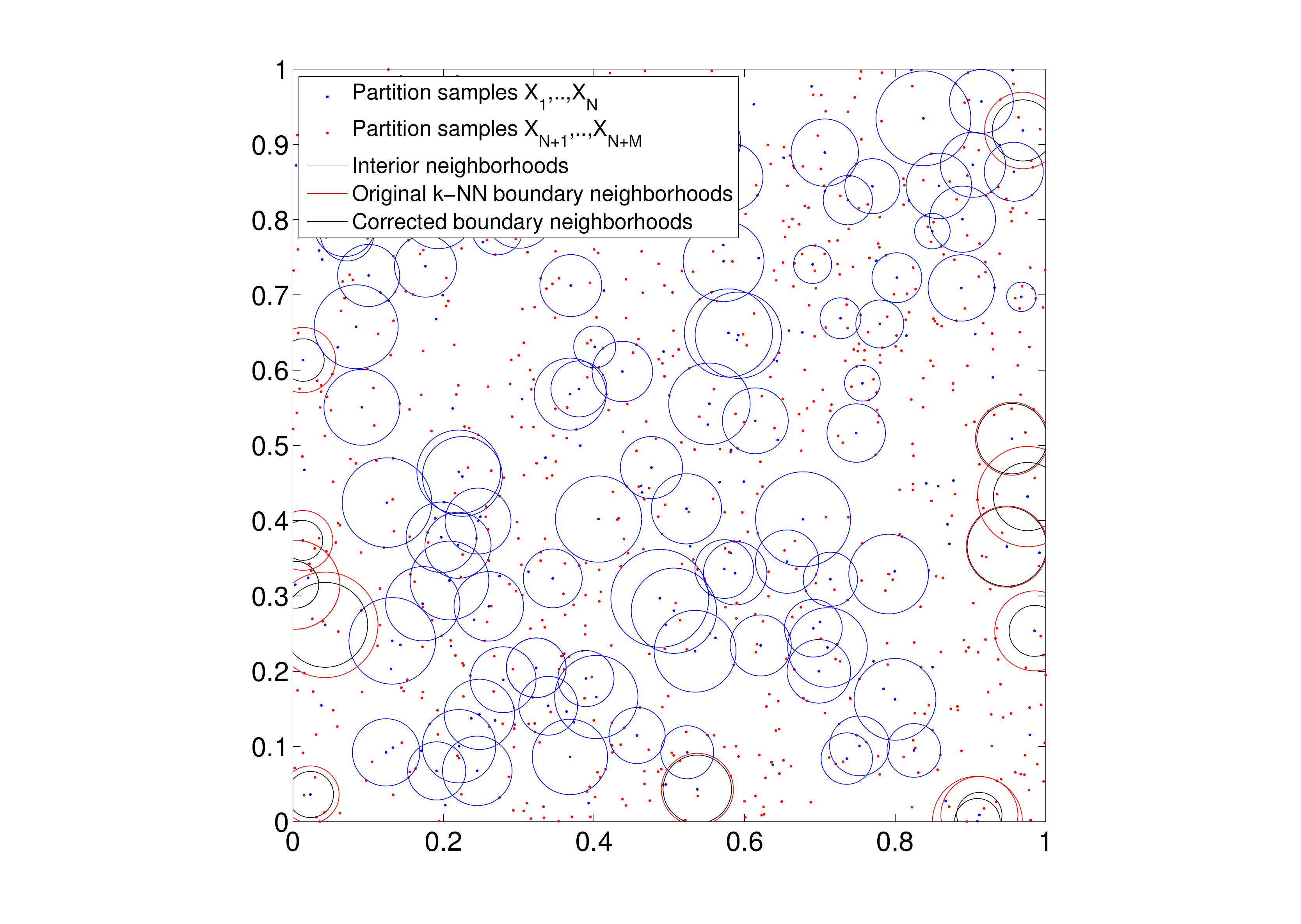}
  \end{center}
  \caption{\small $k$-NN balls centered around a subsample of $2$d uniformly distributed points. Note that the original $k$-NN balls centered at points close to boundary (red) over spill the boundary. The modified $k$-NN neighborhoods (black) corresponding to the corrected corrected density estimate $\mb{\tilde{f}}_k$ compensate for the over spill.}
  \label{fig-knnballs}
\end{figure}

\section{Main results}
\label{sec:main}
Let $\mb{Z}$ denote an independent realization drawn from $f$. Also, define $\mb{Z}_{-1} \in {\cal S}_I$ to be $\mb{Z}_{-1} = \text{arg} \min_{x \in {\cal S}_I} d(x,{\mb{Z}})$.  Define $h(X) = \Gamma^{(2/d)}((d+2)/2)f^{-2/d}(X)tr[\nabla^2(f(X))]$. Denote the $n$-th partial derivative of $g(x,y)$ wrt $x$ by $g^{(n)}(x,y)$. Also, let $g'(x,y) := g^{(1)}(x,y)$ and $g''(x,y) := g^{(2)}(x,y)$. For some fixed $0 < \epsilon< 1$, define $p_l = ((k-1)/M)(1-\epsilon)\epsilon_0$ and $p_u = ((k-1)/M)(1+\epsilon)\epsilon_\infty $. Also define $\epsilon_1 = 1/(c_d{\cal D}^d)$, where ${\cal D}$ is the diameter of the bounded set ${\cal S}$ and define $q_l = ((k-1)/M)\epsilon_1$ and $q_u = (1+\epsilon)\epsilon_\infty$. Let $\mb{p}$ be a beta random variable with parameters $k,M-k+1$. 
\subsection{Assumptions}  
\label{sec:assump}  
$({\cal {A}}.0)$ : Assume that $M$, $N$ and $T$ are linearly related through the proportionality constant $\alpha_{frac}$ with: $0 < \alpha_{frac} < 1$, $M = \alpha_{frac}T$ and $N = (1-\alpha_{frac})T$. $({\cal {A}}.1)$ : Let the density $f$ be uniformly bounded away from $0$ and finite on the set ${\cal S}$, i.e., there exist constants $\epsilon_0$, $\epsilon_\infty$ such that $0 < \epsilon_0 \leq f(x) \leq \epsilon_\infty < \infty$ $\forall x \in {\cal S}$. $({\cal {A}}.2)$: Assume that the density $f$ has continuous partial derivatives of order $2\nu$ in the interior of the set ${\cal S}$ where $\nu$ satisfies the condition $(k/M)^{2\nu/d} = o(1/M)$, and that these derivatives are upper bounded. $({\cal {A}}.3)$: Assume that the function $g(x,y)$ has $\lambda$ partial derivatives w.r.t. $x$, where $\lambda$ satisfies the conditions  $k^{-\lambda} = o(1/M)$ and  $O(({\lambda^2((k/M)^{2/d} + 1/M)})/{M}) = o(1/M)$. $({\cal {A}}.4)$: Assume that $\max\{6,2\lambda\} < k <= M$. $({\cal {A}}.5)$: Assume that the absolute value of the functional $g(x,y)$ and its partial derivatives are strictly bounded away from $\infty$ in the range $\epsilon_0 < x < \epsilon_\infty$ for all $y$. $({\cal {A}}.6)$: Assume that $\sup_{x \in (q_l,q_u) } |(g^{(r)}/r!)^2(x,y)|e^{-3k^{(1-\delta)}}  < \infty,$ $\expect[\sup_{x \in (p_l,p_u)} |(g^{(r)}/r!)^2(x/\mb{p},y)|]  < \infty,$ for $r=3, \lambda$.

\subsection{Bias and Variance}
{ 
Below the asymptotic bias and variance of the BPI estimator of general functionals of the density $f$ are specified. These asymptotic forms will be used to establish a form for the asymptotic MSE. 
\begin{theorem}
\label{knnbiasH}
The bias of the BPI estimator $\hat{\mb{G}}_k(f)$ is given by
\begin{eqnarray}
\label{Bias}
\mathbb{B}[\hat{\mb{G}}_N(\mb{\tilde{f}}_k)] &=& c_{1}\left({\frac{k}{M}}\right)^{2/d} + c_{2}\left(\frac{1}{k}\right) + c_3(k,M,N) +  O(\epsilon_{BC}) + o\left(\frac{1}{k} + \left(\frac{k}{M}\right)^{2/d}\right), \nonumber 
\end{eqnarray}
where $c_3(k,M,N) = \expect[1_{\{\mb{Z} \in {{\cal S-S}_I}\}} (g(f(\mb{Z}_{-1}),\mb{Z}_{-1}) - g(f(\mb{Z}),\mb{Z}))] = O(k/M)^{2/d}$, and the constants $c_1 = \expect{[g'(f(\mb{Z}),\mb{Z})h(\mb{Z})]}$, $c_2 = \expect{[f^2(\mb{Z})g''(f(\mb{Z}),\mb{Z})/2]}$.
\end{theorem}

\begin{theorem}
\label{knnvarH}
The variance of the BPI estimator $\hat{\mb{G}}_N(\mb{\tilde{f}}_k)$ is given by
\begin{eqnarray}
\label{Variance}
\var[\hat{\mb{G}}_N(\mb{\tilde{f}}_k)] &=& c_4\left(\frac{1}{N}\right)+ c_5\left(\frac{1}{M}\right) + O(\epsilon_{BC}) + o\left(\frac{1}{M} + \frac{1}{N}\right), \nonumber
\end{eqnarray}
where the constants $c_4=\var[g(f(\mb{Z}),\mb{Z})]$ and $c_5=\var[f(\mb{Z})g'(f(\mb{Z}),\mb{Z})]$.
\end{theorem}

\begin{proof}
We briefly sketch the proof here. The above theorems have been stated more generally and proved in Appendix D. The principal idea here involves Taylor series expansions of the functional $g(\tilde{\mb{f}}_{k}(X),X)$ about the true value $g({{f}(X)},X)$, and subsequently (a) using the moment properties of density estimates derived in Appendix A to obtain the leading terms, and (b) bounding the remainder term in the Taylor series and showing that it can be ignored in comparison to the leading terms. 
\end{proof}

The leading terms $c_1(k/M)^{2/d} + c_2/k$ arise due to the bias and variance of $k$-NN density estimates respectively (see Appendix A), while the term $c_3(k,M,N)$ arises due to boundary correction (see Appendix B). Henceforth, we will refer to $c_3(k,M,N)$ by $c_3$. It is shown in Appendix B that $c_3 = O((k/M)^{2/d})$ (\ref{IIorder}). The term $O(\epsilon_{BC})$ arises from a concentration inequality that gives the probability of the event ${\cal I}_N \notin {\cal S-S}_I$ as $1-O(\epsilon_{BC})$.  Observe that if $k$ increases logarithmically in $M$, specifically $(\log(M))^{2/(1-\delta)}/k \to 0$, then $ O(\epsilon_{BC}) = o(N/M^3) = o(1/T)$.

The term $c_4/N$ is due to approximation of the integral $\int g(f(x),x)f(x) dx$ by the sample mean $(1/N) \sum_{i=1}^{N} g(f(\mb{X}_i),\mb{X}_i)$. The term $c_5/M$ on the other hand is due to the covariance between density estimates $\tilde{\mb{f}}(\mb{X}_i)$ and $\tilde{\mb{f}}(\mb{X}_j)$, $i \neq j$.

The constants $c_2, c_4$ and $c_5$ are once again functionals of the form $\int \tilde{g}(f(x),x) f(x) d\mu(x)$ and can be estimated using the proposed BPI estimator (\ref{eq:plugin}). On the other hand, the constant $c_1$ requires estimation of second order partial derivatives of $f$ in addition to estimating the density $f$. The partial derivatives might be estimated using the methods described in~\cite{raykar}, $c_{1}$ could in principle be estimated in this manner. 

To estimate $c_3$, we observe that $||\mb{Y} - \mb{Y}_{-1}|| = O((k/M)^{1/d})$ with probability $1-O(N{\cal C}(k))$, and that $Pr(\mb{Y} \in {{\cal S-S}_I}) = O((k/M)^{1/d})$. Let $h = \mb{Y} - \mb{Y}_{-1}$. Then,
\begin{eqnarray}
c_3 &=& \expect[1_{\{\mb{Y} \in {{\cal S-S}_I}\}} (g(f(\mb{Y}_{-1}),\mb{Y}_{-1}) - g(f(\mb{Y}),\mb{Y}))] \nonumber \\
&=& \expect[1_{\{\mb{Y} \in {{\cal S-S}_I}\}} g'(f(\mb{Y}_{-1}),\mb{Y}_{-1})(f(\mb{Y}) - f(\mb{Y}_{-1}))]  + O((k/M)^{3/d}) + O({\cal C}(k)) \nonumber \\
&=& \expect[1_{\{\mb{X}_1 \in {{\cal S-S}_I}\}} g'(f(\mb{Y}_{-1}),\mb{Y}_{-1}) <\nabla f(\mb{Y}_{-1}),h>]  + O((k/M)^{3/d}) + O({\cal C}(k)) \nonumber.
\end{eqnarray}
The constant $c_{3}$ can then be estimated as
$$\hat{c}_{3} = (1/N) \sum_{\mb{X}_i \in {\cal B}_N} g'(\hat{\mb{f}}_k(\mb{X}_{n(i)}),\mb{X}_{n(i)}) <\widehat{\nabla f}(\mb{X}_{n(i)}),\mb{X}_i - \mb{X}_{n(i)}>,$$
where the estimate $\widehat{\nabla f}$ of the gradient ${\nabla f}$ of $f$ might once again be estimated using the methods described in~\cite{raykar}.

\subsection{Central limit theorem}
In addition to the results on bias and variance shown in the previous section, it is shown here that the BPI estimator, appropriately normalized, weakly converges to the normal distribution. The asymptotic behavior of the BPI estimator is studied under the following limiting conditions: (a) $k/M \to 0$, (b) $k \to \infty$ and (c) $N \to \infty$. As shorthand, the above limiting assumptions will be collectively denoted by $\Delta\to 0$. 
\begin{theorem}
\label{knncltH}
The asymptotic distribution of the BPI estimator $\hat{\mb{G}}_N(\mb{\tilde{f}}_k)$ is given by
\begin{eqnarray}   
\label{CLT}
\lim_{\Delta \to 0} Pr\left(\frac{\hat{\mb{G}}_N(\mb{\tilde{f}}_k)-\expect[\hat{\mb{G}}_N(\mb{\tilde{f}}_k)]}{{\sqrt{\var[\hat{\mb{G}}_N(\mb{\tilde{f}}_k)]}}} \leq \alpha \right) = Pr(\mb{S} \leq \alpha), \nonumber
\end{eqnarray}
where $\mb{S}$ is a standard normal random variable.
\end{theorem}

\begin{proof}
Define the random variables $\{\mb{Y}_{M,i}; i = 1,\ldots,N\}$ for any fixed $M$ 
\begin{equation}
\mb{Y}_{M,i} = \frac{g({\mb{\tilde{f}}_k(\mb{X}_i)},\mb{X}_i) - \expect[g({\mb{\tilde{f}}_k(\mb{X}_i)},\mb{X}_i)]}{\sqrt{\var[g({\mb{\tilde{f}}_k(\mb{X}_i)},\mb{X}_i)]}}, \nonumber
\end{equation}
The key idea here is to recognize that $\mb{Y}_{M,i}$ are exchangeable random variables. Blum et.al. \cite{chernoff} showed that for exchangeable $0$ mean, unit variance random variables $\mb{Z_i}$, the sum $\mb{S}_N = \frac{1}{\sqrt{N}}\sum_{i=1}^N \mb{Z}_i$ converges in distribution to $N(0,1)$ if and only if $Cov(\mb{Z}_1,\mb{Z}_2) = 0$ and $Cov(\mb{Z}_1^2,\mb{Z}_2^2) = 0$. In our case, 
\begin{eqnarray}
Cov(\mb{Y}_{M,i},\mb{Y}_{M,j}) &=& O(1/M), \nonumber \\
Cov(\mb{Y}_{M,i}^2,\mb{Y}_{M,j}^2) &=& O(1/M). \nonumber 
\end{eqnarray}
As $M$ gets large, we then have that $Cov(\mb{Y}_{M,i},\mb{Y}_{M,j}) \to 0$ and $Cov(\mb{Y}_{M,i}^2,\mb{Y}_{M,j}^2) \to 0$. We then extend the work by Blum et.al. to show that convergence in distribution to $N(0,1)$ holds in our case as both $N$ and $M$ get large. These ideas are rigorously treated in Appendix E.
\end{proof}

The CLT for $k$-NN estimators of R\'enyi entropy was alluded to by Leonenko~et.al.~\cite{leo} by inferring from experimental results. Theorem 3.3 establishes the CLT for BPI estimators of arbitrary functionals, including R\'enyi entropy. This result allows one to define approximate finite sample confidence intervals on the estimated values of the functionals and define p-values .

\section{Analysis of M.S.E}

Theorem \ref{knnbiasH} implies that $k \to \infty$ and $k/M \to 0$ in order that the BPI estimator $\hat{\mb{G}}_N(\mb{\tilde{f}}_k)$ be asymptotically unbiased. Likewise, Theorem \ref{knnvarH} implies that $N \to \infty$ and $M \to \infty$ in order that the variance of the estimator converge to $0$. It is clear from Theorem \ref{knnbiasH} that the MSE is minimized when $k$ grows in polynomially in $M$. Throughout this section, we assume that $k=k_0M^{r}$ for some $r \in (0,1)$. This implies that $O(\epsilon_{BC}) = O(N{\cal C}(k)) = o(1/M) = o(1/T)$.  Figures \ref{fig-label6} and \ref{fig-label7} illustrate the asymptotic behavior of the density estimate and the plug-in estimate with increasing sample size.

\begin{figure}[!t]
  \begin{center}
    \includegraphics[width=3in]{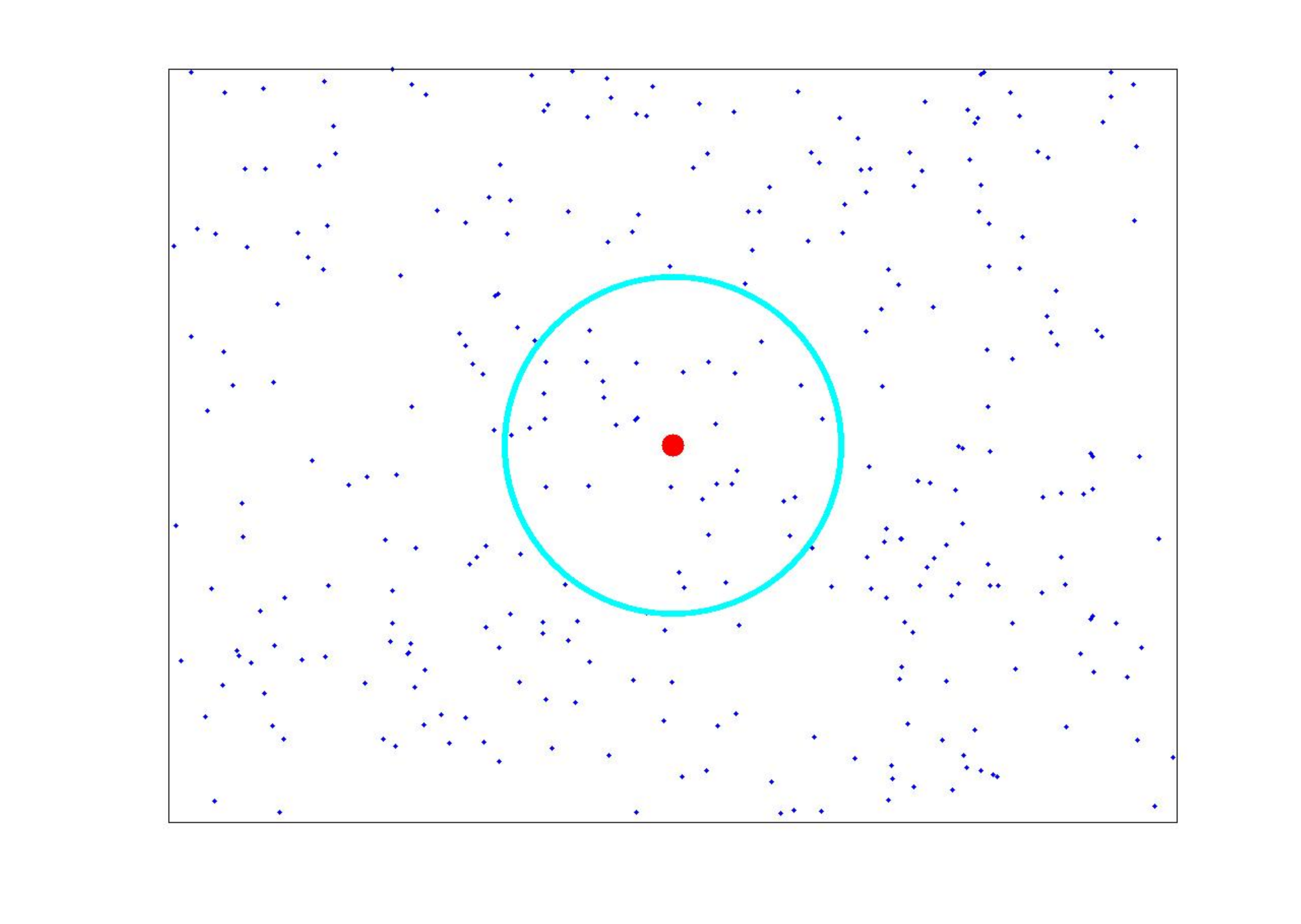}
    \includegraphics[width=3in]{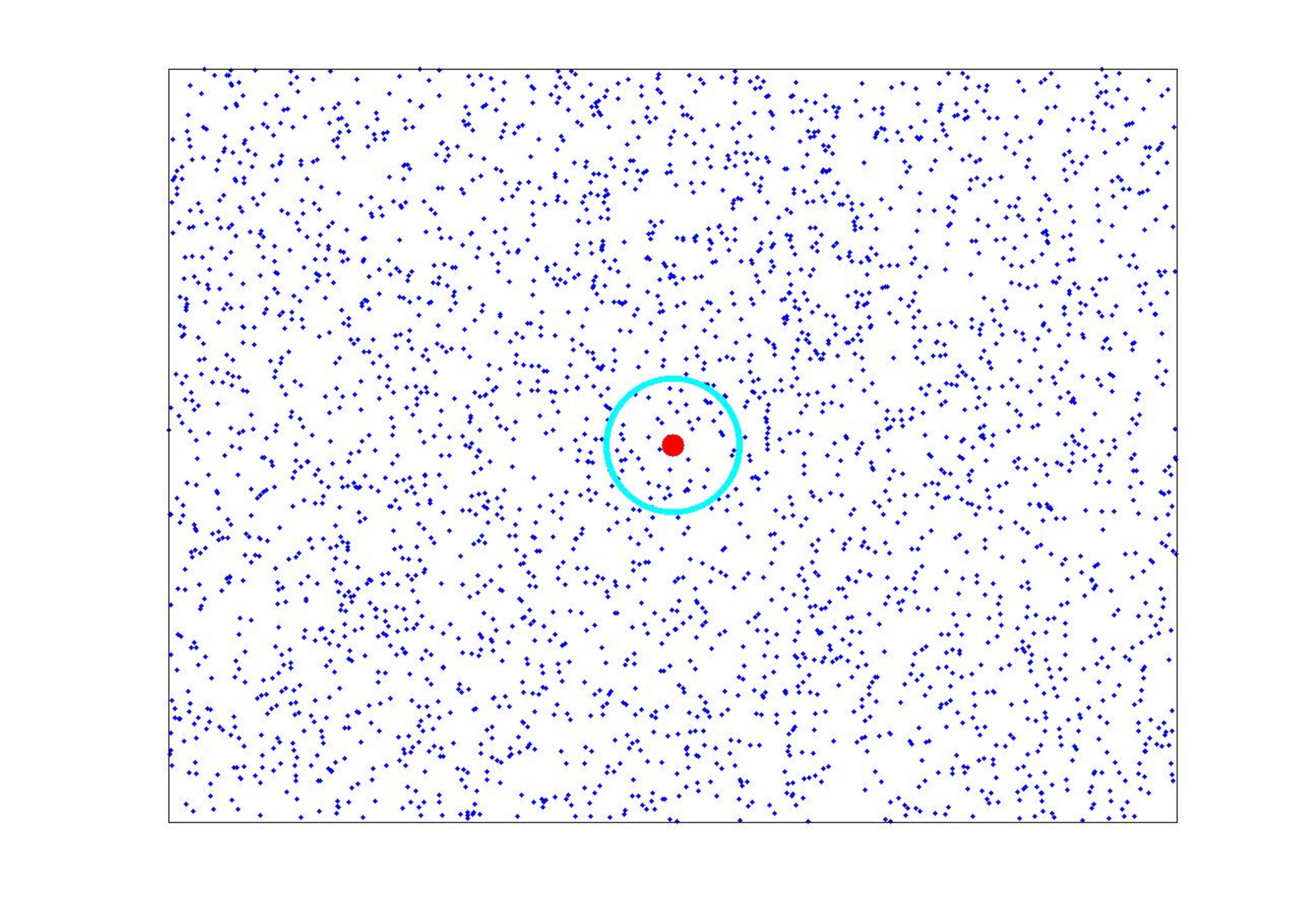}
  \end{center}

  \caption{\small Asymptotics. Variation of density estimate with increasing $k$ and $M$}
  \label{fig-label6}
\end{figure}

\begin{figure}[!t]
  \begin{center}
    \includegraphics[width=3in]{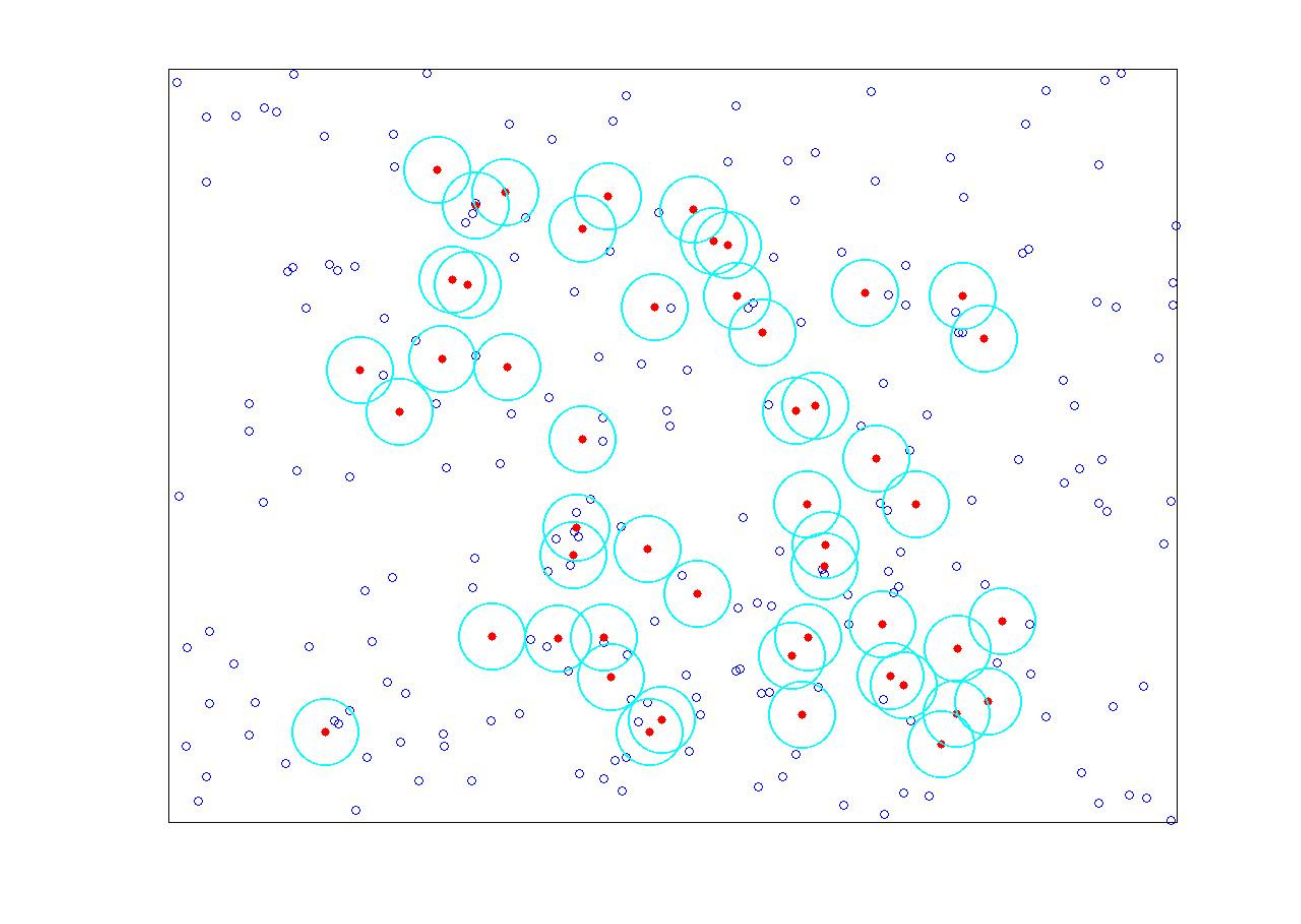}
    \includegraphics[width=3in]{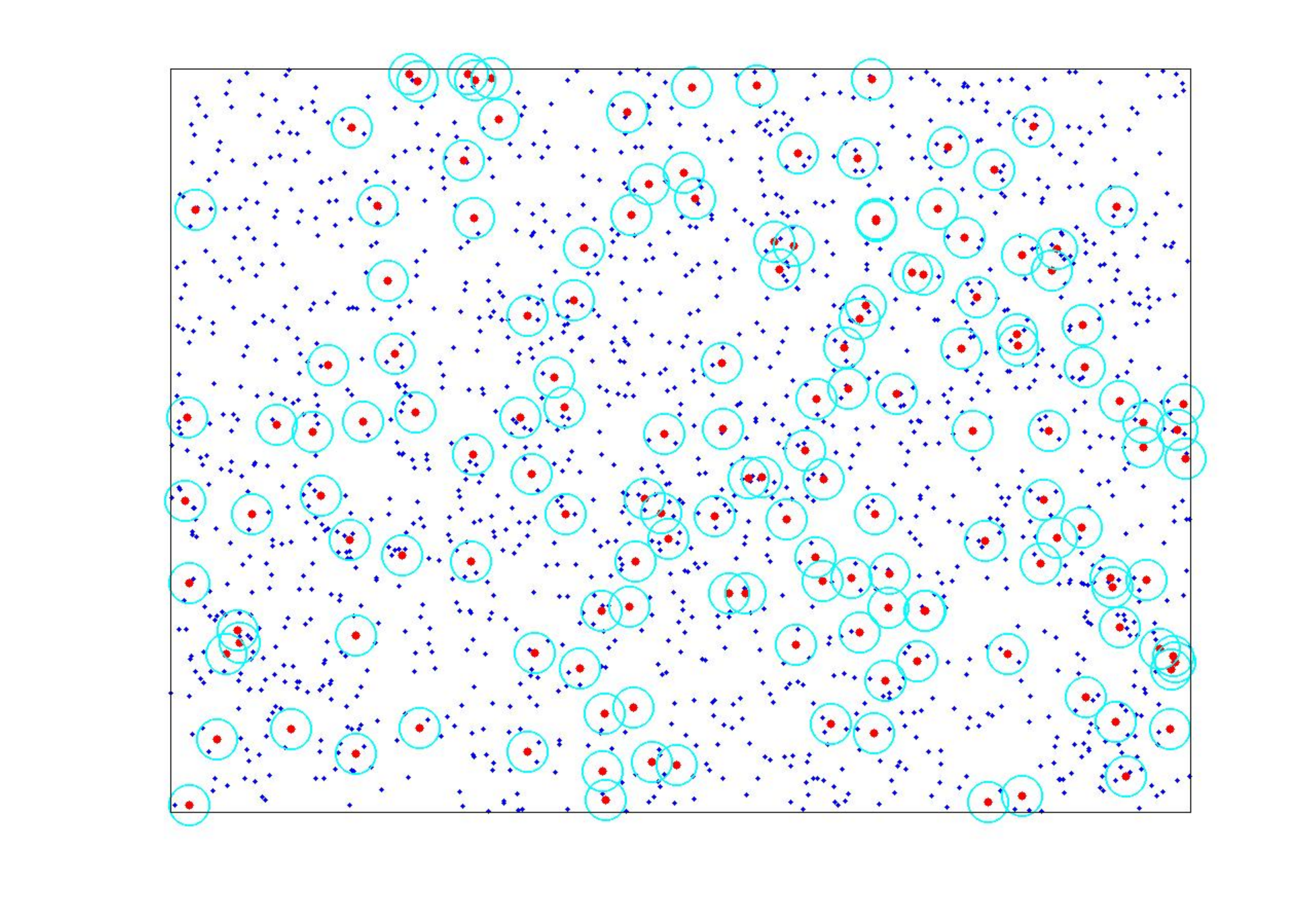}
  \end{center}

  \caption{\small Asymptotics. Variation of plug-in estimate with increasing $k$, $M$ and $N$}
  \label{fig-label7}
\end{figure}

\subsection{Assumptions}
Under the condition $k=k_0M^{r}$, the assumptions $({\cal {A}}.2)$ and $({\cal {A}}.3)$ reduce to the following equivalent conditions: $({\cal {A}}.2)$: Let the density $f$ have continuous partial derivatives of order $2r$ in the interior of the set ${\cal S}$ where $r$ satisfies the condition ${2r(1-t)/d} > 1$. $({\cal {A}}.3)$: Let the functional $g(x,y)$ have $\lambda$ partial derivatives w.r.t. $x$, where $\lambda$ satisfies the conditions  $t\lambda >1$. 

\subsection{Optimal choice of parameters}
In this section, we obtain optimal values for $k$,$M$ and $N$ for minimum M.S.E.

\subsubsection{Optimal choice of $k$}
Theorems III.1 and III.2 provide an optimal choice of $k$ that minimizes asymptotic MSE. Minimizing the MSE over $k$ is equivalent to minimizing the square of the bias over $k$. Define $c_o = c_1 + c_3/(k/M)^{2/d}.$ The optimal choice of $k$ is given by 
\begin{eqnarray}
k_{opt} &=& \underset{k}{\operatorname{arg\,min}} \, {\mathbb{B}}(\hat{\mb{G}}_N(\mb{\tilde{f}}_k)) = \lfloor{k_0M^{\frac{2}{2+d}}}\rfloor, 
\label{eq:optk}
\end{eqnarray} where $\lfloor x \rfloor$ is the closest integer to $x$, and the constant $k_0$ is defined as $k_0 = (|c_2|d/2|c_0|)^{\frac{d}{d+2}}$ when $c_0c_2 > 0$ and as $k_0 = (|c_2|/|c_0|)^{\frac{d}{d+2}}$ when $c_0c_2 < 0$. 

Observe that the constants $c_0$ and $c_2$ can possibly have opposite signs. When $c_0c_2 > 0$, the bias evaluated at $k_{opt}$ is  $b_0^+M^{\frac{-2}{2+d}}(1+o(1))$ where $b_0^+ = c_0k_0^{2/d}+c_2/k_0$. Let $k_{frac} = k_0M^{\frac{2}{2+d}} - k_{opt}$. When $c_0c_2 < 0$, observe that $c_0 ((k_{frac} + k_{opt})/M)^{2/d} + c_2/(k_{frac} + k_{opt})$ is equal to zero. 
When $c_0c_2 < 0$, a higher order asymptotic analysis is required to specify the bias at the optimal value of $k$. In particular, 
\begin{eqnarray}
\mathbb{B}({\mb{\hat{G}}}_N(\mb{\tilde{f}}_k)) &=& c_{1}\left({\frac{k}{M}}\right)^{2/d} + c_{2}\left(\frac{1}{k}\right) \nonumber \\
&& + h_1\left({\frac{k}{M}}\right)^{4/d} + h_{2}\left(\frac{1}{k^2}\right) + h_3\left(\left(\frac{k}{M}\right)^{2/d}\frac{1}{k}\right) \nonumber \\
&& + o\left( \left(\frac{k}{M}\right)^{4/d} + \frac{1}{k^2} + \left(\frac{k}{M}\right)^{2/d}\frac{1}{k} \right) \nonumber
\end{eqnarray}
where the constants are given by $$h_1 = \expect[(1/2)g''(f(\mb{Y}))h^2(X) + g'(f(\mb{Y}))h_o(\mb{Y})],$$ $$h_{2} = \expect[(2/3)g'''(f(\mb{Y}))f^3(\mb{Y})]$$ and $$h_3 = (1-2/d) \expect[g''(f(\mb{Y}))f(\mb{Y})c(\mb{Y})].$$ 
The bias evaluated at $k_{opt}$ is then given by $b_0^-M^{\frac{-4}{2+d}}(1+o(1))$ where the constant $b_0^- = h_1k_0^{4/d} + (h_2+c_2k_{frac})/k_0^2 + (h_3+2c_1k_{frac}/d)k_0^{2/d-1}$.

{Even though the optimal choice $k_{opt}$ depends on the unknown density $f$ (via the constant $k_0$), { we observe from simulations that simply matching the rates, i.e. choosing $k = \bar{k} = M^{2/(2+d)}$, leads to significant MSE improvement}. This is illustrated in Section~\ref{sec:exp}.}

\subsubsection{Choice of $\alpha_{frac} = M/T$}
Observe that the MSE of $\hat{\mb{G}}_N(\mb{\tilde{f}}_k)$ is dominated by the squared bias $(O(M^{-4/(2+d)}))$ as contrasted to the variance $(O(1/N+1/M))$. This implies that the MSE rate of convergence is invariant to the choice of $\alpha_{frac}$. This is corroborated by the experimental results shown in Fig.~\ref{a-compare}.

\subsubsection{Discussion on optimal choice of $k$}
The optimal choice of $k$ grows at a smaller rate as compared to the total number of samples $M$ used for the density estimation step. Furthermore, the rate at which $k/M$ grows decreases as the dimension $d$ increases. This can be explained by observing that the choice of $k$ primarily controls the bias of the entropy estimator. For a fixed choice of $k$ and $M$ $(k<M)$, one expects the bias in the density estimates (and correspondingly in the estimates of the functional $G(f)$) to increase as the dimension increases. For increasing dimension an increasing number of the $M$ points will be near the boundary of the support set. This in turn requires choosing a smaller $k$ relative to $M$ as the dimension $d$ grows.


\subsection{Optimal rate of convergence}

Observe that the optimal bias decays as $b_0^+(T^{\frac{-2}{2+d}})(1+o(1))$ when $c_0c_2 > 0$ and $b_o^-(T^{\frac{-4}{2+d}})(1+o(1))$ when $c_0c_2 < 0$. The variance decays as $\Theta(1/T)(1+o(1))$. 

\subsection{Comparison with results by Baryshnikov~{\it etal}}
\label{sec:compareb}
Recently, Baryshnikov~{\it etal}~\cite{bar} have developed asymptotic convergence results for estimators of $f$-divergence $G(f_0,f) = \int f(x) \phi(f_0(x)/f(x)) dx$ for the case where $f_0$ is known. Their estimators are based on sums of functionals of $k$-NN distances. They assume that they have $T$ i.i.d realizations from the unknown density $f$, and that $f$ and $f_0$ are bounded away from 0 and $\infty$ on their support. The general form of the estimator of Baryshnikov~{\it etal} is given by $$\tilde{\mb{G}}_N(\mb{\hat{f}}_{kS}) = \frac{1}{T} \sum_{i=1}^T g(\hat{\mb{f}}_{kS}(\mb{X}_i)),$$ where $\hat{\mb{f}}_{kS}(\mb{X}_i)$ is the standard $k$-NN density estimator~\cite{fuk} estimated using the $T-1$ samples $\{\mb{X}_1,..,\mb{X}_T\} - \{\mb{X}_i\}$.

Baryshnikov~{\it etal} do not show that their estimator is consistent and do not analyze the bias of their estimator. They show that the leading term in the variance is given by $c_k/T$ for some constant $c_k$ which is a function of the number of nearest neighbors $k$. Finally they show that their estimator, when suitably normalized, is asymptotically normal. 
In contrast, we assume higher order conditions on continuity of the density $f$ and the functional $g$ (see Section 3) as compared to Baryshnikov~{\it etal} and provide results on bias, variance and asymptotic distribution of data-split $k$-NN functional estimators of entropies of the form $G(f) = \int g(f(x)) f(x) dx$.  Note that we also require the assumption that $f$ is bounded away from 0 and $\infty$ on its support. Because we are able to establish expressions on both the bias and variance of the BPI estimator, we are able to specify optimal choice of free parameters $k,N,M$ for minimum MSE. 

For estimating the functional $G(f) = \int g(f(x)) f(x) dx$, the estimator of Baryshnikov can be used by restricting $f_0$ to be uniform. In Appendix C it is shown that under the additional assumption that $({\cal {A}}.6)$ is satisfied by $\tilde{g} = g$, the bias of $\tilde{\mb{G}}_N(\mb{\hat{f}}_{kS})$ is 
\begin{equation}
\mathbb{B}(\tilde{\mb{G}}_N(\mb{\hat{f}}_{kS})) = O((k/T)^{1/d}) + O(1/k).
\label{eq:baryshbias}
\end{equation}
In contrast, Theorem III.~1 establishes that the bias of the BPI estimator $\hat{\mb{G}}_N(\tilde{\mb{f}}_{k})$ decays as $\Theta((k/M)^{2/d} + 1/k) + O(\epsilon_{BC})$ and the variance decays as $\Theta(1/T)$. The bias of the BPI estimator has a higher exponent ($2/d$ as opposed to $1/d$) and this is a direct consequence of using the boundary compensated density estimator $\tilde{\mb{f}}_{k}$ in place of $\mb{\hat{f}}_{k}$. 

It is clear from \ref{eq:baryshbias} that the estimator of Baryshnikov will be unbiased iff $k \to \infty$ as $T \to \infty$. Furthermore, the optimal rate of growth of $k$ is given by $k=T^{1/(1+d)}$. Furthermore, $c_k = \Theta(1)$ and therefore the overall optimal bias and variance of $\tilde{\mb{G}}_N(\mb{\hat{f}}_{kS})$ is given by $\Theta(T^{-1/(1+d)})$ and $\Theta(T^{-1})$ respectively. On the other hand, the optimal bias of the BPI estimator decays as $b_0^+(T^{\frac{-2}{2+d}})(1+o(1))$ when $c_1c_2 > 0$ and $b_o^-(T^{\frac{-4}{2+d}})(1+o(1))$ when $c_1c_2 < 0$ and the optimal variance decays as $\Theta(1/T)$. The BPI estimator therefore has faster rate of MSE convergence. Experimental MSE comparison of Baryshnikov's estimator against the proposed BPI estimator is shown in Fig.~\ref{a-compare}.

\section{Bias correction factors}
\label{sec:compareRS}
\label{sec:entropy}
When the density functional of interest is the Shannon entropy ($g(u) = -\log(u)$) or the R\'enyi -$\alpha$ entropy($g(u) = u^{\alpha-1}$), a bias correction can be added to the BPI estimator that accelerates rate of convergence. Goria~et.al.~\cite{leo2} and Leonenko~et.al.~\cite{leo} developed consistent Shannon and R\'enyi estimators with bias correction. The authors of ~\cite{litt} analyzed the bias for these estimators. When combined with the results of Baryshnikov~{\it etal}, one can easily deduce the variance of these estimators and establish a CLT.

Let $\hat{\mb{H}}_S$ be the Shannon entropy estimate $\tilde{\mb{G}}_N(\mb{\hat{f}}_{kS})$ with the choice of functional $g(x) = -\log(x)$. Let $\hat{\mb{I}}_{\alpha,S}$ be the estimate of the R\'enyi $\alpha$-integral estimate $\tilde{\mb{G}}_N(\mb{\hat{f}}_{kS})$ with the choice of functional $g(x) = x^{\alpha-1}$. Define $\tilde{\mb{H}}_S = \hat{\mb{H}}_S + [\log(k-1) - \Psi(k)]$, where $\psi(.)$ is the digamma function, and $\tilde{\mb{I}}_{\alpha,S} = [(\Gamma(k+(1-\alpha))/\Gamma(k))(k-1)^{\alpha-1}]^{-1} \hat{\mb{I}}_{\alpha,S}$. Also define the R\'enyi entropy estimator to be $\tilde{\mb{H}}_{\alpha,S} = (1-\alpha)^{-1} \log(\tilde{\mb{I}}{_{\alpha,S}})$.  The estimators $\tilde{\mb{H}}_S$ and $\tilde{\mb{H}}_{\alpha,S}$ are the Shannon and R\'enyi entropy estimators of Goria~{\it {etal}}~\cite{leo} and Leonenko~{\it {etal}}~\cite{leo2} respectively. In~\cite{litt}, it is shown that the bias of $\tilde{\mb{H}}_S$ and $\tilde{\mb{I}}_{\alpha,S}$ is given by $\Theta((k/T)^{1/d})$, while the variance was shown by Baryshnikov~{\it etal} to be  $O(1/T)$. In contrast, by~(\ref{eq:baryshbias}), the bias of $\hat{\mb{H}}_S$ and $\hat{\mb{I}}_{\alpha,S}$ is given by $\Theta((k/T)^{1/d} + (1/k))$~(\ref{eq:baryshbias}). This can be understood as follows. From the results by~\cite{litt}, we have
\begin{equation}
 \mathbb{E}[\hat{\mb{H}}_S] = I - [\log(k-1) - \Psi(k)] + c_{0,0}(k/T)^{1/d} + o((k/T)^{1/d})
\end{equation}
and
\begin{equation}
\label{correctionterm}
 \mathbb{E}[\hat{\mb{I}}{_{\alpha,S}}] = [(\Gamma(k+(1-\alpha))/\Gamma(k))(k-1)^{\alpha-1}] I_\alpha + c_{0,\alpha}(k/T)^{1/d} + o((k/T)^{1/d})
\end{equation}
for some functionals of the density $c_{0,0}$ and $c_{0,\alpha}$. Note that $[(\Gamma(k+(1-\alpha))/\Gamma(k))(k-1)^{\alpha-1}] = 1 + O(1/k)$  and $\Psi(k) = \log(k-1) + O(1/k)$ as $k \to \infty$. From the above equations, the scale factor $[(\Gamma(k+(1-\alpha))/\Gamma(k))(k-1)^{\alpha-1}]$ and the additive factor $[\log(k-1) - \Psi(k)]$ account for the $O(1/k)$ terms in the expressions for bias of $\hat{\mb{H}}_S$ and $\hat{\mb{I}}_{\alpha,S}$, thereby removing the requirement that $k \to \infty$ for asymptotic unbiasedness. These bias corrections can be incorporated into the BPI estimator as follows.

\subsection{Main results}
For a general function $g(x,y)$, if there exist functions $g_1(k,M)$ and $g_2(k,M)$, such that 
\begin{eqnarray}
&(i)& \expect[g((k-1)x/M\mb{p},y)] = g(x,y)g_1(k,M) + g_2(k,M) + o(1/M), \nonumber \\
&(ii)& ((k-1)/M)\expect[g'((k-1)x/M\mb{p},y)\mb{p}^{2/d-1}] = g'(x,y)(k/M)^{2/d} + o((k/M)^{2/d}), \nonumber \\
&(iii)& \lim_{k \to \infty} g_1(k,M) = 1, \nonumber \\
&(iv)& \lim_{k \to \infty} g_2(k,M) = 0, 
\label{eq:gcond}
\end{eqnarray}
then define the BPI estimator with bias correction as 
\begin{eqnarray}
\label{eq:pluginwithbc}
  \hat{\mb{G}}_{N,BC}(\mb{\tilde{f}}_k) &=& \frac{\hat{\mb{G}}_{N}(\mb{\tilde{f}}_k) - g_2(k,M)}{g_1(k,M)}.
\end{eqnarray}

\subsubsection{Bias and Variance}
In addition to the assumptions listed in section~\ref{sec:assump}, assume that $k = O((\log(M))^{2/(1-\delta)})$. Below the asymptotic bias and variance of the BPI estimator with bias correction are specified.
\begin{theorem}
\label{knnbiasHRS}
The bias of the BPI estimator $\hat{\mb{G}}_{N,BC}(\mb{\tilde{f}}_k)$ is given by
\begin{eqnarray}
\label{BiasRS}
\mathbb{B}[\hat{\mb{G}}_{N,BC}(\mb{\tilde{f}}_k)] &=& c_{1}\left({\frac{k}{M}}\right)^{2/d} + c_{3}(k,M,N) + o\left(\left(\frac{k}{M}\right)^{2/d}\right).
\end{eqnarray}
\end{theorem}
\begin{theorem}
\label{knnvarHRS}
The variance of the BPI estimator $\hat{\mb{G}}_{N,BC}(\mb{\tilde{f}}_k)$ is given by
\begin{eqnarray}
\label{VarianceRS}
\var[\hat{\mb{G}}_{N,BC}(\mb{\tilde{f}}_k)] &=& c_4\left(\frac{1}{N}\right)+ c_5\left(\frac{1}{M}\right) + o\left(\frac{1}{M} + \frac{1}{N}\right). \nonumber
\end{eqnarray}
\end{theorem}
\subsubsection{CLT}
\begin{theorem}
\label{knncltHRS}
The asymptotic distribution of the BPI estimator $\hat{\mb{G}}_{N,BC}(\mb{\tilde{f}}_k)$ is given by
\begin{eqnarray}   
\label{CLTRS}
\lim_{\Delta \to 0} Pr\left(\frac{\hat{\mb{G}}_{N,BC}(\mb{\tilde{f}}_k)-\expect[\hat{\mb{G}}_{N,BC}(\mb{\tilde{f}}_k)]}{{\sqrt{\var[\hat{\mb{G}}_{N,BC}(\mb{\tilde{f}}_k)]}}} \leq \alpha \right) = Pr(\mb{S} \leq \alpha), \nonumber
\end{eqnarray}
where $\mb{S}$ is a standard normal random variable.
\end{theorem}

\subsubsection{MSE}
Theorem IV.~1 specifies the bias of the BPI estimator, $\hat{\mb{G}}_{N,BC}(\mb{\tilde{f}}_k)$, as $\Theta((k/M)^{2/d})$. Theorem IV.~2 specifies the variance as $\Theta(1/N+1/M)$. By making $k$ increase logarithmically in $M$, specifically, $k = O((\log(M))^{2/(1-\delta)})$ for any value $\delta \in (2/3,1)$, the MSE is given by the rate $\Theta(((\log(T))^{2/(1-\delta)}/T)^{4/d})$. The BPI estimator therefore has a faster rate of convergence in comparison to both Baryshnikov~{\it etal}'s estimators $\hat{\mb{H}}_{S}$ and $\hat{\mb{I}}_{\alpha,S}$ (MSE $= \Theta(T^{-2/(1+d)})$) and Leonenko~{\it etal}'s and Goria~{\it etal}'s estimators $\tilde{\mb{H}}_{S}$ and $\tilde{\mb{I}}_{\alpha,S}$ (MSE $= \Theta(T^{-2/d})$). Experimental MSE comparison of Leonenko's estimator against the BPI estimator in Section V shows the MSE of the BPI estimator to be significantly lower. Finally, note that such bias correction cannot be applied for general entropy functionals, and the bias correction factors cannot in general be incorporated. In the next section, the application of BPI estimators for estimation of Shannon and R\'enyi entropies is illustrated. 

\subsection{Shannon and R\'enyi entropy estimation}

For the case of Shannon entropy ($g(u) = -\log(u)$), it can be verified that $g_1(k,M) = 1$, $g_2(k,M) = \psi(k) - \log(k-1)$ satisfy (\ref{eq:gcond}). Similarly, for the case of R\'enyi entropy ($g(u) = u^{\alpha-1}$), $g_1(k,M) = (\Gamma(k)/\Gamma(k+1-\alpha))(1/(k-1)^{\alpha-1})$, $g_2(k,M)=0$ satisfy (\ref{eq:gcond}). 

For Shannon entropy ($g(u) = -\log(u)$) and R\'enyi entropy ($g(u) = u^{\alpha-1}$), the assumptions in Section \ref{sec:assump} reduce to the following under the condition $k = O((\log(M))^{2/(1-\delta)})$. Assumption $({\cal {A}}.1)$ is unchanged. Assumption $({\cal {A}}.2)$ holds for any $r$ such that $2r>d$. The assumption $({\cal {A}}.3)$ is satisfied by the choice of $\lambda = \log(M)$. Assumption $({\cal {A}}.4)$ holds for ($g(u) = -\log(u)$) and ($g(u) = u^{\alpha-1}$). Next, it will be shown that  $({\cal {A}}.5)$ is also satisfied by ($g(u) = -\log(u)$) and ($g(u) = u^{\alpha-1}$). 

We note that $\tilde{g} = (g^{(3)}/6)^2$ for the choice of $g(u) = -\log(u)$ is given by $\tilde{g} = cu^{-6}$ for some constant $c$. Therefore, 
\begin{eqnarray}
\sup_{x \in (q_l,q_u) } |\tilde{g}(x,y)|e^{-3k^{(1-\delta)}} &=& |c \epsilon^{-6}_1| (M/k)^{6} O(e^{-3k^{(1-\delta)}}) \nonumber \\
&=& |c \epsilon^{-6}_1| (M/k)^{6} O(e^{-3(\log(M))^{2}}) \nonumber \\
&=& |c \epsilon^{-6}_1| O(e^{-3(\log(M))^{2} + 6\log(M) - 6\log(k)}) = o(1), \nonumber
\end{eqnarray}
and by (\ref{eq:existence1}), $\expect[\sup_{x \in (p_l,p_u) } |\tilde{g}(x/\mb{p},y)|] = |c|((1-\epsilon)\epsilon_0)^{-6}\expect[(M\mb{p}/(k-1))^6] = |c|((1-\epsilon)\epsilon_0)^{-6}O(1) = O(1).$ Similarly, $\tilde{g} = (g^{(\lambda)}/(\lambda!))^2$ for the choice of $g(u) = -\log(u)$ is given by $\tilde{g} = \lambda^{-2}u^{-2\lambda}$. Then, 
\begin{eqnarray}
\sup_{x \in (q_l,q_u) } |\tilde{g}(x,y)|e^{-3k^{(1-\delta)}} &=&  O((M/k)^{2\lambda} e^{-3k^{(1-\delta)}}) \nonumber \\
&=& O((M/k)^{2\lambda} e^{-3(\log(M))^{2}}) \nonumber \\
&=& O(e^{-3(\log(M))^{2} + 2(\log(M))^2 - 2\log(M)\log(k)}) = o(1), \nonumber
\end{eqnarray}
and by (\ref{eq:existence1}), $\expect[\sup_{x \in (p_l,p_u) } |\tilde{g}(x/\mb{p},y)|] = O(\expect[(M\mb{p}/(k-1))^{2\lambda})] = O(1)$. In an identical manner, $({\cal {A}}.5)$ is satisfied when $g(u) = u^{\alpha-1}$. 

To summarize, for functions $g(u) = -\log(u)$ and $g(u) = u^{\alpha-1}$, Theorem~\ref{knnbiasHRS}, ~\ref{knnvarHRS} and ~\ref{knncltHRS} hold under the following assumptions: (i) $({\cal A}.0)$, (ii) $({\cal {A}}.1)$, (iii) the density $f$ has bounded continuous partial derivatives of order greater than $d$ and (iv) $k = O((\log(M))^{2/(1-\delta)})$. Furthermore the proposed BPI estimator $\hat{\mb{G}}_{N,BC}(\mb{\tilde{f}}_k)$ can be used to estimate Shannon entropy ($g(u) = -\log(u)$) and R\'enyi entropy ($g(u) = u^{\alpha-1}$) at MSE rate of $\Theta(((\log(T))^{2/(1-\delta)}/T)^{4/d})$. 

\section{Estimation of Shannon Mutual information}

The joint entropy of random vectors $\mb{X}$ and $\mb{Y}$ with joint density $f_{XY}$ is given by
\begin{equation}
H(\mb{X},\mb{Y}) = -\int f_{XY} \log(f_{XY}) d\mu,
\end{equation}
where $f_{XY}$ is the joint density of $\mb{X}$ and $\mb{Y}$. The Shannon MI between two random vectors $\mb{X}$ and $\mb{Y}$ is then given by
\begin{equation}
    I(\mb{X};\mb{Y}) = H(\mb{X}) + H(\mb{Y}) - H(\mb{X},\mb{Y}).
\end{equation}

We use the following BPI estimator to estimate Shannon MI from $N+M$ $d$-dimensional i.i.d samples $\{(\mb{X_i},\mb{Y_i}); i=1,\ldots,N+M\}$ of the underlying joint density $ f_{XY}$. We estimate the Shannon MI by estimating the individual entropies. We estimate the joint Shannon entropy $H(\mb{X},\mb{Y})$ from samples using the \emph{plug-in} estimate
\begin{equation}
\label{plug}
\mb{\hat{H}(X,Y)} = \frac{1}{N}\sum_{i=1}^N -\log(\mb{{\tilde{f}{_{k}}(X_i,Y_i)}}) + \log(k-1) -\psi(k),
\end{equation}
where $\mb{\hat{f}{_{XY}}}$ is a $k$ nearest neighbor density estimate ($k$NN) estimated using the remaining $M$ samples. 

The $k$NN density estimate \cite{quu} is given by
\begin{equation}
  \mb{\tilde{f}_{k}}(X,Y) = \frac{k-1}{M\mb{\mb{V_k}}(X,Y)},
\end{equation}
where $\mb{V_k}(X,Y)$ is the volume corresponding to the $k$th nearest neighbor distance between the point of density estimation $(X,Y)$ and the $M$ i.i.d samples $\{(\mb{X_i},\mb{Y_i}); i=N+1,\ldots,N+M\}$.  

We estimate the marginal entropies by first obtaining estimates of the marginal density using $k$NN density estimates 
\begin{equation}
\mb{\tilde{f}{_{k}}}(X) = \frac{k-1}{M\mb{\mb{V_k}}(X)},
\end{equation}
where $\mb{V_k}(X)$ is the volume corresponding to the $k$th nearest neighbor distance between the point of density estimation $X$ and the $M$ i.i.d samples $\{\mb{X_i}; i=N+1,\ldots,N+M\}$, and then plugging the estimated marginals into Eq. \ref{plug2}.
\begin{equation}
\label{plug2}
\mb{\hat{H}(X)} = \frac{1}{N}\sum_{i=1}^N -\log(\mb{{\tilde{f}{_{k}}(X_i)}}) + \log(k-1) - \psi(k).
\end{equation}

Define the BPI estimator of Shannon MI:
\begin{equation}
\mb{\hat{I}}_N = \mb{\hat{H}(X)} + \mb{\hat{H}(Y)} - \mb{\hat{H}(X,Y)}.
\end{equation}

We make the following assumptions: (i) $({\cal A}.0)$, (ii) $({\cal {A}}.1)$, (iii) the density $f_{XY}$ has bounded continuous partial derivatives of order greater than $d$ and (iv) $k = O((\log(M))^{2/(1-\delta)})$. Note that the results here require  cross moments between density estimates of the joint and marginal densities, which while not discussed in this report, can be obtained in exactly the same manner as computing cross moments between the same density.  

\begin{theorem}
\label{knnbiasHMI}
The bias of the BPI estimator $\mb{\hat{I}}_N$ is given by
\begin{eqnarray}
\label{BiasMI}
\mathbb{B}[\mb{\hat{I}}_N] &=& c_{1}\left({\frac{k}{M}}\right)^{2/d} + c_{3}(k,M,N) + o\left(\left(\frac{k}{M}\right)^{2/d}\right).
\end{eqnarray}
\end{theorem}
\begin{theorem}
\label{knnvarHMI}
The variance of the BPI estimator $\mb{\hat{I}}_N$ is given by
\begin{eqnarray}
\label{VarianceMI}
\var[\mb{\hat{I}}_N] &=& c_4\left(\frac{1}{N}\right)+ c_5\left(\frac{1}{M}\right) + o\left(\frac{1}{M} + \frac{1}{N}\right), \nonumber
\end{eqnarray}
where 
\begin{eqnarray}
c_v &=& Var\left[\log\left(\frac{f_X(\mb{X})f_Y(\mb{Y})}{f_{XY}(\mb{X},\mb{Y})}\right)\right]. \nonumber
\end{eqnarray}
\end{theorem}
\subsubsection{CLT}
\begin{theorem}
\label{knncltHMI}
The asymptotic distribution of the BPI estimator $\mb{\hat{I}}_N$ is given by
\begin{eqnarray}   
\label{CLTMI}
\lim_{\Delta \to 0} Pr\left(\frac{\mb{\hat{I}}_N-\expect[\mb{\hat{I}}_N]}{{\sqrt{\var[\mb{\hat{I}}_N]}}} \leq \alpha \right) = Pr(\mb{S} \leq \alpha), \nonumber
\end{eqnarray}
where $\mb{S}$ is a standard normal random variable.
\end{theorem}

\section{Simulations}
\label{sec:exp}
\begin{figure}[!t]
\centering
\includegraphics[scale=.3]{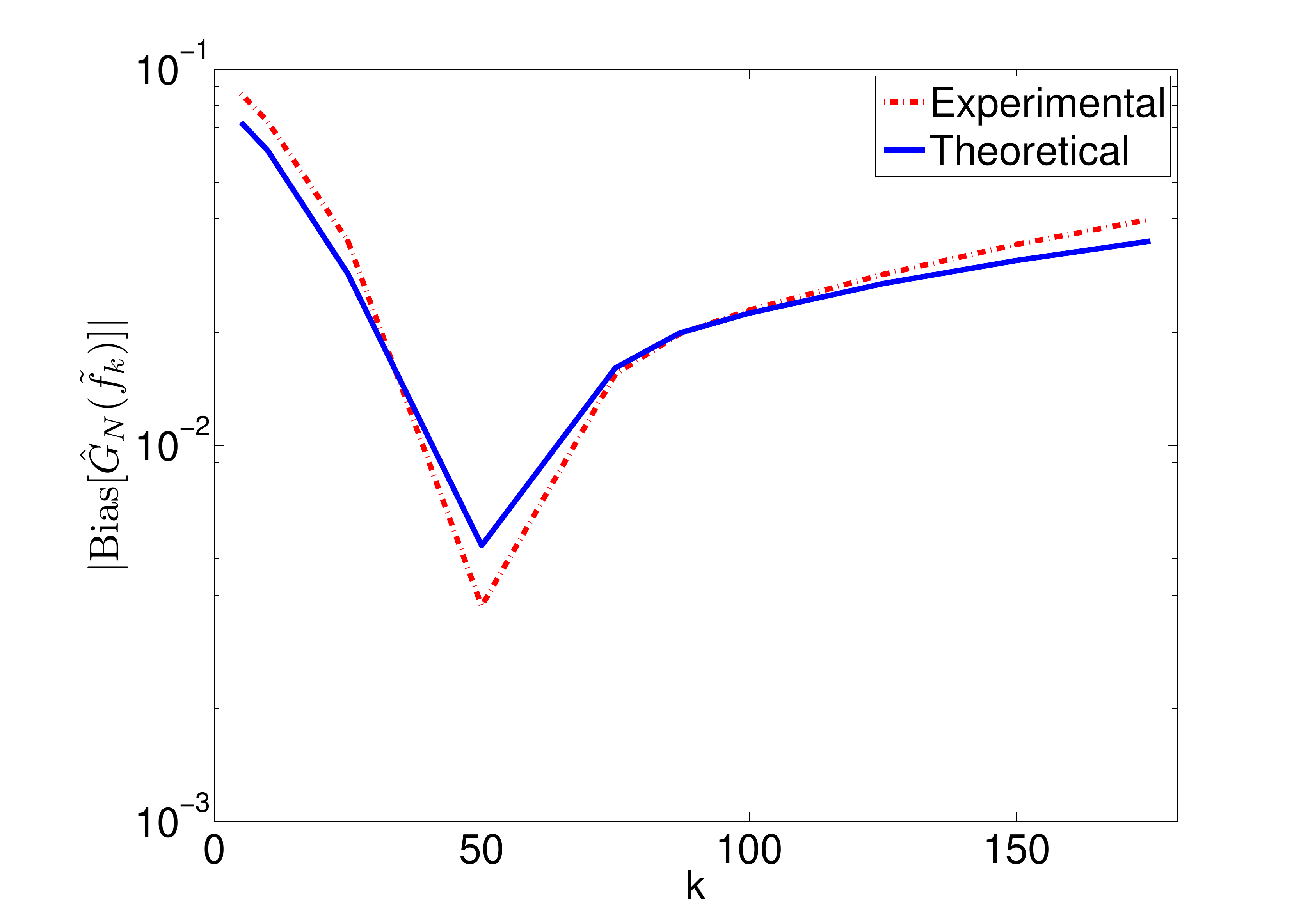}
\caption{Comparison of theoretically predicted bias of BPI estimator $\hat{\mb{G}}_N(\mb{\tilde{f}}_k)$ against experimentally observed bias as a function of $k$. The Shannon entropy ($g(u) = -\log(u)$) is estimated using the BPI estimator $\hat{\mb{G}}_N(\tilde{\mb{f}}_k)$ on $T=10^4$ i.~i.~d.~ samples drawn from the $d=3$ dimensional uniform-beta mixture density (\ref{mixturesimul}). $N,M$ were fixed as $N=3000$, $M=7000$ respectively. The theoretically predicted bias agrees well with experimental observations. {The predictions of our asymptotic theory  therefore extend to the finite sample regime}. The theoretically predicted optimal choice of $k_{opt} = 52$ also minimizes the empirical bias. }
\label{S4}
\end{figure}
\begin{figure}[!t]
\centering
\includegraphics[scale=.3]{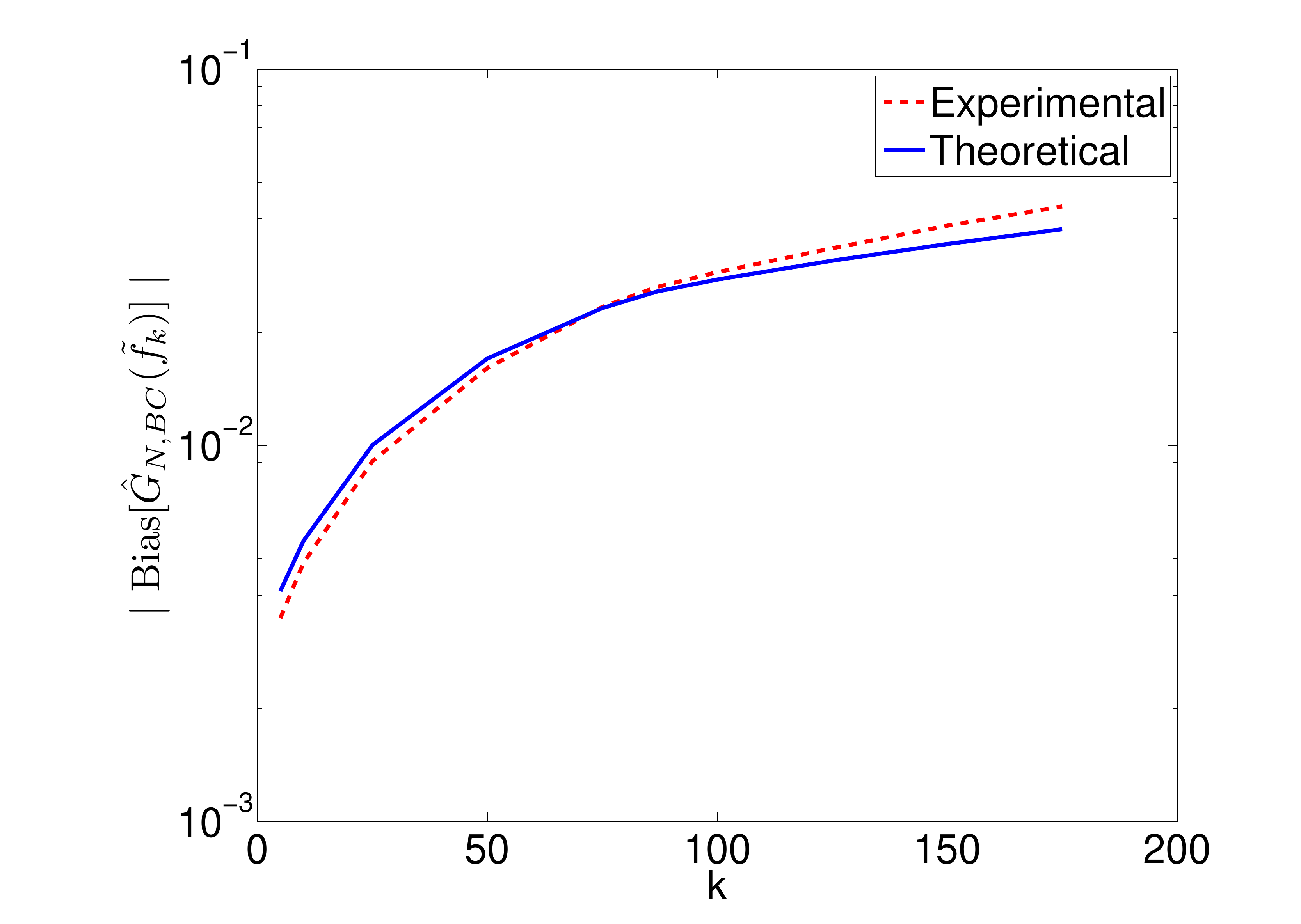}
\caption{Comparison of theoretically predicted bias of BPI estimator $\hat{\mb{G}}_{N,BC}(\mb{\tilde{f}}_k)$ against experimentally observed bias as a function of $k$. The Shannon entropy ($g(u) = -\log(u)$) is estimated using the proposed BPI estimator $\hat{\mb{G}}_{N,BC}(\tilde{\mb{f}}_k)$ on $T=10^4$ i.~i.~d.~ samples drawn from the $d=3$ dimensional uniform-beta mixture density (\ref{mixturesimul}). $N,M$ were fixed as $N=3000$, $M=7000$ respectively. The empirical bias is in agreement with the bias approximations of Theorem IV.~1 and monotonically increases with $k$.}
\label{S44}
\end{figure}

\begin{figure}[!t]
\centering
\includegraphics[scale=.3]{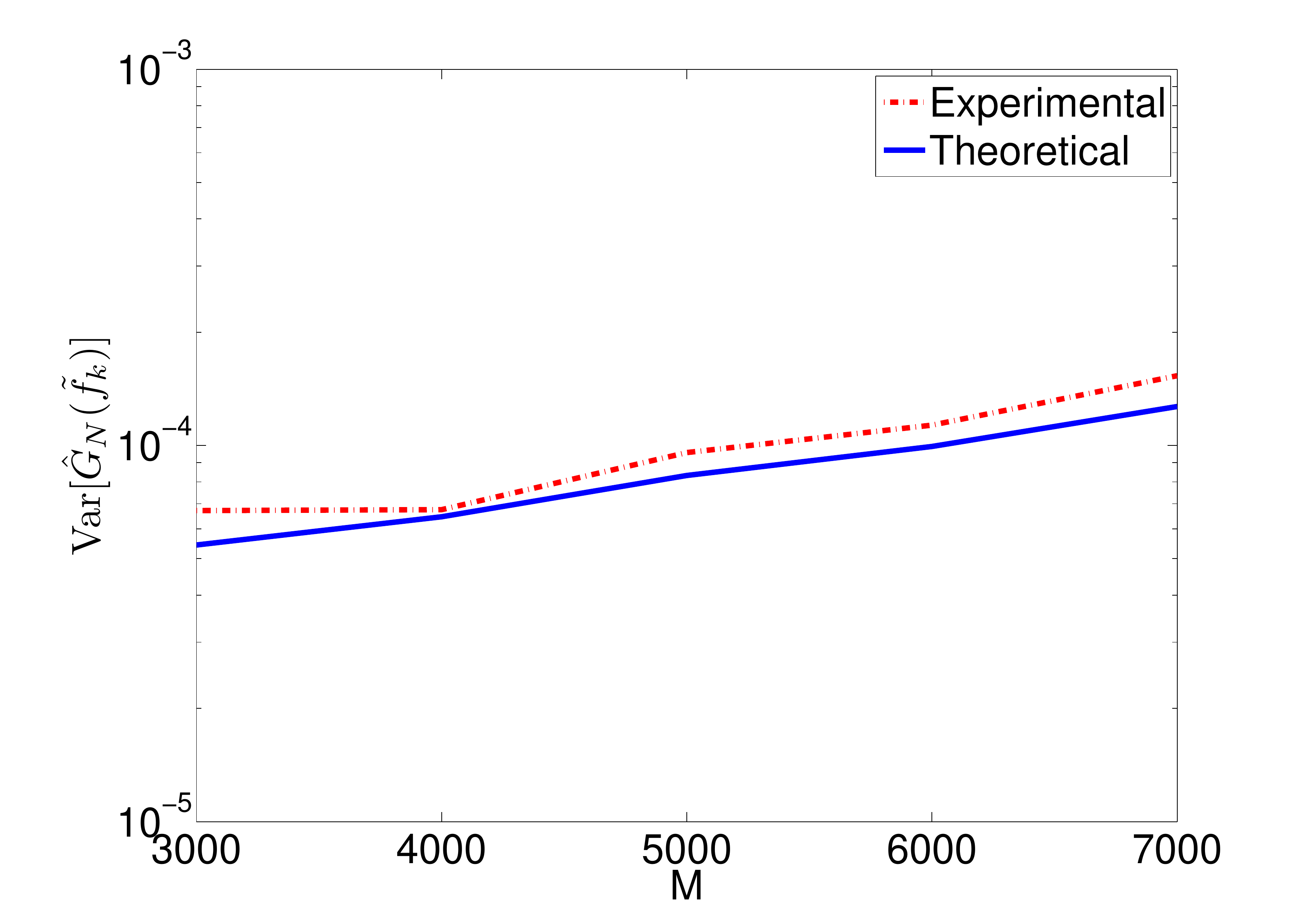}
\caption{Comparison of theoretically predicted variance of BPI estimator $\hat{\mb{G}}_N(\mb{\tilde{f}}_k)$ against experimentally observed variance as a function of $M$. The Shannon entropy ($g(u) = -\log(u)$) is estimated using the proposed BPI estimator $\hat{\mb{G}}_N(\tilde{\mb{f}}_k)$ on $T=10^4$ i.~i.~d.~ samples drawn from the $d=3$ dimensional uniform-beta mixture density (\ref{mixturesimul}). $k$ is chosen to be $k_{opt} = k_0M^{2/(2+d)}$. The theoretically predicted variance agrees well with experimental observations.}
\label{S5}
\end{figure}
Here the theory established in Section 3 and Section 4 is validated. A three dimensional vector $\underline{X}=[X_1, X_2, X_3]^T$ was generated on the unit cube according to the i.i.d. Beta plus i.i.d. uniform mixture model: \begin{equation} f(x_1,x_2,x_3)= (1-\epsilon) \prod_{i=1}^3 f_{a,b}(x_i) + \epsilon, \label{mixturesimul} \end{equation} where $f_{a,b}(x)$ is a univariate Beta density with shape parameters $a$ and $b$. For the experiments the parameters were set to $a=4,b=4$, and $\epsilon=0.2$. {The Shannon entropy ($g(u) = -\log(u)$) is estimated using the BPI estimators  $\hat{\mb{G}}_N(\tilde{\mb{f}}_k)$ and $\hat{\mb{G}}_{N,BC}(\tilde{\mb{f}}_k)$}. 



In Fig.~\ref{S4}, the bias approximations of Theorem III.~1 are compared to the empirically determined estimator bias of $\hat{\mb{G}}_N(\tilde{\mb{f}}_k)$. $N$ and $M$ are fixed as $N=3000$, $M=7000$. Note that the theoretically predicted optimal choice of $k_{opt} = 52$ minimizes the experimentally obtained bias curve. Thus, even though our theory is asymptotic it provides useful predictions for the case of finite sample size, specifying bandwidth parameters that achieve minimum bias. {Further note that by matching rates, i.e. choosing $k=\bar{k}= M^{2/(2+d)} = 83$ also results in significantly lower MSE when compared to choosing $k$ arbitrarily ($k<10$ or $k>150$)}. {In Fig.~\ref{S44}, the bias approximations of Theorem IV.~1 are compared to the empirically determined estimator bias of $\hat{\mb{G}}_{N,BC}(\tilde{\mb{f}}_k)$. Observe that the empirical bias, in agreement with the bias approximations of Theorem IV.~1, monotonically increases with $k$.}

In Fig.~\ref{S5}, the empirically determined variance of $\hat{\mb{G}}_N(\tilde{\mb{f}}_k)$ is compared with the variance expressed by Theorem III.~2 for varying choices of $N$ and $M$, with fixed $N+M=10,000$.  The theoretically predicted variance agrees well with experimental observations.
\begin{figure}[!t]
\centering
\includegraphics[scale=.3]{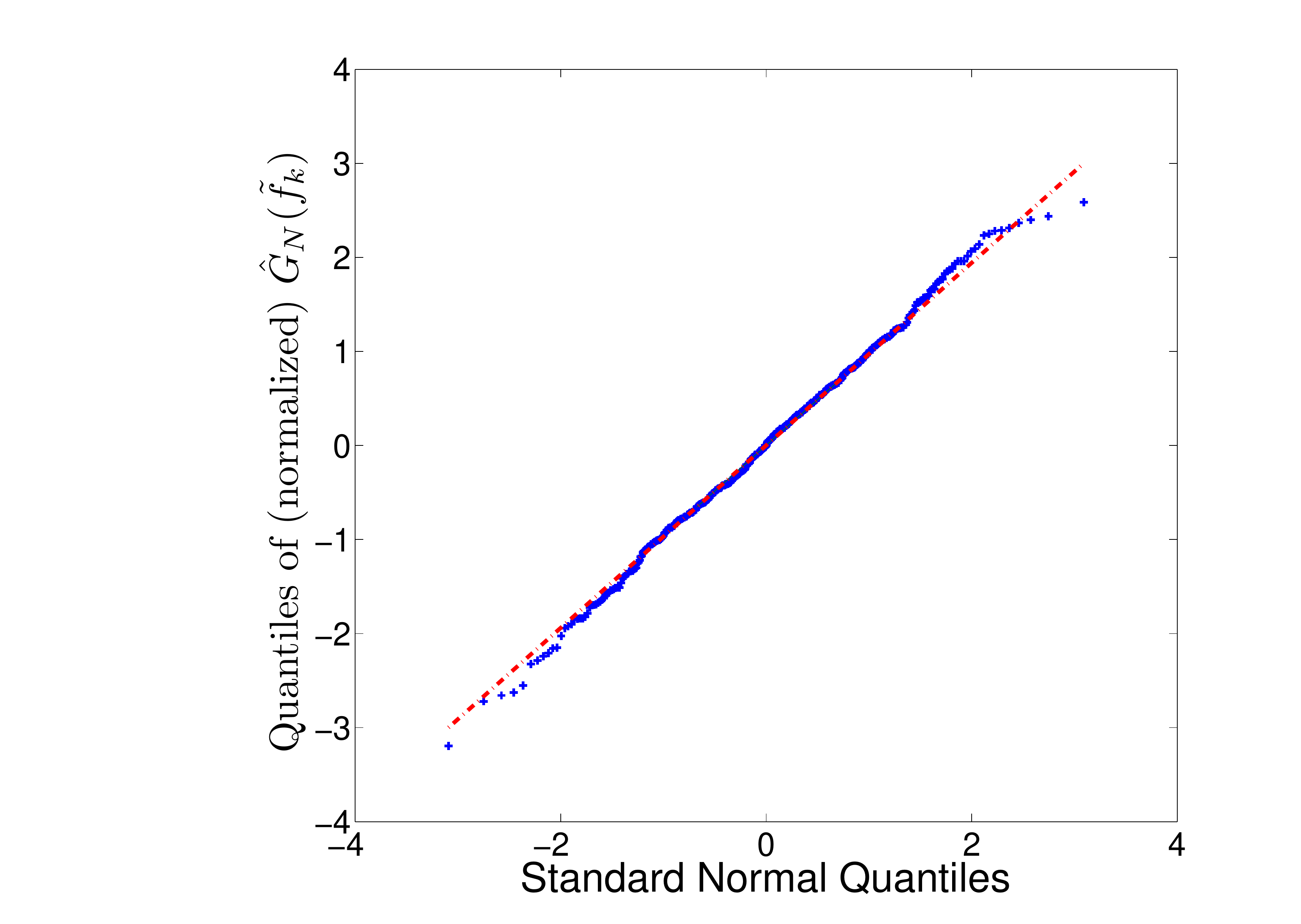}
\caption{Q-Q plot comparing the quantiles of the BPI estimator $\hat{\mb{G}}_N(\mb{\tilde{f}}_k)$ (with $g(u) = -\log(u)$) on the vertical axis to a standard normal population on the horizontal axis. The Shannon entropy ($g(u) = -\log(u)$) is estimated using the proposed BPI estimator $\hat{\mb{G}}_N(\tilde{\mb{f}}_k)$ on $T=10^4$ i.~i.~d.~ samples drawn from the $d=3$ dimensional uniform-beta mixture density (\ref{mixturesimul}). $k,N,M$ are fixed as $k=k_{opt}=52$, $N=3000$ and $M=7000$ respectively. The approximate linearity of the points validates our central limit theorem~\ref{knncltH}.}
\label{S6}
\end{figure}
A Q-Q plot of the normalized BPI estimate $\hat{\mb{G}}_N(\tilde{\mb{f}}_k)$ and the standard normal distribution is shown in Fig.~\ref{S6}. The linear Q-Q plot validates the Central Limit Theorem III.~3 on the uncompensated BPI estimator. 
\begin{figure}[!t]
\centering
\includegraphics[scale=.3]{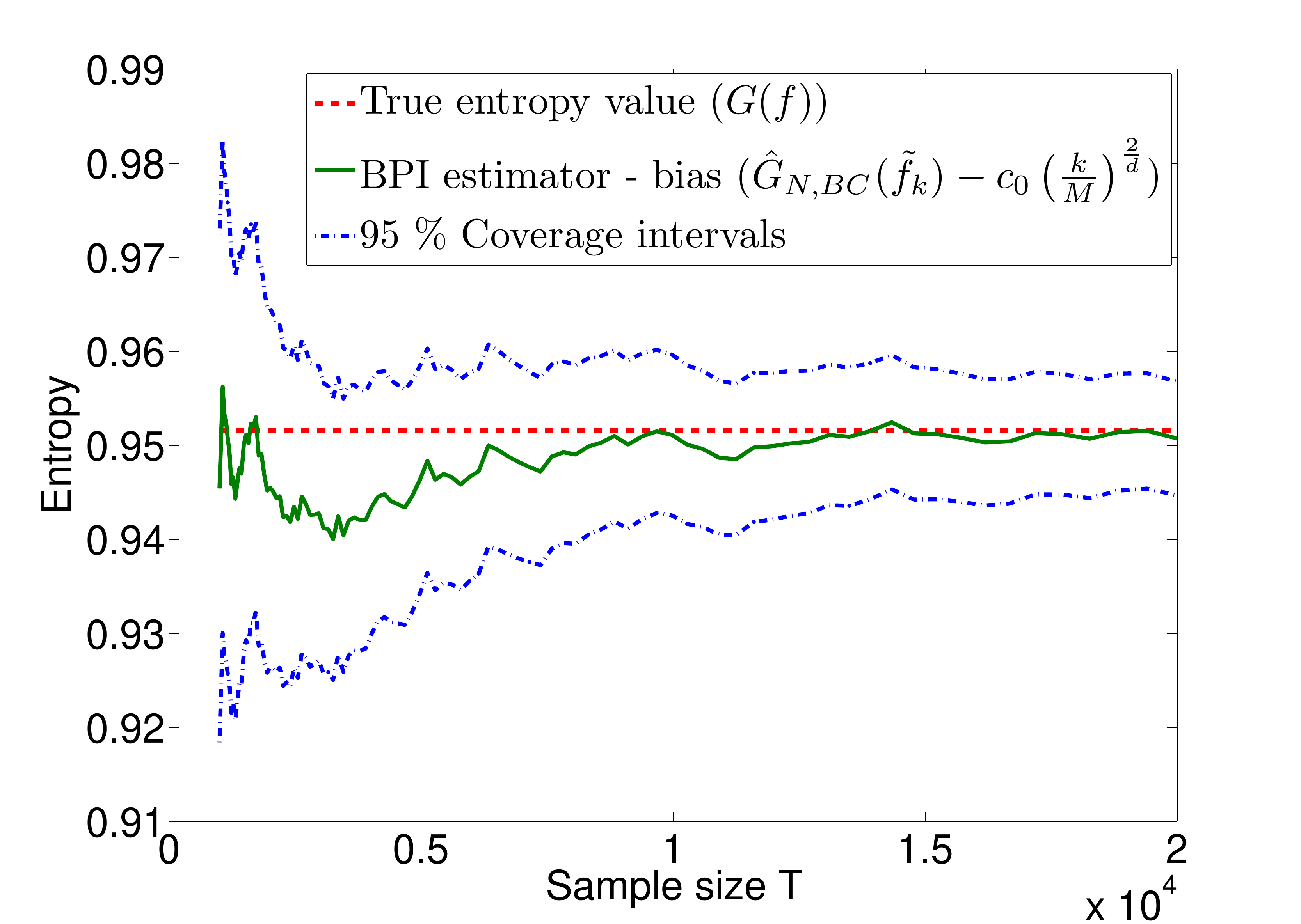}
\caption{{$95\%$ coverage intervals of BPI estimator $\hat{\mb{G}}_{N,BC}(\tilde{\mb{f}}_k)$, predicted using the Central limit theorem~\ref{knncltH}, as a function of sample size $T$. The Shannon entropy ($g(u) = -\log(u)$) is estimated using the proposed BPI estimator $\hat{\mb{G}}_{N,BC}(\tilde{\mb{f}}_k)$ on $T$ i.~i.~d.~ samples drawn from the $d=3$ dimensional uniform-beta mixture density (\ref{mixturesimul}). The lengths of the coverage intervals are accurate to within $12\%$ of the empirical confidence intervals obtained from the empirical distribution of the BPI estimator}.}
\label{S8}
\end{figure}
To verify that the predicted confidence intervals were indeed as advertised, the empirically determined and theoretically predicted confidence intervals were compared in Fig.~\ref{S9A}.  The lengths of the predicted confidence intervals are accurate to within $12\%$ of the length of the true confidence intervals. 
\begin{figure}[!t]
\centering
\includegraphics[scale=.3]{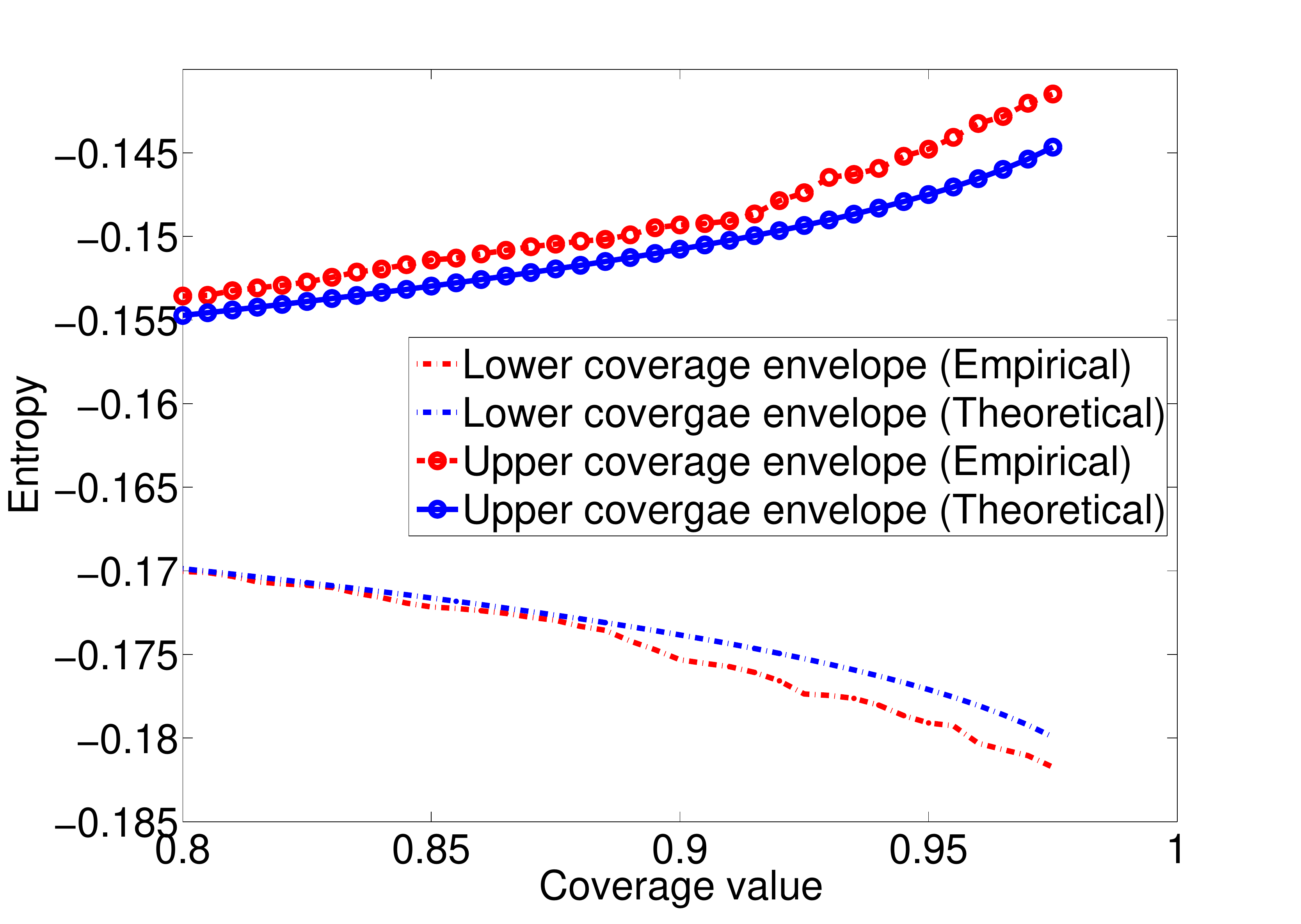}
\caption{Empirically determined and theoretically predicted coverage envelopes as a function of coverage values. There is good agreement between the theoretically predicted and empirical coverage intervals.}
\label{S9A}
\end{figure}

We additionally show in Fig.~\ref{S8A} a plot of the empirically determined  estimator bias (via simulation) vs the bias predicted by our theory as a function of sample size $T$, which matches the theoretical prediction.

\begin{figure}[!t]
\centering
\includegraphics[scale=.3]{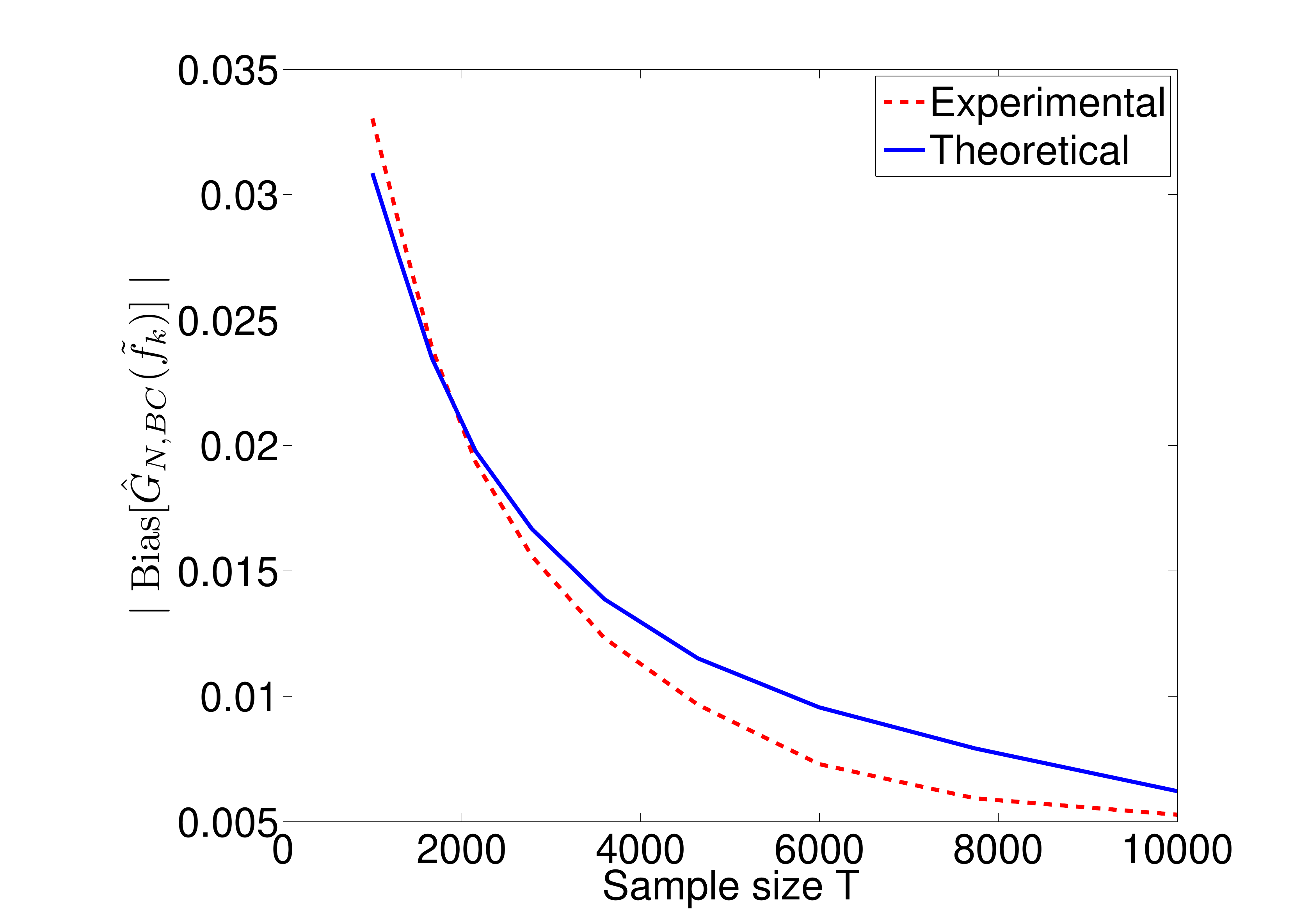}
\caption{Comparison of theoretically predicted bias of BPI estimator $\hat{\mb{G}}_{N,BC}(\mb{\tilde{f}}_k)$ against experimentally observed bias as a function of sample size $T$. The Shannon entropy ($g(u) = -\log(u)$) is estimated using the proposed BPI estimator $\hat{\mb{G}}_{N,BC}(\tilde{\mb{f}}_k)$ on $T$ i.~i.~d.~ samples drawn from the $d=3$ dimensional uniform-beta mixture density (\ref{mixturesimul}). The empirical bias is in agreement with the bias approximations of Theorem IV.~1 and monotonically decreases with $T$.}
\label{S8A}
\end{figure}

\begin{figure}[tb]
\centering
\includegraphics[width=4in]{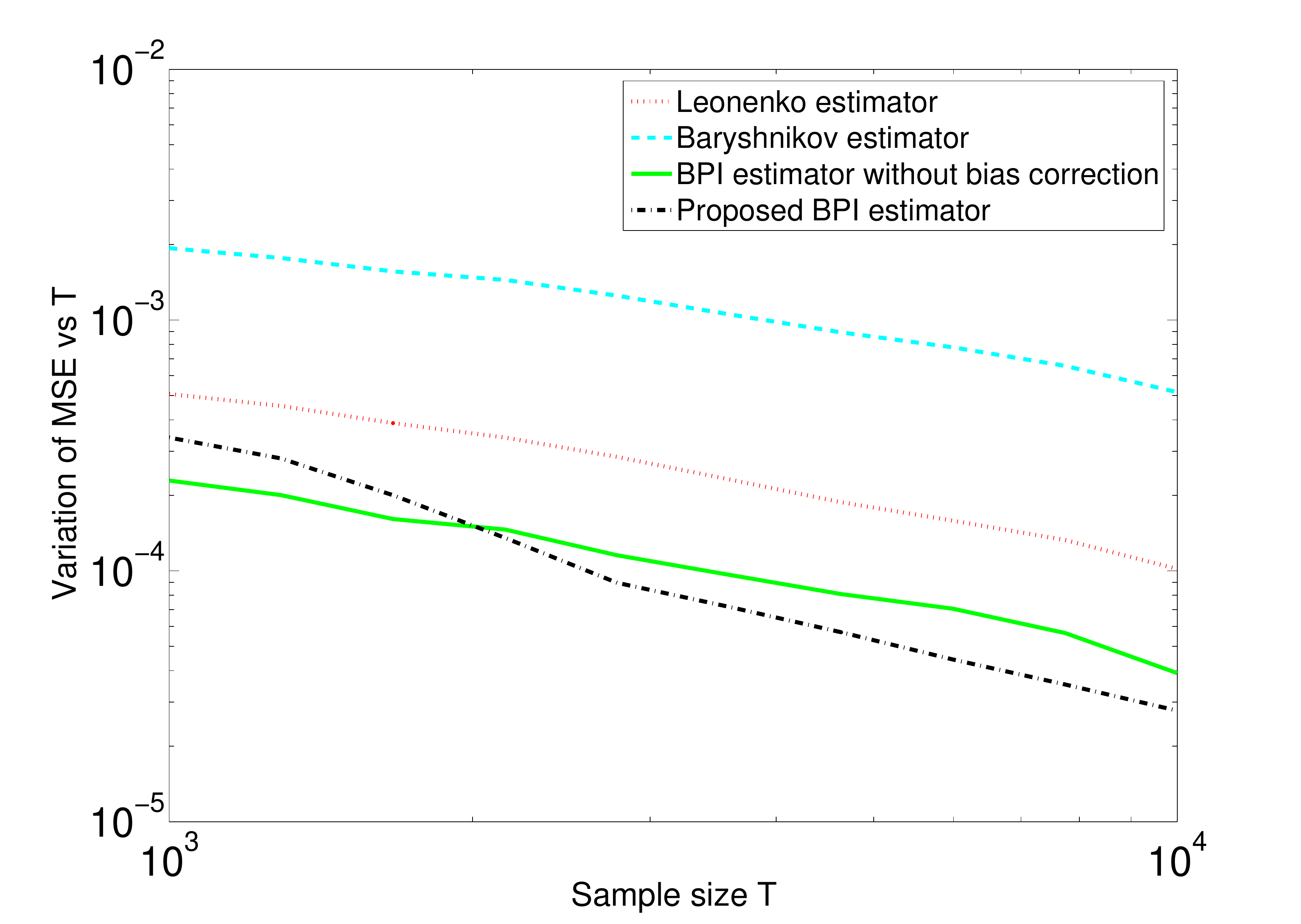}
\caption{Variation of MSE of $k$-nearest neighbor estimator of Leonenko~{\it etal}~\cite{leo} and the $k$-nearest neighbor estimator of Baryshnikov~{\it etal}~\cite{bar} and BPI estimators with and without boundary correction, as a function of sample size $T$. The R\'enyi entropy ($g(u) = u^{\alpha-1}$) is estimated for $\alpha=0.5$ using these estimators on $T$ i.~i.~d.~ samples drawn from the $d=3$ dimensional uniform-beta mixture density (\ref{mixturesimul}). The figure shows that the proposed BPI estimator has the fastest rate of convergence.}
\label{a-compare}
\end{figure} 
{For Shannon entropy ($g(u) = -\log(u)$), the uncompensated and compensated BPI estimators are related by$$\hat{\mb{G}}_{N,BC}(\mb{\tilde{f}}_k) = \hat{\mb{G}}_{N}(\mb{\tilde{f}}_k) + \log(k-1) - \psi(k). $$} {The variance and normalized distribution of these estimators are therefore identical.  Consequently, Fig.~\ref{S5} and Fig.~\ref{S6} also validate Theorem IV.~2 and Theorem IV.~3 respectively.}

{Finally, using the CLT, the $95\%$ coverage intervals of the BPI estimator $\hat{\mb{G}}_{N,BC}(\tilde{\mb{f}}_k)$ are shown as a function of sample size $T$ in Fig.~\ref{S8}. The lengths of the predicted confidence intervals are accurate to within $12\%$ of the true confidence intervals (determined by simulation over the range of $80\%$ to $100\%$ coverage - data not shown). These coverage intervals can be interpreted as confidence intervals on the true entropy, provided that the constants $c_1,..,c_5$ can be accurately estimated. }


\subsection{Experimental comparison of estimators}
The R\'enyi $\alpha$-entropy ($g(u) = u^{\alpha-1}$) is estimated for $\alpha=0.5$, with the same underlying 3 dimensional mixture of the beta and uniform densities defined above. Several estimators are compared: Baryshnikov's estimator $ \hat{\mb{I}}_{\alpha,S}$, the $k$-NN estimator $ \tilde{\mb{I}}_{\alpha,S}$ of Leonenko~{\it etal}~\cite{leo}, the BPI estimator without bias correction $\hat{\mb{G}}_N(\mb{\tilde{f}}_k)$ and the proposed BPI estimator with bias correction $\hat{\mb{G}}_{N,BC}(\mb{\tilde{f}}_k)$. The results are shown in Fig.~\ref{a-compare}. It is clear from the figure that the BPI estimator $\hat{\mb{G}}_{N,BC}(\mb{\tilde{f}}_k)$ has the fastest rate of convergence, consistent with our theory. Note that, in agreement with our analysis in Section~\ref{sec:compareb}, the bias uncompensated BPI estimator $\hat{\mb{G}}_{N}(\mb{\tilde{f}}_k)$ outperforms Baryshnikov's estimator $ \hat{\mb{I}}_{\alpha,S}$.

\section{Application to structure discovery}

Discovering structural dependencies among random variables from a multivariate sample is an important task in signal processing, pattern recognition and machine learning. Based on dependence relationships, the density function of the variables can be modeled using factor graphs. When the sample is highly structured, the corresponding factor graph configuration is sparse. Sparse factor graphs correspond to joint multivariate distributions which separate into a parsimonious product of few lower dimensional distributions. The inherent low-dimensional nature of this product leads to a compact representation of the variables having sparse factor graph configurations.

In practice, these structure dependencies have to be discovered from sample realizations of the multivariate distribution. Discovering dependencies when parametric probability density models are not known a priori is an important restriction of the above problem. For parametric distribution estimates, the errors are of order $O(1/N)$ if the true distribution is included in the parametric model. If not, a non-vanishing bias will dominate the error yielding an even higher error than that of a nonparametric distribution estimate (e.g. $k$NN estimates). In this restricted setting, recourse is therefore taken to nonparametric methods. 

Chow et.al. \cite{chow} proposed an elegant solution to structure discovery of Markov tree distributions and provided a nonparametric algorithm to obtain the optimal tree. Ihler et.al. \cite{ihler} developed the method of nonparametric hypothesis tests for structure discovery. 

Nonparametric methods, while asymptotically consistent, can uncover incorrect factor graph structure when estimated from a finite number of samples. This is distinctly true for small sample sizes. While consistency is an important qualitative property, there is clearly an important motivation for quantitative characterization of performance in structure discovery. In this work, we analyze factor graph structure discovery in the finite sample size setting.

We present a class of $k$-nearest neighbor ($k$NN) based nonparametric geometric algorithms to discover factor graph structure among variables. We provide results on mean square error of the nonparametric estimates, which can be optimized over free parameters, thereby guaranteeing improved correct structure discovery. In addition, we provide confidence intervals on these nonparametric estimates to determine the probability of false error in choosing an incorrect structure model. These results are an direct extension of our work on optimized nonparametric estimates of divergence measures introduced earlier.

As a consequence of our statistical analysis, we introduce the notion of dependence-based \textbf{dimension} for factor graph models and show that comparing models within the same dimension class is an easier task with lower probability of false error as compared to comparing models across different dimensions.

\subsection{Factor graphs}

Factor graphs are bipartite graphs used to represent factorizations of probability density functions. Consider a set of variables
$\underline{X}=\{X_1,X_2,\ldots,X_T\}$ and let $\{S_j \subseteq \{X_1,X_2,\dots,X_n\}, j=1,\ldots,m\}$ be a set of subsets of $\underline{X}$. Let $g(X_1,\dots,X_T)$ denote a probability density function on the random vector $\underline{X}$. For the factorization $g(X_1,\dots,X_T) = \prod_{j=1}^m f_j(S_j)$ of the density function, the corresponding factor graph $G = (\underline{X},\underline{F},E)$ consists of variable vertice's $\underline{X}$ , factor vertices's $\underline{F}=\{f_1,f_2,\dots,f_m\}$, and edges $E$. The edges in the factor graph depend on the factorization as follows: there is an undirected edge between factor vertex $f_j$ and variable vertex $X_k$ when $X_k \subseteq S_j$. 

\subsection{Factor graph discovery}

\paragraph{Problem statement:} Consider a set of factor graphs $\{g_i(X_1,\dots,X_T),i=1,\dots,I\}$. We seek to find the factor graph configuration from this set that best models the data. 

The Kullback-Leibler (KL) divergence measure induces a \textbf{geometry} on the space of probability distributions. On this induced geometry, we naturally define the best factor graph configuration $g_o$ to be the one closest to the actual distribution $p(X_1,\dots,X_T)$ in terms of KL divergence (c.f. \cite{chow}). 
\begin{equation}
g_{o} = \underset{g_i}{\operatorname{arg\,min}} \, KL(p||g_i) = \underset{g_i}{\operatorname{arg\,min}} \, H_c(p,g_i),
\end{equation}
where $H_c(p,g_i) = -\int p \log g_i$ is the cross-entropy between $p$ and $g_i$. In practice, these cross-entropy terms have to be estimated from the finite data sample. \textbf{Errors in estimation of cross-entropy terms can result in incorrect factor graph discovery}. 

The problem considered by \cite{chow} is a specific instance of discovering factor graph structure. For the class of Markov tree factor graphs considered by \cite{chow}, the cross entropy reduces to a sum of pairwise Shannon mutual information terms between variables with edges in the Markov tree. In their work, they empirically estimate the mutual information terms from the data using nonparametric estimators which are consistent. However, they do not take into account the error in the mutual information estimates when estimated from finite samples. 

\subsection{Disjoint factor graph discovery}
In order to illustrate the effect of nonparametric estimation from finite sample size on factor graph discovery, we restrict our attention to disjoint factor graphs (\cite{ihler}). For $i=1,\ldots,I$, let
\begin{equation}
g_i(X_1,X_2,\dots,X_T) = \prod_{j=1}^m p(S^{(i)}_j),
\end{equation}
where $S^{(i)}_j \cap S^{(i)}_k = \phi$ whenever $j \neq k$, and $p(.)$ denotes the marginal density function. In this case of disjoint factor graphs, the cross-entropy takes the following simple form:
\begin{equation}
\label{crossentropy}
H_c(p,g_i) = \sum_j H(S^{(i)}_j),
\end{equation}
where $H(S^{(i)}_j)$ is the Shannon entropy of the variables $S^{(i)}_j$ under the true distribution $p$. 

For example, consider the disjoint factor graph $g(X_1,\ldots,X_5) = p(X_1,X_2)p(X_3)p(X_4,X_5)$. The cross-entropy for this factor graph is given by $H_c(p,g) = H(X_1,X_2) + H(X_3) + H(X_4,X_5)$.

Consider two disjoint factor graph configurations: (a) $n(X_1,\ldots,X_T) = \prod_{i=1}^{m_1} f(R_i)$ and (b) $l(X_1,\ldots,X_T) = \prod_{j=1}^{m_2} f(S_j)$. Denote the dimension of $R_i$ by $d^{n}_i$ and $S_j$ by $d^{l}_j$. We note that $\sum_{i=1}^{m_1} d^{(n)}_{i} = \sum_{j=1}^{m_2} d^{(l)}_{j} = T$. Based on the above formulation, in order to compare the two potential factor graph models $n$ and $l$, we need to compare the respective cross-entropy terms. The cross entropy test is stated below.
\paragraph{Cross entropy test:} The cross entropy test to compare between models $n$ and $l$ is given by 
\begin{equation}
{{H_c}(p,n)-{H_c}(p,l) = \sum_{i=1}^{m_1}{H}(R_i) - \sum_{j=1}^{m_2}{H}(S_j) \gleq 0}.
\end{equation} 
We {estimate} these entropy terms in the test statistic ${H_c}(p,n)-{H_c}(p,l)$ from sample realizations using $k$NN plug-in estimators introduced earlier. 

\subsection{Errors in factor graph discovery}

To illustrate the effect of estimation error in factor graph discovery, again consider the two factor graph models $n(X_1,\ldots,X_T) = \prod_{i=1}^{m_1} f(R_i)$ and $l(X_1,\ldots,X_T) = \prod_{j=1}^{m_2} f(S_j)$. 

The cross entropy test (Eq. \ref{crossentropy}) between models $n$ and $l$ is $H_c(p,n)-H_c(p,l) \gleq 0$. We replace this optimal cross entropy test with the following \textbf{surrogate} cross entropy test:

\begin{equation}
\hat{H_c}(p,n)-\hat{H_c}(p,l) = \sum_{i=1}^{m_1}\hat{H}(R_i) - \sum_{j=1}^{m_2}\hat{H}(S_j) \gleq 0.
\end{equation} 
where we estimate entropy terms $\hat{H}(R_i)$ or $\hat{H}(S_j)$ using independent realizations of the underlying density $p$. To elaborate, if we have $V$ samples $\{\underline{X}^{(1)}, \ldots,\underline{X}^{(V)}\}$ from the density $p$, we partition these $V$ samples into $m_1+m_2$ disjoint subsets of size $N+M$ each. This implies that $N+M \approx V/(m_1+m_2)$. We then use each subset to estimate entropy using the partitioning strategy as discussed earlier. 

Denote the coefficients corresponding to the entropy estimate $\hat{H}(R_i)$ of the subset of variables $R_i$ in the factor graph model $n$  by $c_{n_i1}$, $c_{n_i2}$ and $c_{n_i4}$. Using the theorems established in this report, we have the following results:

\textbf{Mean:} The mean of this surrogate test statistic is then given by
\begin{eqnarray}
{\mathbb{E}_p[\hat{H_c}(p,n)-\hat{H_c}(p,l)]} &{=}& {H_c(p,n)-H_c(p,l)} \nonumber \\
&{+}& {\sum_{i=1}^{m_1} c_{n_i1} \left({\frac{k}{M}}\right)^{2/d^{(n)}_{i}} - \sum_{j=1}^{m_2} c_{l_j1} \left({\frac{k}{M}}\right)^{2/d^{(l)}_j}} \nonumber \\
&{+}& {\sum_{i=1}^{m_1} c_{n_i2}/k - \sum_{j=1}^{m_2} c_{l_j2}/k}.
\end{eqnarray} 
\textbf{Variance:} The variance of the surrogate test statistic is then given by the sum of the variance of the individual entropy estimates (by independence)
\begin{eqnarray}
{\var_p[\hat{H_c}(p,n)-\hat{H_c}(p,l)]} &{=}& \left(\sum_{i=1}^{m_1} c_{n_i4} + \sum_{j=1}^{m_2} c_{l_j4}\right)\left(\frac{1}{N}\right).
\end{eqnarray} 
\textbf{Weak convergence:} Again, by independence of the individual entropy estimates, we have the following weak convergence law
\begin{equation}
{\lim_{N,M \to \infty} Pr\left(\frac{\sqrt{N}(\hat{H_c}(p,n)-\hat{H_c}(p,l)-\mathbb{E}_p[\hat{H_c}(p,n)-\hat{H_c}(p,l)])}{\sqrt{{\var_p[\hat{H_c}(p,n)-\hat{H_c}(p,l)]}}} \leq \alpha \right) = Pr\left({Z} \leq \alpha\right)}, \label{theeq6476}
 \end{equation}
where $Z$ is standard normal.

\subsection{Discussion}

From the above expressions for the mean, variance and weak convergence law of the surrogate test statistic, we make the following observations:

\begin{enumerate}
\item
The bias term is dependent on the dimension of the factors of the factor graph models $d^{(n)}_i$ and $d^{(l)}_j$. The variance term is independent of dimension. Furthermore, it is clear that the bias term dominates the MSE as the dimension of the factors grows.
\item
For better performance in discovering factor graph structure using cross entropy tests, it is clear that we want the MSE of the surrogate test statistic to be small. A significant route to achieving this is to get the bias from each factor graph cross entropy estimate in the estimated test statistic to cancel. This is to say, we want 
\begin{eqnarray}
{\mathbb{E}_p[\hat{H_c}(p,n)-\hat{H_c}(p,l)]} &\approx& {H_c(p,n)-H_c(p,l)} \nonumber \\
\Rightarrow {\mathbb{E}_p[\hat{H_c}(p,n)]-\hat{H_c}(p,n)} &\approx& {\mathbb{E}_p[\hat{H_c}(p,l)]-\hat{H_c}(p,l)} \nonumber \\
\Rightarrow \sum_{i=1}^{m_1} c_{n_i1} \left({\frac{k}{M}}\right)^{2/d^{(n)}_{i}} + \sum_{i=1}^{m_1} c_{n_i2}/k &\approx& \sum_{j=1}^{m_2} c_{l_j1} \left({\frac{k}{M}}\right)^{2/d^{(l)}_j} + \sum_{j=1}^{m_2} c_{l_j2}/k.
\end{eqnarray} 
\item
This cancellation effect will be maximized when the dimensions of the factor graph subsets $R_i$ and $S_j$ match. That is to say, we want $m_1=m_2$ and furthermore $d^{(n)}_i = d^{(l)}_j$. In this case, the bias from each cross entropy estimate are of the same order and will nearly cancel. 

On the other hand, when there is a mismatch in dimension, the bias from one cross entropy estimate will dominate the bias from the other cross entropy estimate, resulting in significant bias in the surrogate test statistic.

In both these cases, the variance of the surrogate test statistic will be of the same order $O(1/N)$.

\item
This gives rise to notion of multivariate dimension for factor graphs. Index the factorizations according to the vector ${ E=[e_1,e_2,...,e_p]}$, where $e_i$ is an integer between $0$ and $T$ that counts the number of factors of order $i$, i.e. involving a marginal density over $i$ variables. The \textbf{dimension} $E$ of factor graph configurations partitions the factor graphs into equivalence classes having nearly constant cross entropy estimate bias. 

For two factor graph models $n$ and $l$ with dimensions $E_n$ and $E_l$, we will refer to $n$ as a higher dimensional model  relative to $l$ if the last non-zero entry of $E_n-E_l$ is positive.

\item
As discussed earlier, the bias will not be a significant factor when comparing models over an equivalence class having fixed values of $E$. On the other hand, the bias will be significant when comparing models across different values of $E$, resulting in higher probability of error in factor graph discovery.
\item
Prior knowledge of the equivalence class will therefore translate into much improved performance in factor graph discovery as compared to prior knowledge that mixes between equivalence classes.
\item
We note that the number of samples required to maintain a constant level of bias grows \textbf{geometrically} with dimension $E$.
\item
Using the expressions for the bias and variance of the surrrogate test statistic, we can optimize over the free parameters: (a) the choice of partition $N$ and $M$ for fixed total sample size $N+M$ and (b) the choice of bandwidth parameter $k$, for minimum MSE. 
\item
Using the weak convergence law, we can theoretically predict the probability of choosing model $n$ over model $l$ using the surrogate cross entropy test.
\end{enumerate}



\subsection{Experiment}

We illustrate the implications of our analysis with a toy example. Let $f_\beta(x,a,b,d)$ denote a beta density of dimension $d$ with parameters $a$ and $b$. Now let $f_\mu(x,d) = 0.5f_\beta(x,5,2,d) + 0.5f_\beta(x,2,5,d)$ be a mixture of beta densities. When $d>1$, the mixing of densities ensures there is strong dependence between the variates.

\begin{figure}[!th]

  \begin{center}
    \includegraphics[width=2in]{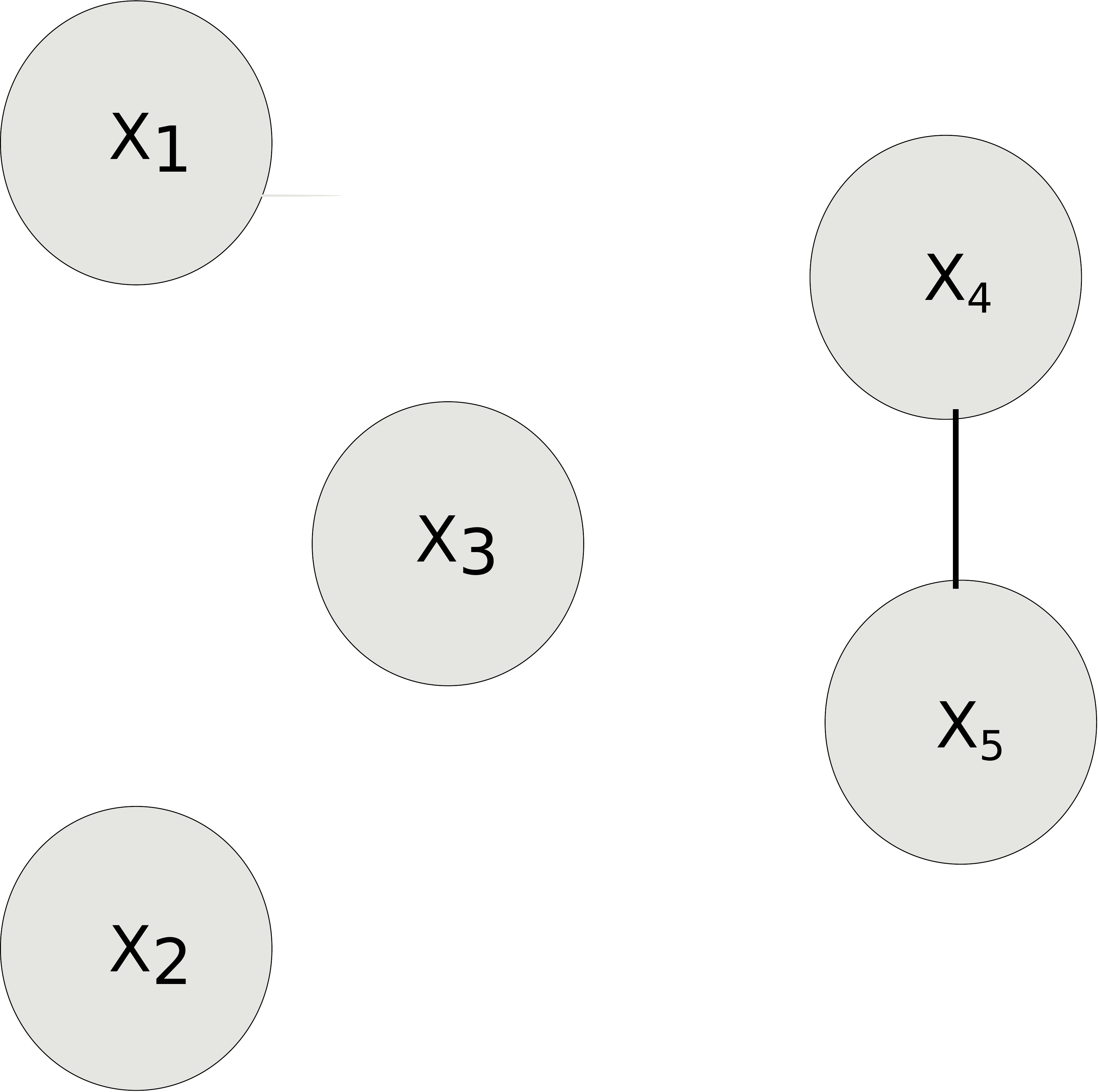}
  \end{center}

  \caption{\small True factor graph representation of the 5-dimensional joint density $p(X_1,\dots,X_5) = f_\mu(X_1,1)f_\mu(X_2,1)f_\mu(X_3,1)f_\mu(X_4,X_5,2)$.}
\label{S10}
  \end{figure}

We draw $V = 10^5$ independent sample realizations from the joint density $p(X_1,\dots,X_5) = f_\mu(X_1,1)f_\mu(X_2,1)f_\mu(X_3,1)f_\mu(X_4,X_5,2)$. 
\begin{center}
\begin{small}
\begin{tabular*}{0.8093\textwidth}{| c | c | c | c | }
  \hline 
	{}	&{} E &{} True &{} False  \\
  \hline
  {}l & {}$[1,0,0,1,0]$ &{} $f(X_1,X_2,X_4,X_5)f(X_3)$ & {}$f(X_1,X_2,X_3,X_4)f(X_5)$  \\
  \hline 
{}  m  & {}$[1,2,0,0,0]$  &{} $f(X_1,X_2)f(X_4,X_5)f(X_3)$ & {}$f(X_1,X_3)f(X_2,X_4)f(X_5)$  \\
  \hline
{}  n & {}$[3,1,0,0,0]$  & {}$f(X_4,X_5)f(X_1)f(X_2)f(X_3)$ & {}$f(X_2,X_4)f(X_1)f(X_3)f(X_5)$ \\
  \hline
\end{tabular*}
\end{small}
\end{center}
\textbf{Experiment} The table above shows six different factor graph models. We compare each true model against each false model. Denote the true models by $l_T$, $m_T$ and $n_T$ and the corresponding false models by $l_F$, $m_F$ and $n_F$. We note that the true cross entropy terms $H_c(p,l_T) = H_c(p,m_T) = H_c(p,n_T)$ and $H_c(p,l_L)=H_c(p,m_L)=H_c(p,n_L)$. This guarantees level playing field when comparing each true model against each false model using the surrogate cross entropy test. 

For the surrogate cross entropy test, we set $N=.2*10^4$, $M=.8*10^4$ and $k=20$. We note that the maximum value of $m_1+m_2$ for the above set of tests is $8$ and that $V/8 > (N+M)$. This choice of $N$ and $M$ therefore ensures that there are enough samples $V$ to guarantee sufficient number of independent samples for estimating individual entropies (see Section 5).

The table below lists the probability (experimental/theoretical prediction\footnote{The theoretical prediction requires estimation of constants $c_{l_i1},c_{l_i2}$ and $c_{l_i3}$. These constants were estimated from the data using oracle Monte Carlo methods which utilized the true form of the density $p$. In practice, when the true form of $p$ is never known, we adopt methods given by \cite{raykar} to estimate these constants from data.}) of choosing the false model over the true model for the various tests. 
\begin{center}
\begin{small}
\begin{tabular*}{0.7309\textwidth}{| c | c | c | c | }
  \hline 
 {} Same true vs Same false  & {} $l_T$ vs $l_F$  & {} $m_T$ vs $m_F$ & {}$n_T$ vs $n_F$  \\
  \hline
  {} Error (Exp/Theor) & {} 0.071/0.032  & {} 0.067/0.066 & {} 0.068/0.028  \\
  \hline 
{}  High true vs Low false  & {} $l_T$ vs $m_F$  & {} $l_T$ vs $n_F$ &{} $m_T$ vs $n_F$  \\
  \hline
  {} Error (Exp/Theor)&{}  0/0  & {} 0/0  & {} 0/0 \\
  \hline
  {} Low true vs High false  & {} $m_T$ vs $l_F$  & {} $n_T$ vs $l_F$ & {} $n_T$ vs $m_F$  \\
  \hline
  {} Error  {} (Exp/Theor)&  {} 0.689/0.732 & {} 0.995/1.000  & {} 0.691/0.665  \\
  \hline
\end{tabular*}
\end{small}
\end{center}
\textbf{Explanation} For the class of models above, the set of constants $\{c_{n_i1},c_{l_j1}\}$ are always negative. As a result, when comparing a high dimensional model to a low dimensional model, the additional bias will strongly tilt the test statistic towards the higher dimensional model. As a result, there is a greater chance of detecting the higher dimension model in the surrogate cross entropy test, irrespective of whether the higher dimensional model is true or false. 

To elaborate, when the high dimensional model is true and the low dimensional model is false, the bias will further tilt the test statistic towards the high dimensional model, resulting in zero false detections. On the other hand, when the low dimensional model is true, the bias in the surrogate test statistic deviates towards the high dimensional model, resulting in a high number of false detections. When we compare factor graph models within the same class of dimension, the bias from the cross entropy estimates for each model nearly cancel, resulting in a surrogate test statistic with much smaller bias as compared to the above two cases. As a result, the number of false detections is correspondingly low when comparing models within the same dimension. 

By the same argument, for factor graph models where the set of constants $\{c_{n_i1},c_{l_j1}\}$ are positive, we can conclude that the surrogate test statistic will be biased towards lower dimensional models. 

%


\section{Application to intrinsic dimension estimation}

In this work we introduce a new dimensionality estimator that is based on fluctuations of the sizes of nearest neighbor balls centered at a subset of the data points.  In this respect it is similar to Costa's $k$-nearest neighbor (kNN) graph dimension estimator \cite{Costa&etal:ICASSP04} and to Farahmand's dimension estimator based on nearest neighbor distances \cite{Farahmand&etal:ICML07}.  The estimator can also be related to the Leonenko's R\'{e}nyi entropy estimator \cite{Leonenko&etal:AnnStat08}.   However, unlike these estimators, our new dimension estimator is derived directly from a mean squared error (M.S.E.) optimality condition for partitioned kNN estimators of multivariate density functionals.  This guarantees that our estimator has the best possible M.S.E. convergence rate among estimators in its class. Empirical experiments are presented that show that this asymptotic optimality translates into improved performance in the finite sample regime.

\subsection{Problem formulation}

Let ${\cal Y} = \{\mb{Y}_1,\ldots,\mb{Y}_{T}\}$ be $T$ independent and identically distributed sample realizations in $\mathbb{R}^D$ distributed according to density $f$. Assume the random vectors in $\cal{Y}$ are constrained to lie on a d-dimensional Riemannian submanifold $\cal{S}$ of $\mathbb{R}^D$ $(d<D)$. We are interested in estimating the intrinsic dimension $d$.

\subsection{Log-length statistics}

Let $\gamma>0$ be any arbitrary number and $\alpha = \gamma/d$. Partition the $T$ samples in $\cal{Y}$ into two disjoint sets ${\cal X}$ and ${\cal Z}$ of size $\lfloor{T/2\rfloor}$ each. Denote the samples of ${\cal X}$ as ${\cal X} = \{\mb{X}_1,\ldots,\mb{X}{_{\lfloor{T/2\rfloor}}}\}$ and ${\cal Z}$ as ${\cal Z} = \{\mb{Z}_{1},\ldots,\mb{Z}{_{\lfloor{T/2\rfloor}}}\}$.

Partition ${\cal X}$ into $N$ 'target' and $M$ 'reference' samples $\{\mb{X}{_{1}},\ldots,\mb{X}{_{N}}\}$ and $\{\mb{X}{_{N+1}},\ldots,\mb{X}{_{\lfloor{T/2\rfloor}}}\}$ respectively with $N+M=\lfloor{T/2\rfloor}$. Partition $\cal{Z}$ in an identical manner. Now consider the following statistics based on the partitioning of sample space:
\begin{equation}
\mb{L_k}({\cal X}) = \frac{\gamma}{N} \sum_{i=1}^{N} \log\left(\mb{R_{k}}(\mb{X}{_i})\right), \nonumber
\end{equation}
where $\mb{R_{k}}(\mb{X}_{i})$ is the Euclidean $k$ nearest neighbor ($k$NN) distance from the target sample $\mb{X}_i$ to the $M$ reference samples $\{\mb{X}{_{N+1}},\ldots,\mb{X}{_{\lfloor{T/2\rfloor}}}\}$ . This partitioning of samples is illustrated in Fig. \ref{fig:unifball}.

\begin{figure}[!ht]
  \begin{center}
    \includegraphics[width=6in]{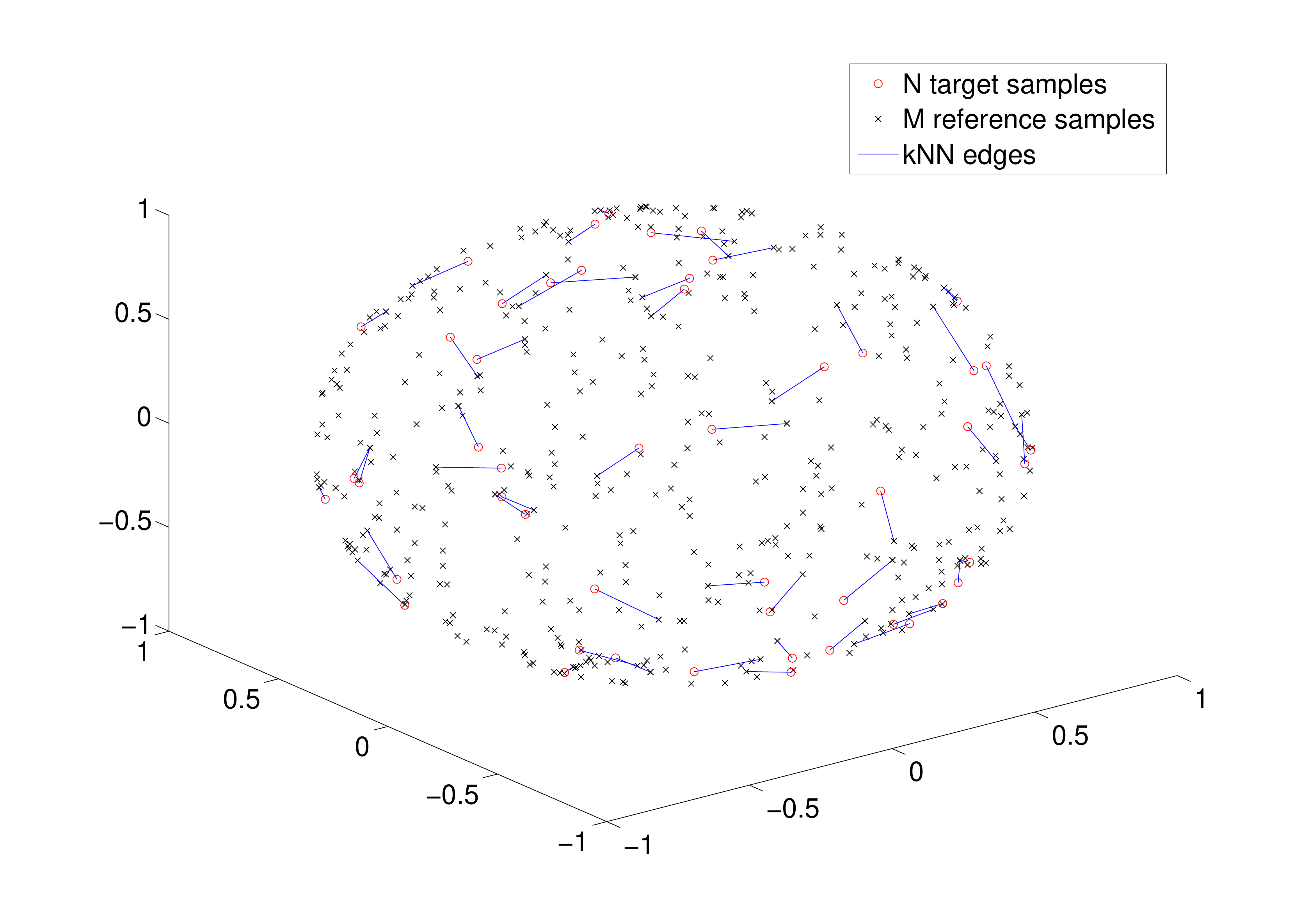}
  \end{center}
\caption{kNN edges on sphere manifold with uniform distribution for $d=2$, $D=3$, and $k=5$.}
\label{fig:unifball}
\end{figure}

\subsection{Relation to $k$NN density estimates}

Under the condition that $k/M$ is small, the Euclidean $k$NN distance $\mb{R_{k}}(\mb{X}{_i)}$ approximates the $k$NN distance on the submanifold $\cal{S}$. The $k$NN density estimate \cite{fuk} of $f$ at $\mb{X}{_i}$ based on the $M$ samples $\mb{X}{_{N+1}},\ldots,\mb{X}{_{N+M}}$ is then given by
\begin{equation}
\mb{\hat{f}_k}(\mb{X_i}) = \frac{k-1}{M}\frac{1}{c_d\mb{R_k}(\mb{X}{_i})^d} = \frac{k-1}{M}\frac{1}{\mb{V_k}(\mb{X}{_i})}, \nonumber
\end{equation}
where $c_d$ is the volume of the unit ball in $d$ dimensions and therefore $\mb{V_k}(\mb{X_i})$ is the volume of the $k$NN ball. This implies that $\mb{L_k}({\cal X})$ can be rewritten as follows:
\begin{eqnarray}
\label{eq:relation}
\mb{L_k}({\cal X}) &=& \frac{\gamma}{N} \sum_{i=1}^{N} \log\left(\mb{R_{k}}(\mb{X}{_i})\right) \nonumber \\
&=& \log\left(\frac{k-1}{Mc_d}\right)^{\alpha} + \frac{1}{N}\sum_{i=1}^{N} \log\left(\mb{\hat{f}_k}(\mb{X}{_i})\right)^{-\alpha} \nonumber \\
&=& \alpha \log (k-1) - \frac{\alpha}{N}\sum_{i=1}^{N} \log\mb{\hat{f}_k}(\mb{X}{_i}) \nonumber \\
&& -\alpha \log(c_dM). 
\end{eqnarray}
As eq.~(\ref{eq:relation}) indicates, the log-length statistics is linear with respect to $\log(k-1)$ with a slope of $\alpha$. This prompts the idea of estimating $\alpha$ (and later $d$) from the slope of $\mb{L_k}({\cal X})$ as a function of $\log(k-1)$.
\subsection{Intrinsic dimension estimate based on varying bandwidth $k$}

Let $k_1$ and $k_2$ be two different choices of bandwidth parameters. Let $\mb{L_{k_1}}({\cal X})$ and $\mb{L_{k_2}}({\cal Z})$ be the length statistics evaluated at bandwidths $k_1$ and $k_2$ using data ${\cal X}$ and ${\cal Z}$ respectively. A natural choice for the estimate of $\alpha$ would then be
\begin{eqnarray}
\mb{\mb{\hat{\alpha}}} &=& \frac{\mb{L_{k_2}}({\cal Z})-\mb{L_{k_1}}({\cal X})}{\log(k_2-1)-\log(k_1-1)} \nonumber \\
&=& \alpha + \frac{\nu}{N}\sum_{i=1}^{N} \left(\log\mb{\hat{f}_{k_2}}(\mb{Z}{_i}) - \log\mb{\hat{f}_{k_1}}(\mb{X}{_i})\right) \nonumber \\
&=& \alpha + \nu (\mb{\hat{E}_{k_2}}({\cal Z})-\mb{\hat{E}_{k_1}}({\cal X})), \nonumber
\end{eqnarray}
where 
\begin{equation}
\mb{\hat{E}_{k}}({\cal X}) = \frac{1}{N}\sum_{i=1}^{N} \log(\mb{\hat{f}_{k}}(\mb{X_i})), \nonumber
\end{equation}
and $\nu = {-\alpha}/{\log((k_2-1)/(k_1-1))}$. The intrinsic dimension estimate is related to $\mb{\hat{\alpha}}$ by the simple relation $\mb{\mb{\hat{d}}} = \gamma/\mb{\mb{\hat{\alpha}}}$. 

\subsection{Statistical properties of intrinsic dimension estimate}
\label{sec:properties}
We can relate the error in estimation of $\alpha$ to the error in dimension estimation as follows:
\begin{eqnarray}
\mb{\hat{d}}-d &=& \gamma\left(\frac{1}{\mb{\hat{\alpha}}}-\frac{1}{\alpha}\right) \nonumber \\
&=& \gamma\frac{\alpha-\mb{\hat{\alpha}}}{\mb{\hat{\alpha}}\alpha} \nonumber \\
&=& -\frac{\gamma}{\alpha^2}(\mb{\hat{\alpha}}-\alpha) + o(\mb{\hat{\alpha}}-\alpha). \nonumber
\end{eqnarray}
Define $\kappa = -{\gamma\nu}/{\alpha^2}$. We recognize that the density functional estimate $\mb{\hat{E}_{k}}({\cal X})$ is in the form of the plug-in estimators introduced in this report. Using the results on the bias, variance and asymptotic distribution of the density functional estimate $\mb{\hat{E}_{k}}({\cal X})$ established in this report and the above relation between the errors $\mb{\hat{d}}-d$ and $\mb{\hat{\alpha}}-\alpha$, we then have the following statistical properties for the estimate $\mb{\hat{d}}$:

\noindent
\textbf{Estimator bias}
\begin{eqnarray}
\expect[\mb{\hat{d}}]-d &=& \kappa c_{b_1}\left(\left({\frac{k_2}{M}}\right)^{2/d}-\left({\frac{k_1}{M}}\right)^{2/d}\right) \nonumber \\
&+& \kappa c_{b_2}\left(\left(\frac{1}{k_2}\right)-\left(\frac{1}{k_1}\right)\right) \nonumber \\
&+& o\left(\frac{1}{k_1} + \frac{1}{k_2} + \left(\frac{k_1}{M}\right)^{2/d} + \left(\frac{k_2}{M}\right)^{2/d}\right). \nonumber 
\end{eqnarray}
\noindent
\textbf{Estimator variance}
\begin{eqnarray}
\var(\mb{\hat{d}}) &=& 2\kappa^2c_v\left(\frac{1}{N}\right) + o\left(\frac{1}{M} + \frac{1}{N}\right). \nonumber
\end{eqnarray}

\noindent
\textbf{Central limit theorem}

Let $\mb{Z}$ be a standard normal random variable. Then, 
\begin{equation}
\lim_{N,M \to \infty} Pr\left(\frac{\mb{\hat{d}}-\mathbb{E}[\mb{\hat{d}}]}{\sqrt{2\kappa^2c_v/N}} \leq \alpha \right) = Pr(\mb{Z} \leq \alpha). \nonumber 
\end{equation}
\subsection{Optimal selection of parameters}
We have theoretical expressions for the mean square error (M.S.E) of the dimension estimate $\mb{\hat{d}}$, which we can optimize over the free parameters $k_1$, $k_2$, $N$ and $M$. We restrict our attention to the case $k_2=2k$; $k_1=k$. The M.S.E. of $\mb{\hat{d}}$ (ignoring higher order terms) is given by
\begin{eqnarray}
\label{M.S.E.}
\textrm{M.S.E.}(\mb{\hat{d}}) &=& (\expect[\mb{\hat{d}}]-d)^2 + \var[\mb{\hat{d}}] \nonumber \\
&=& \left(C_{b_1}\left({\frac{k}{M}}\right)^{2/d} + C_{b_2}\left(\frac{1}{k}\right)\right)^2 \nonumber \\
&+& C_v\left(\frac{1}{N}\right).
\end{eqnarray}
where $C_{b_1} = \kappa2^{({2/d}-1)}$, $C_{b_2} = \kappa/4$ and $C_v = 2\kappa^2c_v$. 
\subsection*{Optimal choice of bandwidth}
The optimal value of $k$ w.r.t the M.S.E. is given by 
\begin{eqnarray}
\label{eq:optkD}
k_{opt} &=& \lfloor{k_0M^{\frac{2}{2+d}}}\rfloor.
\end{eqnarray}
where the constant $k_0 = (|C_{b_2}|d/2|C_{b_1}|)^{\frac{d}{d+2}}$. 
\subsection*{Optimal partitioning of sample space}
Under the constraint that $N+M = \lfloor{T/2\rfloor}$ is fixed, the optimal choice of $N$ as a function of $M$ is then given by 
\begin{equation}
\label{eq:optp}
N_{opt}=\lfloor N_0M^{\frac{6+d}{2(2+d)}} \rfloor,
\end{equation}
where the constant $N_0 = {\frac{\sqrt{C_v(2+d)}}{2b_0}}$.

\subsection{Improved estimator based on correlated error}
Consider the following alternative estimator for $\alpha$:
\begin{eqnarray}
\mb{\tilde{\alpha}} &=& \frac{\mb{L_{k_2}}({\cal X})-\mb{L_{k_1}}({\cal X})}{\log(k_2-1)-\log(k_1-1)} \nonumber \\
&=& \alpha + \kappa (\mb{\hat{E}_{k_2}}({\cal X})-\mb{\hat{E}_{k_1}}({\cal X})), \nonumber
\end{eqnarray}
and the corresponding density estimate $\mb{\tilde{d}}$ which satisfies
\begin{equation}
\mb{\tilde{d}}-d = -\frac{\gamma}{\alpha^2}(\mb{\tilde{\alpha}}-\alpha) + o(\mb{\hat{\alpha}}-\alpha), \nonumber
\end{equation}
where both the length statistics at bandwidths $k_1$ and $k_2$ are evaluated using the same sample $X$. The density functional estimates $\mb{\hat{E}_{k_1}}({\cal X})$ and $\mb{\hat{E}_{k_2}}({\cal X})$ will be highly correlated (as compared to the independent quantities $\mb{\hat{E}_{k_1}}({\cal X})$ and $\mb{\hat{E}_{k_2}}({\cal Z})$). This implies that the variance of the difference $\mb{\hat{E}_{k_2}}({\cal X}) - \mb{\hat{E}_{k_1}}({\cal X})$ will be smaller when compared to $\mb{\hat{E}_{k_2}}({\cal Z}) - \mb{\hat{E}_{k_1}}({\cal X})$, (while the expectation remains the same).

Since the estimator bias is unaffected by this modification, the variance reduction suggests that $\tilde{d}$ will be
an improved estimator as compared to $\mb{\hat{d}}$ in terms of M.S.E.. In order to obtain statistical properties for the improved estimator $\mb{\tilde{d}}$ (equivalent to the properties developed
in Section \ref{sec:properties} for the original estimator $\mb{\hat{d}}$), we need to analyze
the joint distribution between $\mb{\hat{f}_{k_1}}(X_i)$ and $\mb{\hat{f}_{k_2}}(X_j)$ for two distinct
values $k_1$ and $k_2$. Our theory, at present, cannot address the case of distinct bandwidths $k_1$ and $k_2$.

Since the estimate $\mb{\tilde{d}}$ has smaller M.S.E. compared to
$\mb{\hat{d}}$, M.S.E. predictions for the estimate $\mb{\hat{d}}$ can serve as
upper bounds on the M.S.E. performance of the improved estimate $\mb{\tilde{d}}$.

\subsection{Simulations}
\label{sec:illust}

\begin{figure}[!t]
  \begin{center}
    \includegraphics[width=4in]{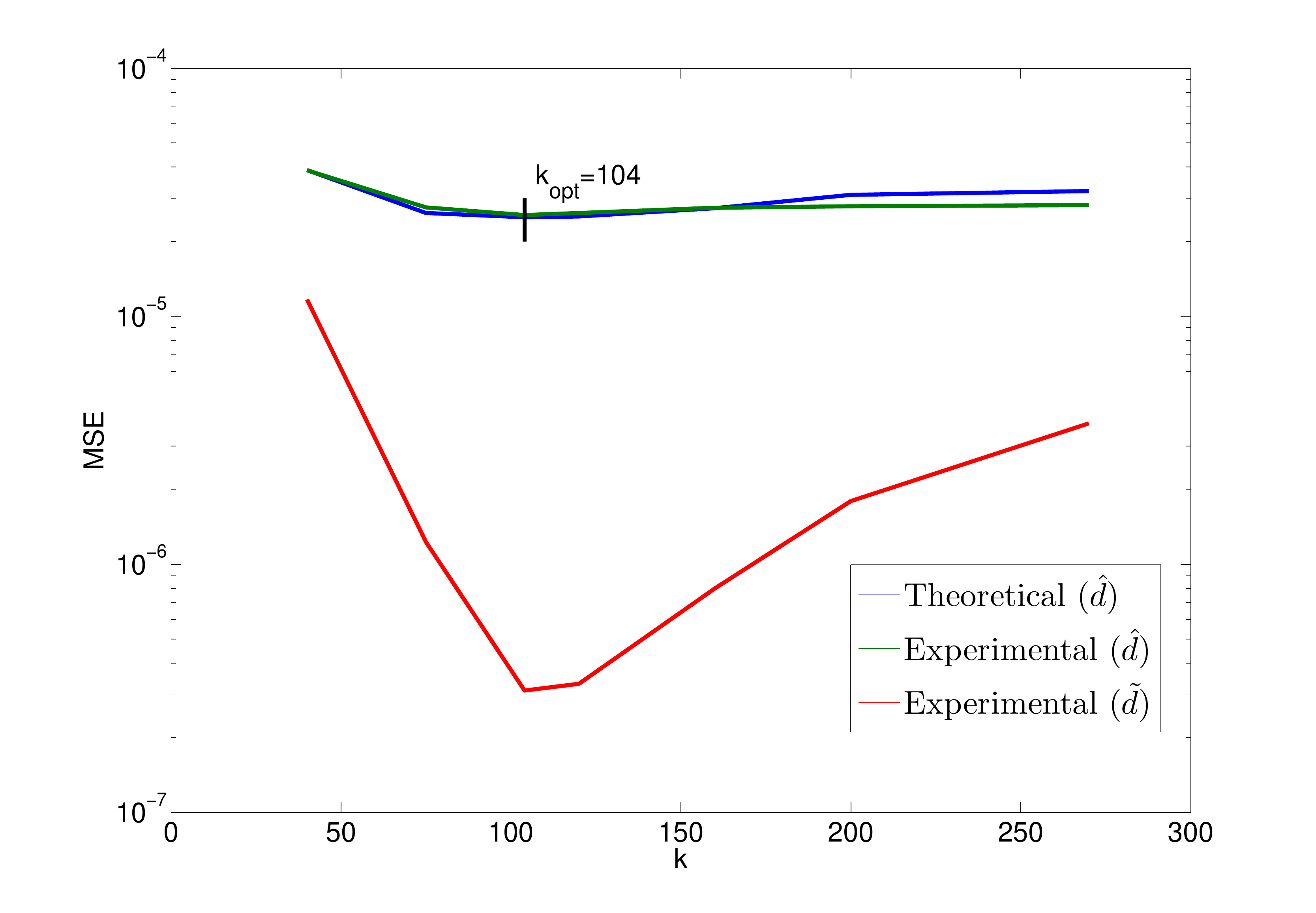}
  \end{center}
\caption{Comparison of theoretically predicted and experimental M.S.E. for varying choices of $k$. The experimental performance of the estimator $\mb{\hat{d}}$ is in excellent agreement with the theoretical expression and, as predicted by our theory, the modified estimator $\mb{\tilde{d}}$ significantly outperforms $\mb{\hat{d}}$.}
\label{fig:res1}
\end{figure}

We generate $T=10^5$ samples ${\cal B}$ drawn from a $d=2$ mixture density $f_m = .8f_\beta + .2f_u$, where $f_\beta$ is the product of two $1$ dimensional marginal beta distributions with parameters $\alpha=2$, $\beta=2$ and $f_u$ is a uniform density in $2$ dimensions. These samples are then projected to a $3$-dimensional hyperplane in $\mathbb{R}^3$ by applying the transformation ${\cal Y}=U{\cal B}$ where $U$ is a $3 \times 2$ random matrix whose columns are orthonormal. We apply our intrinsic dimension estimates on the samples ${\cal Y}$.

\subsection*{Optimal selection of free parameters}
In our first experiment, we theoretically compute the optimal choice of $k$ for a fixed partition with $M=3.5\times10^4$ and $N=1.5\times10^4$. We then show the variation of the theoretical and experimental M.S.E. of the estimate $\mb{\hat{d}}$ and the experimental M.S.E. of the improved estimate $\mb{\tilde{d}}$ with changing bandwidth $k$ in Fig. \ref{fig:res1}. In our second experiment, we compute the optimal partition according to eq.~(\ref{eq:optp}) and show the variation of M.S.E. with varying choices of partition in Fig. \ref{fig:res2}.

From our experiments, we see that there is good agreement between our theory and simulations. As a consequence, we find the theoretically predicted optimal choices of $k,N$and $M$ to minimize the observed M.S.E.. In addition, as predicted by our theory, the modified estimator $\mb{\tilde{d}}$ significantly outperforms $\mb{\hat{d}}$. The theoretically predicted M.S.E. for $\mb{\hat{d}}$ therefore serves as a strict upper bound for the M.S.E. of the improved estimator $\mb{\tilde{d}}$.

\begin{figure}[!t]
  \begin{center}
    \includegraphics[width=4in]{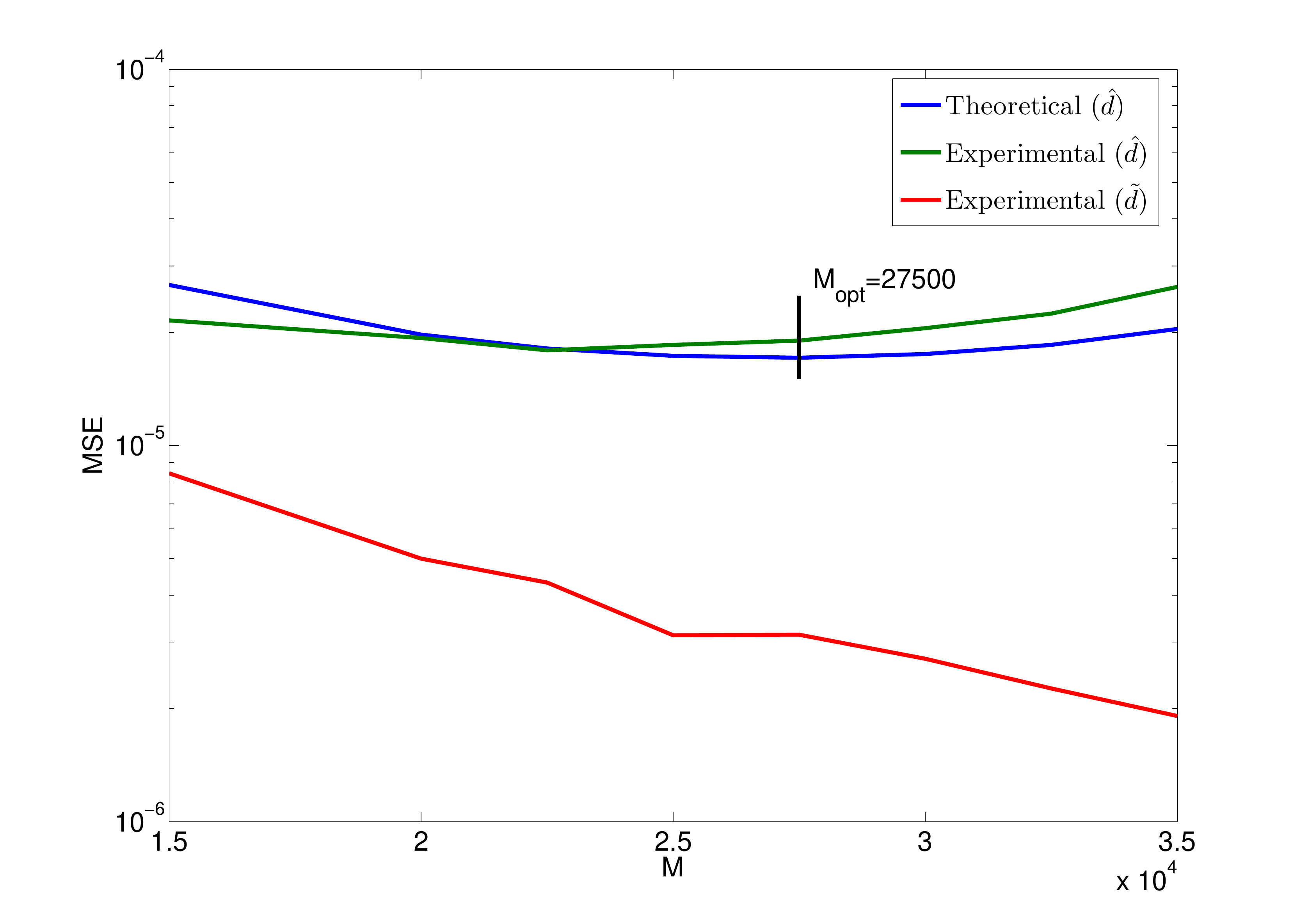}
  \end{center}
\caption{Comparison of theoretically predicted and experimental M.S.E. for varying choices of $M$. The experimental performance of the estimator $\mb{\hat{d}}$ is in excellent agreement with the theoretical expression and, as predicted by our theory, the modified estimator $\mb{\tilde{d}}$ significantly outperforms $\mb{\hat{d}}$.}
\label{fig:res2}
\end{figure}

\subsection*{Comparison of dimension estimation methods}

We compare the performance of our proposed dimension estimators to the estimated proposed by Frahmand et.~al.~\cite{Farahmand&etal:ICML07} (denote as $\mb{\hat{d}}_{f}$) and Costa et.~al.~\cite{Costa&etal:ICASSP04} (denote as $\mb{\hat{d}}_j$). 

Expressions for the optimal bandwidth $k$ (eq.~(\ref{eq:optk})) and partition $N,M$ (eq.~(\ref{eq:optp})) depend on the unknown intrinsic dimension $d$ and constants $c_{b_1}$, $c_{b_2}$ and ${c_v}$ which depend on unknown density $f$. The constants $c_{b_1}$, $c_{b_2}$ and ${c_v}$ can be estimated from the data using plug-in methods similar to the ones used by Raykar et.~al.~\cite{raykar} for optimal bandwidth selection for kernel density estimation . To establish the potential advantages of our dimension estimators we compare an omniscient optimal form of our estimator, for which the true
values of these constants are known, to a suboptimal form of our estimator that does not know the constants.

For the optimal estimator, we theoretically compute the optimal choice for $k$, $N$ and $M$ for different choices of total sample size $T$ (sub-sampled from the initial $10^5$ samples), and use these optimal parameters for the estimators $\mb{\hat{d}}$ and $\mb{\tilde{d}}$. We use this optimal choice of bandwidth $k$ for the estimators $\mb{\hat{d}}_{f}$ and $\mb{\hat{d}}_j$ as well (partitioning not applicable). For the suboptimal estimator, we arbitrarily choose the parameters as follows: fixed $k$ = 20, $N = T/50$, $M=\lfloor{T/2\rfloor}-N$.

\begin{figure}[!t]
  \begin{center}
    \includegraphics[width=4in]{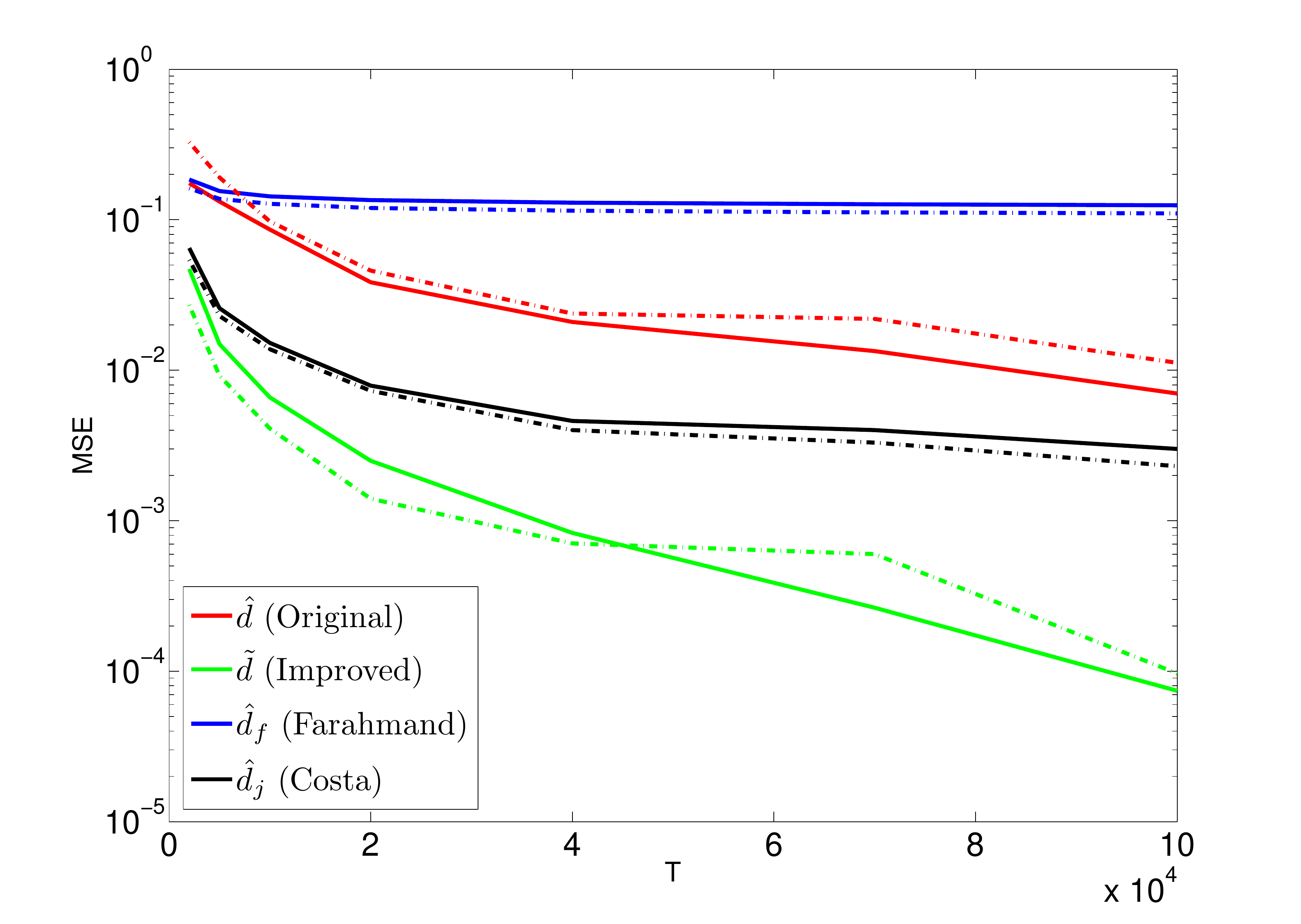}
  \end{center}
\caption{Comparison of performance of dimension estimates (Solid line: Optimal (optimal choice of $k$,$N$ and $M$ as per eq.~(\ref{eq:optk}) and eq.~(\ref{eq:optp})); Dashed line: Suboptimal (fixed $k$ = 20, $N = T/50$, $M=\lfloor{T/2\rfloor}-N$)): The proposed improved kNN distance estimator outperforms all other estimators considered.}
\label{fig:comp}
\end{figure}

The performance of these estimators as a function of sample size $T$ is shown in Fig. \ref{fig:comp}. Estimators with optimal choice of parameters are indicated in solid line, and the suboptimal estimators are indicated in dashed lines.

From our experiments we see that the performance of the original estimator $\mb{\hat{d}}$ with suboptimal choice of parameters is marginally inferior when compared to the estimator with optimal choice of parameters. This does not hold for the other estimators as can be expected since the parameters are optimized w.r.t. the performance of $\mb{\hat{d}}$.

We note that the improved estimator $\mb{\tilde{d}}$ outperforms all other estimators while the performance of our original estimator $\mb{\hat{d}}$ is sandwiched between $\mb{\hat{d}}_f$ and $\mb{\hat{d}}_j$. We conjecture that the performance of $\mb{\hat{d}}_j$ is superior to $\mb{\hat{d}}$ for the same reason that $\mb{\tilde{d}}$ outperforms $\mb{\hat{d}}$: correlated error between different length statistics. 

\subsection*{Anomaly detection in Abilene network data}

Anomalies can be detected in router netowrks by estimating the local dimension at each time point and monitoring change in dimension. The data used is the number of packets sent by each of the 11 routers on the abiline network between January 1-2, 2005. A sample is taken every 5 minutes, leading to 576 samples with an extrinsic dimension pf 11.

\begin{figure}[!t]
  \begin{center}
    \includegraphics[width=4in]{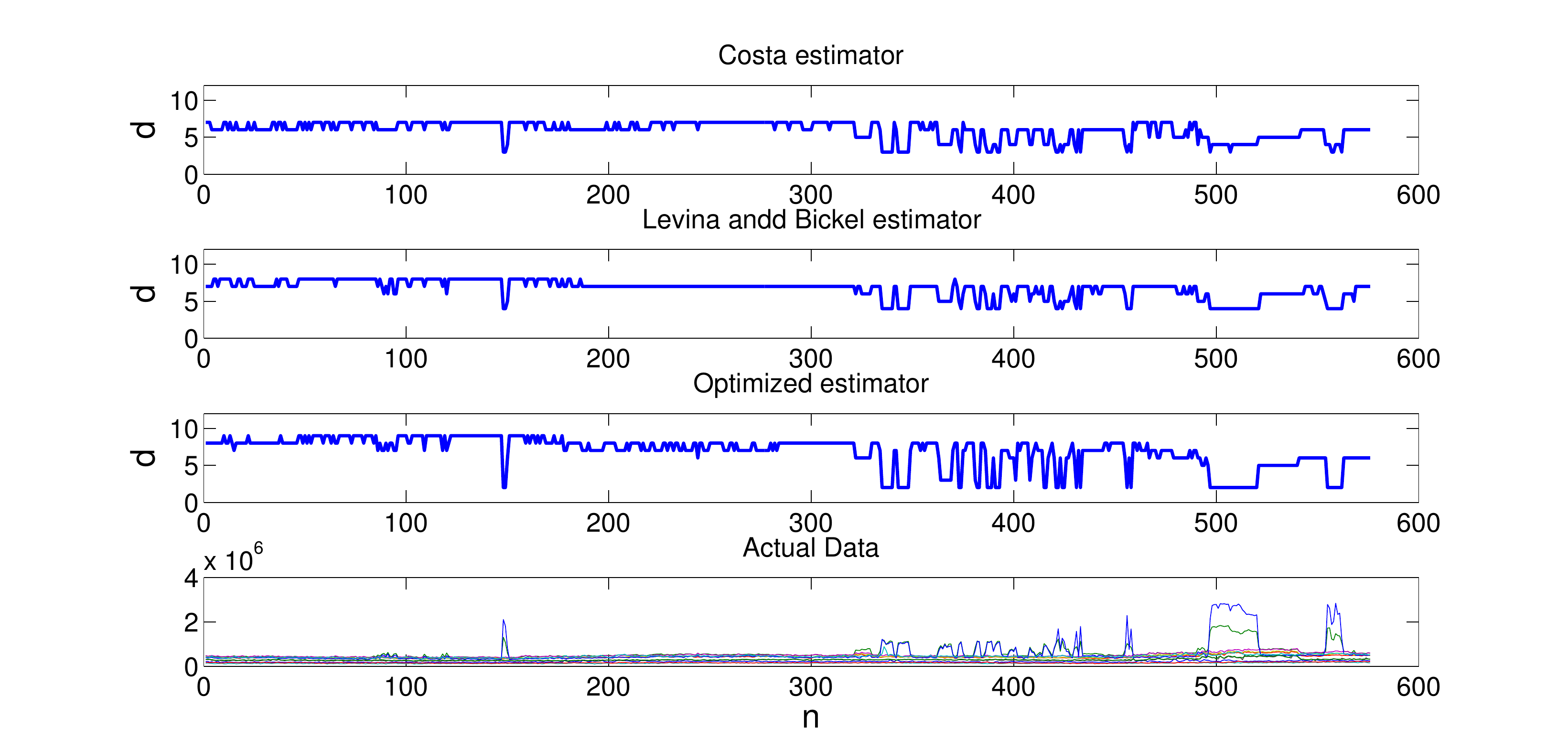}
  \end{center}
\caption{Comparison of performance of dimension estimates for anomaly detection in Abilene network data.}
\label{fig:abilene}
\end{figure}

The performance of different dimension estimators is shown in Fig.~\ref{fig:abilene}. We know that simulataneous peaks in router traffic should imply strong correlation between the routers and therefore lower intrinsic dimension. This behaviour is clearly reflected better by the optimized estimator as compared to the estimator of Costa et.~al.~\cite{Costa&etal:ICASSP04} and Levina and Bickel~\cite{levbick}.

\section{Conclusion}
\label{sec:conc}
A new class of boundary compensated bipartite k-NN density plug-in estimators was proposed for estimation of smooth non-linear functionals of densities that are strictly bounded strictly away from 0 on their finite support. These estimators, called bipartite plug-in (BPI) estimators, correct for bias due to boundary effects and outperform previous $k$-NN entropy estimators in terms of MSE convergence rate. Expressions for asymptotic bias and variance of the estimator were derived estimator in terms of the sample size, the dimension of the samples and the underlying probability distribution. In addition, a central limit theorem was developed for the proposed BPI estimators. The accuracy of these asymptotic results were validated through simulation and it was established that the theory can be used to specify optimal finite sample estimator tuning parameters such as bandwidth and optimal partitioning of data samples.

Our theory has two important by-products: (1) We established similarity between the moments of $k$-NN density estimates and kernel density estimates. This in turn implies that plug-in estimators based on $k$-NN density estimators and kernel density estimators have asymptotically equal rates of convergence. (2) We developed an algorithm for detection and correction of density estimates at boundary points for densities with finite support. This correction helps reduce the bias of density estimates at the boundaries of the support of the density, thereby reducing the overall bias of the plug-in estimators.

Using the theory presented in the paper, one can tune the parameters of the plug-in estimator to achieve
minimum asymptotic estimation MSE. Furthermore, the theory can be used to specify the minimum necessary sample
size required to obtain requisite accuracy. This in turn can be used to predict and optimize performance in
applications like structure discovery in graphical models and dimension estimation for support sets of low intrinsic dimension.
We applied our theory to the problem of estimating Shannon entropy and Shannon mutual information. Furthermore, we used the Shannon entropy estimator to discover structure in high dimensional data and to determine the intrinsic dimension of data samples.


\newpage
For the reader's convenience, the notation used in this paper is listed in the table below.
\begin{center}
  \small{
\begin{tabular}{|c|l|}
  \hline
  Notation & Description \\
  \hline \hline
$\hat{\mb{G}}_N(\mb{\tilde{f}}_k)$ & BPI estimator (\ref{eq:plugin}) \\
  $\hat{\mb{G}}_{N,BC}(\mb{\tilde{f}}_k)$ & BPI estimator with bias compensation (\ref{eq:pluginwithbc}) \\
  $g_1(k,M), g_2(k,M)$ & Bias correction factors \\
  ${\cal S}$ & Support of density $f$ \\
  $d$ & dimension of support ${\cal S}$  \\
  $c_d$ & unit ball volume in $d$ dimensions \\
  $\{\mb{X}_1, \ldots, \mb{X}_T, \mb{Y}, \mb{Z}\}$ & $T+2$ independent realizations drawn from $f$ \\
  ${\cal X}_N$ & $\{\mb{X}_1, \ldots, \mb{X}_N\}$ \\
  ${\cal X}_M$ & $\{\mb{X}_{N+1}, \ldots, \mb{X}_{N+M}\}$ \\
  ${\cal S}_I$ & Interior of support \\
  ${\cal I}_N$ & Interior points subset of ${\cal X}_N$ \\
  ${\cal B}_N$ & Boundary points subset of ${\cal X}_N$ \\
  $\mb{Z}_{-1}$ & Closest interior point to $\mb{Z}$; $\mb{Z}_{-1} = \text{arg} \min_{x \in {\cal S}_I} d(x,{\mb{Z}})$ \\
  $\mb{X}_{n(i)}$ & $\mb{X}_{n(i)} \in \cal{I}_N$ is the interior sample point that is closest to $\mb{X}_{i} \in {\cal B}_N$ \\
  $\delta$ & Constant; $\delta \in (2/3,1)$ \\
  $\epsilon_{BC} = N \exp(-3k^{(1-\delta)})$ & Probability of misclassification of $x \in {\cal S-S}_I$ as interior point \\
  $\mb{d}_k(X)$ & $k$-NN ball radius \\
  $\mb{S}_k(X)$ & $k$-NN ball \\
  $\mb{V}_k(X)$ & $k$-NN ball volume \\
  $\mb{P}(X)$ & Coverage function \\
  $\hat{\mb{f}}_{k}(X)$ & $k$-NN density estimate \\
  $\tilde{\mb{f}}_{k}(X)$ & Boundary corrected $k$-NN density estimate \\
  $g^{(n)}(x,y)$ & $n$-th derivative of $g(x,y)$ wrt $x$ \\
  $\mb{p}$ & beta random variable with parameters $k,M-k+1$ \\
  $\alpha_{frac}$ & Proportionality constant; $M = \alpha_{frac}T$ and $N = (1-\alpha_{frac})T$ \\
  $\epsilon_0$, $\epsilon_\infty$ & constants such that $\epsilon_0 \leq f(x) \leq \epsilon_\infty$ $\forall x \in {\cal S}$ \\
  $2\nu$ & Number of times $f$ is assumed to be differentiable \\
  $\lambda$ & Number of times $g(x,y)$ is assumed to be differentiable wrt $x$ \\
  $c_1,..,c_5$ & Constants appearing in Theorems III.1, III.2, III.3 and IV.1, IV.2, IV.3 \\
  ${\cal C}(k)$ & Function which satisfies the rate of decay condition ${\cal C}(k)= O(e^{-3k^{(1-\delta)}})$ \\
  $k_M$ & $k_M = (k-1)/M$ \\
  $\natural(X)$ & The event $\mb{P}(X)>(1-p_k)k_M$ \\
  $\natural_{-1}(X)$ & The event $\mb{P}(X)<(1+p_k)k_M$ \\
  $\natural \natural(X)$ & The event $(1-p_k)k_M < \mb{P}(X) < (1+p_k)k_M$ \\
  $\mb{e}_k(X)$ & Error function $\mb{e}_k(X) = \hat{\mb{f}}_k(X) - \expect[\hat{\mb{f}}_k(X) \mid X] $ \\
  $\mb{e}(X)$ & Error function $\mb{e}(X) = \tilde{\mb{f}}_k(X) - \expect[\tilde{\mb{f}}_k(X) \mid X] $\\       
\hline
\end{tabular}
}
\end{center}
\newpage

\appendix
\appendixpage

\section{Uniform kernel density estimation}
\label{uniformmoments}

Throughout this section, we will derive results on moments of the uniform kernel density estimates for points in the set ${\cal S'} = \{X: \mb{S_u}(X) \subset {\cal S} \} $. This definition implies that the density $f$ has continuous partial derivatives of order $2r$ in the uniform ball neighborhood for each $X \in {\cal S'}$ where $r$ satisfies the condition $2r(1-t)/d > 1$. This excludes the set of points close to the boundary of the support, where the continuity assumption of the density is not satisfied. We will deal with these points in Appendix C. 

Let $\mb{X}_{1},..,\mb{X}_{M}$ denote $M$ i.i.d realizations of the density f. We will assume that $f$ is continuously differentiable evrywhere in the interior of the sWe seek to estimate the density at $X$ from the $M$ i.i.d realizations $\mb{X}_{1},..,\mb{X}_{M}$. Let $c_d$ denote the volume of a unit hyper-sphere in $d$ dimensions. The uniform kernel density estimator is defined as follows:

\subsection{Uniform kernel density estimator}

The \emph{uniform kernel} density estimator is defined below. The volume of the uniform kernel is given by

\begin{equation}
  V_u(X) = \frac{k}{M},
\end{equation}

and the kernel region is given by

\begin{equation}
S_u(X) = \{Y:c_d||X-Y||^d \leq V_u\}.
\end{equation}

$\mb{l_u}(X)$ denotes the number of points falling in $S_u(X)$ 

\begin{equation}
  \mb{l_u}(X) = \Sigma_{i=1}^{M} 1_{X_i \in S_u(X)},
\end{equation}

and the \emph{uniform kernel} density estimator is defined by

\begin{equation}
  \hat{\mb{f}}_\mb{u}(X) = \frac{\mb{l_u}(X)}{MV_u(X)}.
\end{equation}

The \emph{coverage} of the \emph{uniform kernel} is defined as

\begin{eqnarray}
U(X) = \int_{S_u(X)}f(z) dz = \expect{[1_{\mb{Z} \in S_u(X)}]}.
\end{eqnarray}

We observe that $\mb{l_u}(X)$ is a binomial random variable with parameters $M$ and $U(X)$. Figure \ref{fig-label3} illustrates the \emph{uniform kernel} density estimate.

\begin{figure}[!ht]
  \begin{center}
    \includegraphics[width=6in]{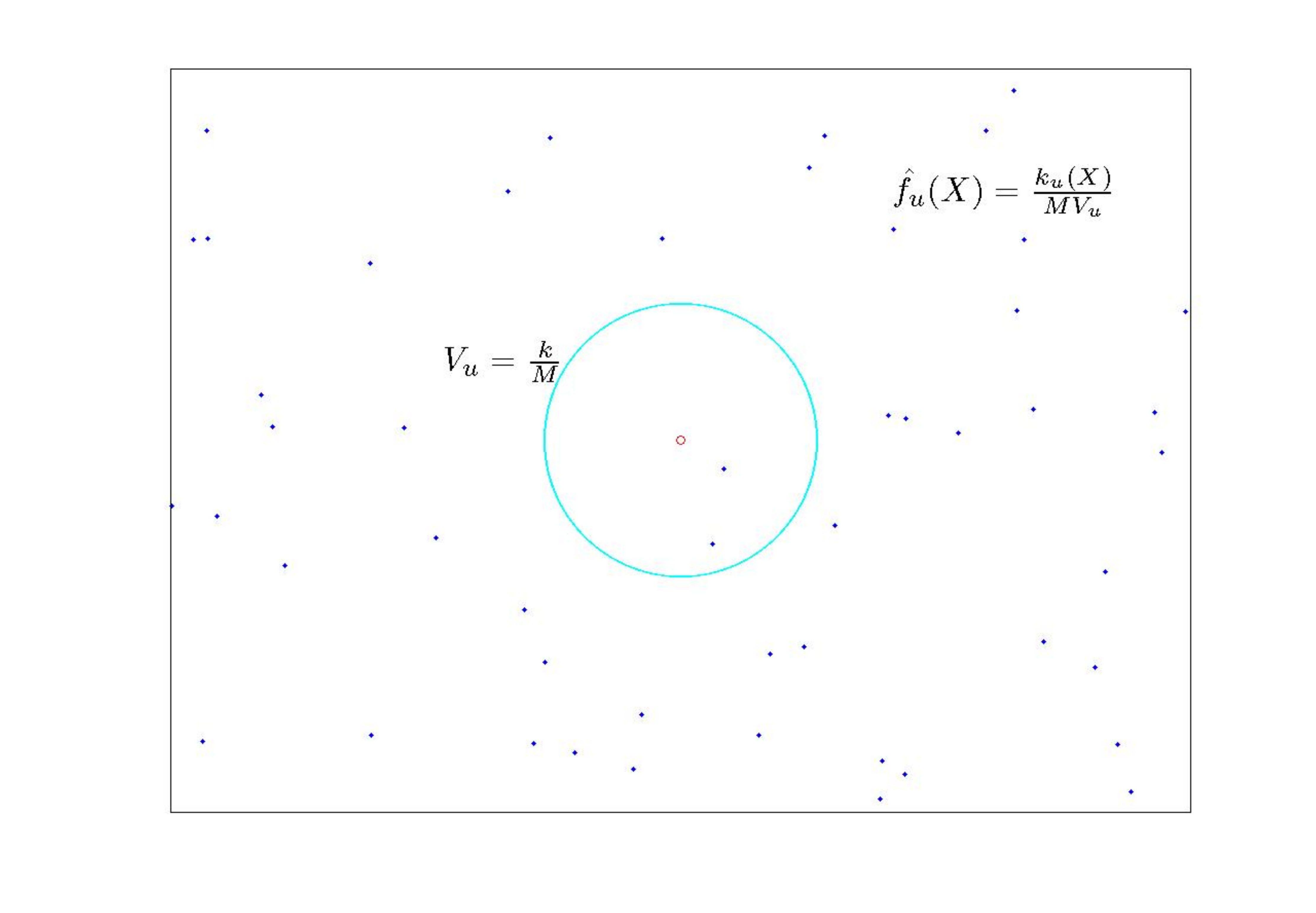}
  \end{center}

  \caption{\small Uniform kernel density estimator.}
  \label{fig-label3}
\end{figure}

\subsection{Taylor series expansion of {coverage}}

\textbf{We assume that the density $f$ has continuous partial derivatives of third order} in a neighborhood of $X$. For small volumes $V_u(X)$ (which is equivalent to the condition that $k/M$ is small), we can represent the {coverage} function $U(X)$ by using a  third order Taylor series expansion of $f$ about about $X$ \cite{fuk}.

\begin{eqnarray}
U(X) &=& \int_{S_u(X)}f(Z) dZ \nonumber \\
&=& f(X)V_u(X) + c(X)V_u^{1+2/d}(X) + o(V_u^{1+2/d}(X)) \nonumber \\
&=& f(X)\frac{k}{M} + c(X){\left(\frac{k}{M}\right)}^{1+2/d} + o\left({\left(\frac{k}{M}\right)}^{1+2/d}\right),  
\end{eqnarray}

where $c(X)=\Gamma^{(2/d)}(\frac{n+2}{2})tr[\nabla^2(f(X))]$. 

\subsection{Concentration inequalities for uniform kernel density}

Because $\mb{l_u}(X)$ is a binomial random variable, we can apply standard Chernoff inequalities to obtain concentration bounds on the density estimate. $\mb{l_u}(X)$ is a binomial random variable with parameters $M$ and $U(X)$.

\subsubsection{Concentration around true density} 

For $0<p<1/2$,
\begin{equation}
Pr({\mb{l_u}(X)>(1+p)MU(X)}) \leq e^{-MU(X)p^2/4}, 
\end{equation}
and 
\begin{equation}
Pr({\mb{l_u}(X)<(1-p)MU(X)}) \leq e^{-MU(X)p^2/4}.
\end{equation}

Using the Taylor expansion of coverage, we then have
\begin{equation}
Pr({\hat{\mb{f}}_\mb{u}(X)>(1+p)(f(X)+O((k/M)^{2/d}))}) \leq \sim e^{-p^2kf(X)/4}, 
\end{equation}
and 
\begin{equation}
Pr({\hat{\mb{f}}_\mb{u}(X)<(1-p)(f(X)+O((k/M)^{2/d}))}) \leq \sim e^{-p^2kf(X)/4}.
\end{equation}

This then implies that 
\begin{equation}
Pr({\hat{\mb{f}}_\mb{u}(X)>(1+p)f(X)}) \leq \sim e^{-p^2kf(X)/4}, 
\end{equation}
and 
\begin{equation}
Pr({\hat{\mb{f}}_\mb{u}(X)<(1-p)f(X)}) \leq \sim e^{-p^2kf(X)/4}.
\end{equation}

Let $\mb{X}$ be a random variable with density $f$ independent of the $M$ i.i.d realizations $\mb{X}_{1},..,\mb{X}_{M}$.
Then,
\begin{eqnarray}
Pr({\hat{\mb{f}}_\mb{u}(\mb{X})>(1+p)f(\mb{X})}) &=& \expect_\mb{X} [Pr({\hat{\mb{f}}_\mb{u}(\mb{X})>(1+p)f(\mb{X})})] \nonumber \\
&\leq& \expect[\sim(e^{-p^2kf(\mb{X})/4})] \nonumber \\
&=& \sim e^{-p^2k/4},
\end{eqnarray}
and
\begin{eqnarray}
Pr({\hat{\mb{f}}_\mb{u}(\mb{X})<(1-p)f(\mb{X})}) &=& \expect_\mb{X} [Pr({\hat{\mb{f}}_\mb{u}(\mb{X})<(1-p)f(\mb{X})})] \nonumber \\
&\leq& \expect[\sim(e^{-p^2kf(\mb{X})/4})] \nonumber \\
&=& \sim e^{-p^2k/4}.
\end{eqnarray}

\subsubsection{Concentration away from $0$}

We can also bound the density estimate away from $0$ as follows:
\begin{eqnarray}
Pr(\hat{\mb{f}}_\mb{u}(\mb{X})=0) &=& \expect_\mb{X} [Pr(\hat{\mb{f}}_\mb{u}(\mb{X})=0] \nonumber \\
&=& \expect[(1-U(X))^M] \nonumber \\
&=& \expect [(1-(kf(X)+o(k)/M)^M] \nonumber \\
&=& \expect [((1-(kf(X)+o(k)/M)^{M/(kf(X)+o(k))})^{kf(X)+o(k)}] \nonumber \\
&=& \expect [\sim (1/e)^{kf(X)+o(k)}] \nonumber \\
&=& \sim e^{-k}. 
\end{eqnarray}

\subsection{Central Moments}

Define the error function of the uniform kernel density,
\begin{eqnarray}
\mb{e_u}(X) &=& \hat{\mb{f}}_\mb{u}(X)-\expect[\hat{\mb{f}}_\mb{u}(X)]. 
\end{eqnarray}
The probability mass function of the binomial random variable $\mb{l_u}(X)$ is given by
\begin{equation}
Pr(\mb{l_u}(X)=l_x) = \binom{M}{l_x} (U(X))^{l_x}(1-U(X))^{M-l_x}. \nonumber \\
\end{equation}
Since $\mb{l_u}(X)$ is a binomial random variable, we can easily obtain moments of the uniform kernel density estimate. These are listed below.

First Moment:
\begin{eqnarray}
\expect{[\hat{\mb{f}}_\mb{u}(X)]}-f(X) &=& \frac{M}{k}U(X) - f(X) \nonumber \\
&=& c(X)\left({\frac{k}{M}}\right)^{2/d} + o\left(\left({\frac{k}{M}}\right)^{2/d}\right).
\end{eqnarray}

Second Moment:
\begin{eqnarray}
\label{secmom}
\var[\hat{\mb{f}}_\mb{u}(X)] &=& \expect{[\mb{e}^2_\mb{u}(X)]}  \nonumber \\
&=& \frac{M}{k^2}{U(X)(1-U(X))} \nonumber \\
&=& f(X)\frac{1}{k} + o\left(\frac{1}{k}\right).
\end{eqnarray}

Higher Moments:
For any integer $r\geq3$, 
\begin{eqnarray}
\label{fourmom}
\expect{[\mb{e}^r_{u}(X)]} &=& O\left(\frac{1}{k^{r/2}}\right).
\end{eqnarray}

\subsection{Covariance}

Let $X$ and $Y$ be two distinct points. Clearly the density estimates at $X$ and $Y$ are not independent. We expect the density estimates to have positive covariance if $X$ and $Y$ are close and have negative covariance if $X$ and $Y$ are far. This is illustrated in Figure \ref{fig-label8}.

\begin{figure}[ht!]
  \begin{center}
    \includegraphics[width=6in]{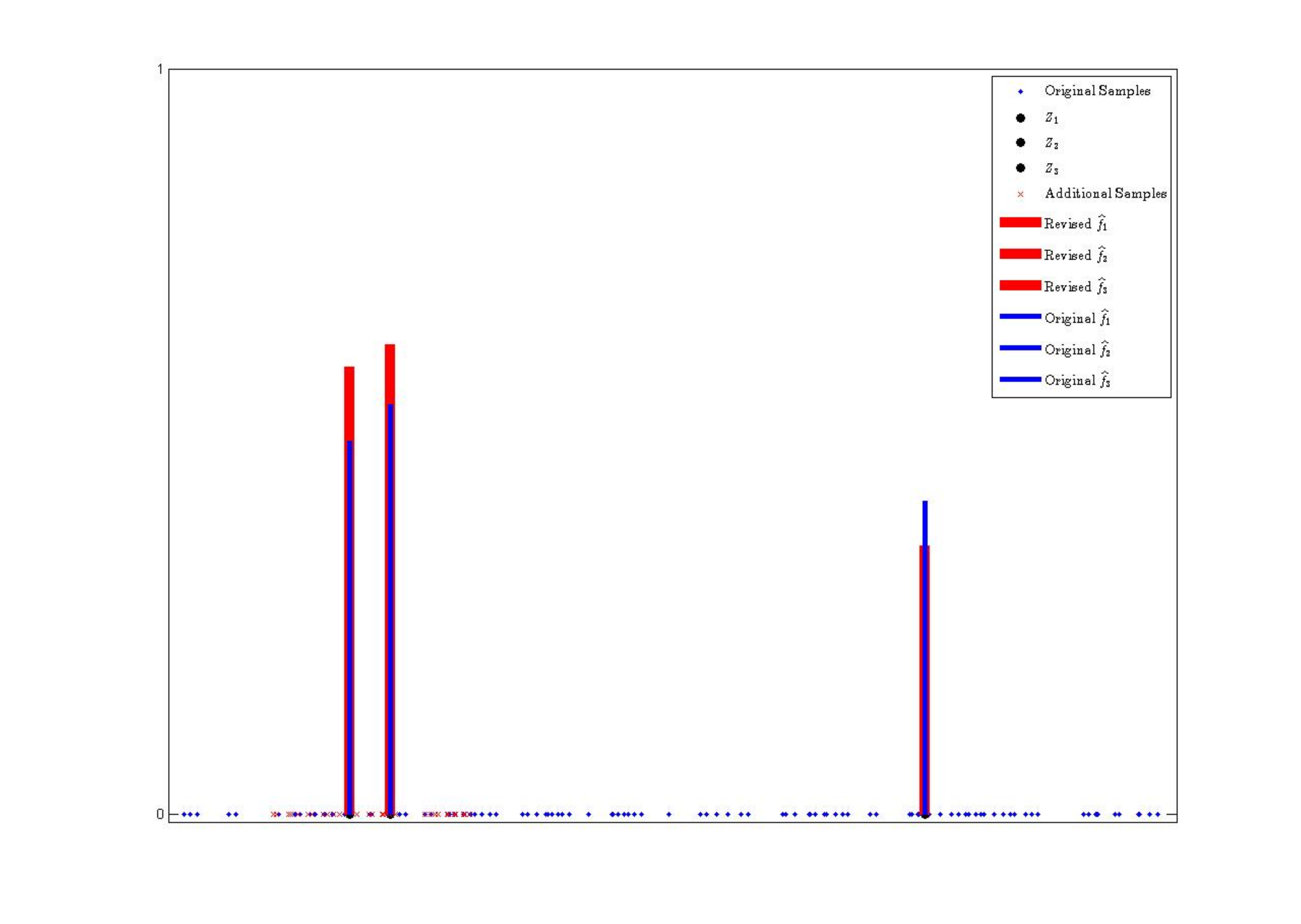}
  \end{center}

  \caption{\small Covariance between uniform kernel density estimates.}
  \label{fig-label8}
\end{figure}

Observe that the uniform kernels are disjoint for the set of points given by $\Psi_u := \{X,Y\}:||X-Y||\geq 2(k/c_dM)^{1/d}$, and have finite intersection on the complement of $\Psi_u$. Indeed we will show that when the uniform balls intersect (and therefore $X$ and $Y$ are close), the density estimates have positive covariance and that they have negative covariance when the uniform kernels are disjoint. Intersecting and disjoint balls are illustrated in Figure \ref{fig-label88}.

\begin{figure}[ht!]
  \begin{center}
    \includegraphics[width=6in]{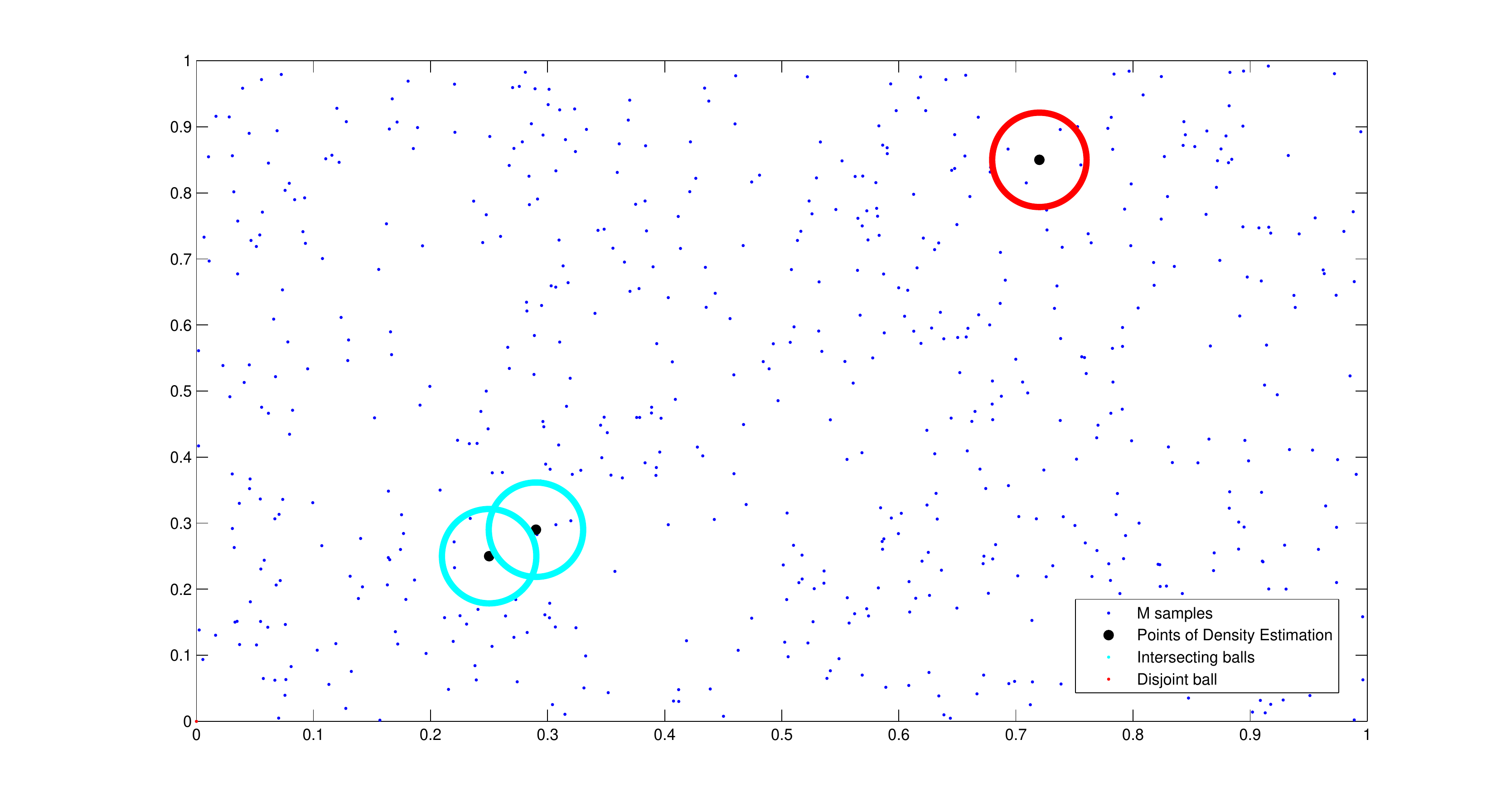}
  \end{center}

  \caption{\small Intersecting and disjoint balls.}
  \label{fig-label88}
\end{figure}

Define,
\begin{equation}
U(X,Y) := \expect[1_{\mb{Z} \in S_u(X)}1_{\mb{Z} \in S_u(Y)}]. 
\end{equation}
%
\paragraph{Intersecting balls}

\begin{lemma}
\label{intersectu}

For a fixed pair of points $\{X,Y\} \in \Psi_u$,

\begin{equation}
Cov[\mb{e_u}(X),\mb{e_u}(Y)] = \frac{-f(X)f(Y)}{M} + o\left(\frac{1}{M}\right). \nonumber 
\end{equation}

\end{lemma}

\begin{proof}

For $\{X,Y\} \in \Psi_u$, we have that $1_{\mb{Z} \in S_u(X)}1_{\mb{Z} \in S_u(Y)} = 0$ and therefore $U(X,Y) =0$.

We then have,

\begin{eqnarray}
Cov[\mb{e_u}(X),\mb{e_u}(Y)] &=& \expect[(\hat{\mb{f}}_\mb{u}(X)-\expect[\hat{\mb{f}}_\mb{u}(X)])(\hat{\mb{f}}_\mb{u}(Y)-\expect[\hat{\mb{f}}_\mb{u}(Y)])] \nonumber \\
&=& \frac{M}{k^2}{}\expect[(1_{\mb{Z} \in S_u(X)}-U(X))(1_{\mb{Z} \in S_u(Y)}-U(Y))] \nonumber \\
&=& \frac{M}{k^2}{}\expect[1_{\mb{Z} \in S_u(X)}1_{\mb{Z} \in S_u(Y)}-U(X)U(Y)] \nonumber \\
&=& \frac{M}{k^2}{}(U(X,Y) - U(X)U(Y)) \nonumber \\
&=& -\frac{M}{k^2}{}[U(X)U(Y)] = \frac{-f(X)f(Y)}{M} + o\left(\frac{1}{M}\right). \nonumber
\end{eqnarray}

\end{proof}

\paragraph{Disjoint balls}

For $\{X,Y\} \in \Psi_u^c$, there is no closed form expression for the covariance. However we have the following lemmas:

Let $R_u(X)$ and $R_u(Y)$ denote the (constant and equal) radii of the uniform balls respectively. Define $\aleph(||X-Y||/R_u(X)) = V(S_u(X) \cap S_u(Y))/V_u(X)$ where $V(S_u(X) \cap S_u(Y))$ is the volume of the intersection of the two balls.

We observe that,

\begin{eqnarray}
\aleph(||X-Y||/R_u(X)) &=& V(S_u(X) \cap S_u(Y))/V_u(X) \nonumber \\
&=& \frac{V[1_{\mb{Z} \in B(0,R_u(X))}1_{\mb{Z} \in B(||Y-X||,R_u(Y))}]}{V_u(X)} \nonumber \\
&=& \frac{V[1_{\mb{Z} \in B(0,1)}1_{\mb{Z} \in B(||Y-X||/R_u(X),1)}]}{V[1_{\mb{Z} \in B(0,1)}]} \nonumber \\
&=& O(1). 
\end{eqnarray}

Because $f$ is assumed to be continuous, we have 

\begin{equation}
U(X,Y) = \expect[1_{\mb{Z} \in S_u(X)}1_{\mb{Z} \in S_u(Y)}] = [f(X)+o(1)]V(S_u(X) \cap S_u(Y)).
\end{equation}


\begin{lemma}
\label{disjointu}
For a fixed pair of points $\{X,Y\}\in {\Psi_u}^c$, 

\begin{equation}
Cov[\mb{e_u}(X),\mb{e_u}(Y)] = O(1/k). \nonumber
\end{equation}

\end{lemma}

\begin{proof}
\begin{eqnarray}
\frac{M}{k^2}U(X,Y) &=& \frac{M}{k^2}[f(X)+o(1)]V(S_u(X) \cap S_u(Y))  \nonumber \\
&=& \frac{f(X)+o(1)}{k}\frac{V(B_X \cap B_Y)}{V_u(X)} \nonumber \\
&=& \frac{f(X)+o(1)}{k} \aleph(||X-Y||/R_u(X)) \nonumber \\
&=& \frac{f(X)}{k} \aleph(||X-Y||/R_u(X)) + o(1/k) \nonumber \\
&=& O(1/k). \nonumber
\end{eqnarray}

Therefore,

\begin{eqnarray}
Cov[\mb{e_u}(X),\mb{e_u}(Y)] &=& \expect[(\hat{\mb{f}}_\mb{u}(X)-\expect[\hat{\mb{f}}_\mb{u}(X)])(\hat{\mb{f}}_\mb{u}(Y)-\expect[\hat{\mb{f}}_\mb{u}(Y)])]\nonumber \\
&=& \frac{M}{k^2}{}(U(X,Y) - U(X)U(Y)) \nonumber \\
&=& \frac{M}{k^2}U(X,Y) - \frac{M}{k^2}U(X)U(Y) \nonumber \\
&=& O(1/k) - \Theta(1/M) \nonumber \\
&=& O(1/k). \nonumber
\end{eqnarray}

\end{proof}

\begin{lemma}

\begin{equation}
\int_y U(X,y) dy = [f(X)+o(1)]V_u(X)^2. \nonumber
\end{equation}

\end{lemma}

\begin{proof}

We note that for $U(X,y) \neq 0$, we need $\{X,y\} \in \Psi_u^c$. We therefore have,
$f(y) = f(X) + o(1)$.

\begin{eqnarray}
\int_y U(X,y) dy &=& \int [f(X)+o(1)]V(S_u(X) \cap S_u(Y)) dy \nonumber \\
&=& V_u(X)[f(X)+o(1)] \int \aleph(||X-y||/R_u(X)) dy \nonumber \\
&=& V_u(X)[f(X)+o(1)] R_u(X)^d \int \aleph(||y||/R_u(X))  d(y/R_u(X))  \nonumber \\
&=& V_u(X)[f(X)+o(1)] \frac{V_u(X)}{c_d} \int \aleph(||y||/R_u(X))  d(y/R_u(X)) \nonumber \\
&=& [f(X)+o(1)] \frac{V_u^2(X)}{c_d} \int \aleph(\delta)  d(\delta). \nonumber 
\end{eqnarray}

The integral $\int \aleph(\delta)  d(\delta)$ can be shown to be equal to $c_d$ for all dimensions $d$.

We then have,

\begin{eqnarray}
\int_y U(X,y) dy &=& [f(X)+o(1)] {V_u^2(X)} \nonumber \\
&=& [f(X)+o(1)] \left(\frac{k}{M}\right)^2. \nonumber 
\end{eqnarray}

\end{proof}

\begin{lemma}
\label{covariancelemma}
Let $\gamma_1(X)$, $\gamma_2(X)$ be arbitrary continuous functions. Let $\mb{X}_{1},..,\mb{X}_{M},\mb{X},\mb{Y}$ denote $M+2$ i.i.d realizations of the density $f$. Then,
\begin{equation}
Cov{\left[\gamma_1(\mb{X})\mb{e_u}(\mb{X}),\gamma_2(\mb{Y})\mb{e_u}(\mb{Y})\right]} = \frac{Cov[\gamma_1(\mb{X})f(\mb{X}),\gamma_2(\mb{X})f(\mb{X})]}{M} + o(1/M). \nonumber
\end{equation}

\end{lemma}

\begin{proof}
\begin{eqnarray}
Cov{\left[\gamma_1(\mb{X})\mb{e_u}(\mb{X}),\gamma_2(\mb{Y})\mb{e_u}(\mb{Y})\right]} 
&=& \expect{\left[\gamma_1(\mb{X})\gamma_2(\mb{Y})(\hat{\mb{f}}_\mb{u}(X)-\expect[\hat{\mb{f}}_\mb{u}(X)])(\hat{\mb{f}}_\mb{u}(Y)-\expect[\hat{\mb{f}}_\mb{u}(Y)])\right]} \nonumber \\
&=&\frac{1}{MV_u(X)V_u(Y)}\expect{\left[\gamma_1(\mb{X})\gamma_2(\mb{Y})(U(\mb{X},\mb{Y})-U(\mb{X})U(\mb{Y}))\right]} \nonumber \\
&=&\frac{1}{MV_u^2(X)}\expect{\left[\gamma_1(\mb{X})\gamma_2(\mb{Y})U(\mb{X},\mb{Y})\right]} \nonumber \\
&-&\frac{1}{MV_u^2(X)}\expect{\left[\gamma_1(\mb{X})\gamma_2(\mb{Y})U(\mb{X})U(\mb{Y})\right]} \nonumber \\
&=& I - II. \nonumber
\end{eqnarray}

\begin{eqnarray}
II = \frac{1}{M}\left(\expect[\gamma_1(\mb{X})f(\mb{X})]\expect[\gamma_2(\mb{Y})f(\mb{Y})]\right). \nonumber
\end{eqnarray}

\begin{eqnarray}
I &=& \frac{1}{MV_u^2(X)}\expect{\left[\gamma_1(\mb{X})\gamma_2(\mb{Y})U(\mb{X},\mb{Y})\right]} \nonumber \\
&=& \frac{1}{MV_u^2(X)}\int \int \gamma_1(x)\gamma_2(y) f(x) f(y) U(x,y) dx dy. \nonumber
\end{eqnarray}

Now for $U(x,y) \neq 0$, we need $\{x,y\} \in \Psi_u^c$. We therefore have,
$\gamma_2(y)f(y) = \gamma_2(x)f(x) + o(1)$.

We then have,
\begin{eqnarray}
I &=& \frac{1}{MV_u^2(X)}\int \int [\gamma_1(x)\gamma_2(x)f^2(x)+o(1)]  U(x,y) dx dy \nonumber \\
&=& \frac{1}{MV_u^2(X)}\int [\gamma_1(x)\gamma_2(x)f^2(x)+o(1)] \left(\int U(x,y) dy\right) dx  \nonumber \\
&=& \frac{1}{MV_u^2(X)}\int [\gamma_1(x)\gamma_2(x)f^2(x)+o(1)] \left((f(x)+o(1))V_u(x)^2\right) dx  \nonumber \\
&=& \frac{1}{M}\int [\gamma_1(x)\gamma_2(x)f^2(x)+o(1)] (f(x)+o(1)) dx \nonumber \\
&=& \frac{1}{M}\left(\expect[\gamma_1(\mb{X})\gamma_2(\mb{X})f^2(\mb{X})]+o(1)\right) \nonumber \\
&=& \frac{1}{M}\expect[\gamma_1(\mb{X})\gamma_2(\mb{X})f^2(\mb{X})] + o(1/M). \nonumber
\end{eqnarray}

\end{proof}

\subsection{Higher cross moments}
\paragraph{Disjoint balls}

We have the following results concerning higher cross moments for disjoint balls:

\begin{lemma} Let $q$,$r$ be positive integers satisfying $q+r > 2$. For a fixed pair of points $\{X,Y\}\in {\Psi_u}^c$,

\begin{eqnarray}
\label{highercrossmoments}
Cov(\mb{e}^q_\mb{u}(X),\mb{e}^r_{u}(Y)) &=& o(1/M). \nonumber
\end{eqnarray}

\end{lemma}

\begin{proof}
For a fixed pair of points $\{X,Y\}\in {\Psi_u}^c$, the joint probability mass function of the functions $\mb{l_u}(X)$,$\mb{l_u}(Y)$ is given by

\begin{equation}
Pr(\mb{l_u}(X)=l_x,\mb{l_u}(Y)=l_y) = 1_{l_x+l_y \leq M}\binom{M}{l_x,l_y} (U(X))^{l_x}(U(Y))^{l_y}(1-U(X)-U(Y))^{M-l_x-l_y}. \nonumber
\end{equation}
%
%
We also have from chernoff inequalities for binomial random variables that 
\begin{eqnarray}
 Pr((1-p)k<\mb{l_u}(X)<(1+p)k) = 1-e^{-p^2k}, \nonumber \\
 Pr((1-p)k<\mb{l_u}(Y)<(1+p)k) = 1-e^{-p^2k}. \nonumber
\end{eqnarray}
Denote the high probability event $\chi$ by ${(1-p)k<\mb{l_u}(X),\mb{l_u}(Y)<(1+p)k}$. Define $\mb{\hat{l}}_\mb{u}(X)$, $\mb{\hat{l}}_\mb{u}(Y)$ to be binomial random variables with parameters $\{U(X)$,$M-q\}$ and $\{U(Y)$,$M-r\}$ respectively. The covariance between powers of density estimates is then given by 
\begin{eqnarray}
&& Cov(\hat{\mb{f}}_\mb{u}^q(X),\hat{\mb{f}}_\mb{u}^r(Y)) = \frac{1}{k^{q+r}}Cov(\mb{l}_\mb{u}^q(X),\mb{l}_\mb{u}^r(Y)) \nonumber \\
&=& \frac{1}{k^{q+r}} \sum l_x^ql_y^r Pr(\mb{l_u}(X)=l_x,\mb{l_u}(Y)=l_y) - \frac{1}{k^{q+r}} \sum l_x^ql_y^r Pr(\mb{l_u}(X)=l_x)Pr(\mb{l_u}(Y)=l_y) \nonumber \\
&=& \sum_\chi \frac{l_x^ql_y^r}{k^{q+r}} \left[Pr(\mb{l_u}(X)=l_x,\mb{l_u}(Y)=l_y) - Pr(\mb{l_u}(X)=l_x)Pr(\mb{l_u}(Y)=l_y)\right] + O(e^{-p^2k}) \nonumber \\
&=& \sum_\chi \frac{f^q(X)f^r(Y)l_x^ql_y^rU^q(X)U^r(Y)}{k^{q+r}(l_x \times \ldots \times l_x-{q+1})(l_y \times \ldots \times l_y-{r+1})} \times \nonumber \\
& & [(M \times \ldots \times M-(q+r-1))Pr(\mb{\hat{l}_u}(X)=l_x,\mb{\hat{l}_u}(Y)=l_y) \nonumber \\
&& - (M \times \ldots \times M-q+1)(M \times \ldots \times M-r+1) Pr(\mb{\hat{l}_u}(X)=l_x)Pr(\mb{\hat{l}_u}(Y)=l_y)] \nonumber \\
&+& o(1/M) \nonumber \\
&=&\left(\frac{f^q(X)f^r(Y)}{M^{q+r}}+O\left(\frac{1}{kM^{q+r}}\right)\right)  \times \nonumber \\
&& \sum_\chi [(M \times \ldots \times M-(q+r-1)) Pr(\mb{\hat{l}_u}(X)=l_x,\mb{\hat{l}_u}(Y)=l_y) \nonumber \\
&& - (M \times \ldots \times M-(q-1))(M \times \ldots \times M-(r-1)) Pr(\mb{\hat{l}_u}(X)=l_x)Pr(\mb{\hat{l}_u}(Y)=l_y)] \nonumber \\
&+& o(1/M) \nonumber 
\end{eqnarray}
\begin{eqnarray}
&=&\left(\frac{f^q(X)f^r(Y)}{M^{q+r}}+O\left(\frac{1}{kM^{q+r}}\right)\right)  \times \nonumber \\
&& [(M \times \ldots \times M-(q+r-1)) - (M \times \ldots \times M-(q-1))(M \times \ldots \times M-(r-1)) ] \nonumber \\
&+& o(1/M) \nonumber \\
&=&\frac{-qrf^q(X)f^r(Y)}{M}+o\left(\frac{1}{M}\right). \nonumber
\end{eqnarray}

Then, the covariance between the powers of the error function is given by 
\begin{eqnarray}
Cov(\mb{e}^q_\mb{u}(X),\mb{e}^r_{u}(Y)) &=& Cov((\hat{\mb{f}}_\mb{u}(X)-\expect[\hat{\mb{f}}_\mb{u}(X)])^q,(\hat{\mb{f}}_\mb{u}(Y)-\expect[\hat{\mb{f}}_\mb{u}(Y)])^r) \nonumber \\
&=& \sum_{a=1}^{q} \sum_{b=1}^{r} \binom{q}{a} \binom{r}{b} (-\expect[\hat{\mb{f}}_\mb{u}(X)])^a(-\expect[\hat{\mb{f}}_\mb{u}(Y)])^bCov(\hat{\mb{f}}_\mb{u}^a(X),\hat{\mb{f}}_\mb{u}^b(Y)) \nonumber \\
&=& \sum_{a=1}^{q} \sum_{b=1}^{r} \binom{q}{a} \binom{r}{b} [(-f(X))^a(-f(Y))^b+o(1)] Cov(\hat{\mb{f}}_\mb{u}^a(X),\hat{\mb{f}}_\mb{u}^b(Y)) \nonumber \\
&=& -f^{q}(X)f^{r}(Y) \sum_{a=1}^{q} \sum_{b=1}^{r} \binom{q}{a} \binom{r}{b}  \frac{(-1)^aa(-1)^bb}{M}+o\left(\frac{1}{M}\right) \nonumber \\
&=& 1_{\{q=1,r=1\}}\left(\frac{-f(X)f(Y)}{M}\right) + o(1/M) \nonumber \\
&=& o(1/M). \nonumber
\end{eqnarray}
where the last step follows from the condition that $q+r>2$.

\end{proof}

\paragraph{Intersecting balls}

For $\{X,Y\}\in {\Psi_u}^c$, we have the following bounds

\begin{lemma}
\label{covariancelemma2}
Let $\gamma_1(X)$, $\gamma_2(X)$ be arbitrary continuous functions. Let $\mb{X}_{1},..,\mb{X}_{M},\mb{X},\mb{Y}$ denote $M+2$ i.i.d realizations of the density $f$. Also let the indicator function ${1_{\Delta_u}}(X,Y)$ denote the event ${\Delta_u}: \{{X},{Y}\}\in {\Psi_u}^c$. For $q$,$r$ positive integers satisfying $q+r > 1$,

\begin{eqnarray}
\expect{\left[\mb{1_{\Delta_u}}(\mb{X},\mb{Y}) \gamma_1(\mb{X})\gamma_2(\mb{Y}) \mb{e}^q_\mb{u}(\mb{X})\mb{e}^r_\mb{u}(\mb{Y})\right]} &=& o\left(\frac{1}{M}\right), \nonumber \\
\end{eqnarray}

\end{lemma}

\begin{proof}
\label{highercrossepsproof}
For $1_{\Delta_u}({X},{Y}) \neq 0$, we have $\{X,Y\} \in \Psi_u^c$. Then,
\begin{eqnarray}
&&\expect{\left[\mb{1_{\Delta_u}}(\mb{X},\mb{Y}) \gamma_1(\mb{X})\gamma_2(\mb{Y}) \mb{e}^q_\mb{u}(\mb{X})\mb{e}^r_\mb{u}(\mb{Y})\right]}  \nonumber \\
&=& \expect{\left[\mb{1_{\Delta_u}}(\mb{X},\mb{Y}) \gamma_1(\mb{X})\gamma_2(\mb{Y})\expect_{\mb{X},\mb{Y}}[\mb{e}^q_\mb{u}({X})\mb{e}^r_\mb{u}({Y})]\right]} \nonumber \\
&\leq& \expect{\left[\mb{1_{\Delta_u}}(\mb{X},\mb{Y}) \gamma_1(\mb{X})\gamma_2(\mb{Y})\sqrt{\expect_{\mb{X}}[\mb{e}^{2q}_\mb{u}({X})]\expect_{\mb{Y}}[\mb{e}^{2r}_\mb{u}({Y})]}\right]} \nonumber \\
&=& \expect{\left[\mb{1_{\Delta_u}}(\mb{X},\mb{Y}) \gamma_1(\mb{X})\gamma_2(\mb{Y})O\left(\frac{1}{k^{q+r/2}}\right)\right]} \nonumber \\
&=& \int{\left[O\left(\frac{1}{k^{q+r/2}}\right)(\gamma_1(x)\gamma_2(x) + o(1))\right] \left(\int {\Delta_u}({x},{y}) dy \right)} dx  \nonumber \\
&=& \int{\left[O\left(\frac{1}{k^{q+r/2}}\right)(\gamma_1(x)\gamma_2(x) + o(1))\right] \left(2^{d}\frac{k}{M} \right)} dx  \nonumber \\
&=& o\left(\frac{1}{M}\right). \nonumber
\end{eqnarray}
where the bound is obtained using the Cauchy-Schwarz inequality and using Eq.\ref{fourmom}.
\end{proof}

We can succinctly state the results derived in the last two lemmas in the form of the following lemma:

\begin{lemma}
\label{hcm}
Let $\gamma_1(X)$, $\gamma_2(X)$ be arbitrary continuous functions. Let $\mb{X}_{1},..,\mb{X}_{M},\mb{X},\mb{Y}$ denote $M+2$ i.i.d realizations of the density $f$. If $q$,$r$ are positive integers satisfying $q+r>2$

\begin{eqnarray}
Cov{\left[\gamma_1(\mb{X})\mb{e}^q_\mb{u}(\mb{X}), \gamma_2(\mb{Y})\mb{e}^r_\mb{u}(\mb{Y})\right]} &=& o(1/M). \nonumber
\end{eqnarray}

\end{lemma}

\begin{proof}

The result for the case $q=1$, $r=1$ was established earlier in Lemma \ref{covariancelemma}.
\begin{eqnarray}
Cov{\left[\gamma_1(\mb{X})\mb{e}^q_\mb{u}(\mb{X}), \gamma_2(\mb{Y})\mb{e}^r_\mb{u}(\mb{Y})\right]} = I + D, \nonumber
\end{eqnarray}
where '$I$' stands for the contribution form the intersecting balls and '$D$' for the contribution from the dis-joint balls. $I$ and $D$ are given by
\begin{eqnarray}
I &=& \expect{\left[\mb{1_{\Delta_u}}(\mb{X},\mb{Y}) Cov \left[\gamma_1({X})\mb{e}^q_\mb{u}({X}), \gamma_2({Y}) \mb{e}^r_\mb{u}({Y}) \right]\right]},\nonumber \\
D &=& \expect{\left[\mb{(1-\mb{1_{\Delta_u}}(\mb{X},\mb{Y}))} Cov \left[\gamma_1({X})\mb{e}^q_\mb{u}({X}), \gamma_2({Y}) \mb{e}^r_\mb{u}({Y}) \right]\right]}. \nonumber 
\end{eqnarray}

We have already established in the previous lemma that 
\begin{equation}
I = o\left(\frac{1}{M}\right). \nonumber
\end{equation}
Now,
\begin{eqnarray}
D &=&  \expect{\left[(1-\mb{1_{\Delta_u}}(\mb{X},\mb{Y})) \gamma_1(\mb{X})\gamma_2(\mb{Y}) \expect_{\mb{X},\mb{Y}}[Cov(\mb{e}^q_\mb{u}({X}),\mb{e}^r_\mb{u}({Y}))]\right]} \\ \nonumber 
&=&  \expect{\left[(1-\mb{1_{\Delta_u}}(\mb{X},\mb{Y})) \gamma_1(\mb{X})\gamma_2(\mb{Y}) o(1/M) \right]} \nonumber \\
&=&  o\left(\frac{1}{M}\right). \nonumber
\end{eqnarray}
This concludes the proof.
\end{proof}

\section{$k$-NN density estimation}
In this appendix, moment properties of the standard $k$-NN density estimate $\hat{\mb{f}}_{k}(X)$ are derived conditioned on $X_1, \ldots, X_N$. As the samples $X_1, \ldots, X_N, X_{N+1}, \ldots, X_T$, $T=M+N$ are i.i.d., these conditional moments are independent of the $N$ samples $\mb{X}_{1},..,\mb{X}_{N}$.

\subsection{Preliminaries}
Let $d(X,Y)$ denote the Euclidean distance between points $X$ and $Y$ and ${\mb{d}^{(k)}_X}$ denote the Euclidean distance between a point X and its $k$-th nearest neighbor amongst $\mb{X}_{N+1},..,\mb{X}_{N+M}$. Let $c_d$ denote the unit ball volume in $d$ dimensions. The $k$-NN region is  $$\mb{S}_k(X) = \{Y:d(X,Y)\leq{\mb{d}{^{(k)}_X}}\}$$ and the volume of the $k$-NN region is $$\mb{\mb{V}}_k(X) = \int_{\mb{S}_k(X)}{dZ}.$$ The standard {$k$-NN} density estimator~\cite{quu} is defined as $$\hat{\mb{f}}_{k}(X) = \frac{k-1}{M\mb{\mb{V}}_k(X)}.$$ Define the coverage function as $$\mb{P}(X) = \int_{\mb{S}_k(X)} f(Z) dZ.$$ Define spherical regions $$S_r(X) =  \{Y \in {\mathbb{R}^d}:d(X,Y) \leq r\}.$$ 

\subsection{Concentration inequality for coverage probability}
\label{concH}

It has been previously established that $\mb{P}(X)$ has a beta distribution with parameters $k$, $M-k+1$. \cite{fuk}. 
Consider a binomial random variable with parameters $M$ and $P$ with distribution function $Bi(.|M,P)$ and a beta random variable with parameters $k$ and $M-k+1$ with distribution function $Be(.|k,M-k+1)$. We have the following identity,
\begin{equation}
Be(P|k,M-k+1) = 1-Bi(k-1|M,P).
\end{equation}

The following Chernoff bounds for binomial random variables have also been established previously.
When $k<MP$, $Bi(k|M,P) \leq exp{\left[-{(MP-k)^2}/{2PM}\right]}$, and when $k>MP$, $1-Bi(k|M,P) \leq exp{\left[-{(MP-k)^2}/{2PM}\right]}$. We therefore have that for some $0<p<1/2$,
\begin{equation}
\label{concincc}
Pr((1-p)(k-1)/M<\mb{P}(X)<(p+1)(k-1)/M) = O(e^{-p^2k/2}).
\end{equation}

Define $$k_M = (k-1)/M.$$ Let $\natural(X)$ denote the event 
\begin{equation} \mb{P}(X)<(p_k+1)k_M, 
\label{boundonP} \end{equation}
where $p_k = {\sqrt{6}}/(k^{\delta/2})$. Then, $1-Pr(\natural(X)) = O(e^{-p_k^2k/2}) = O(e^{-3k^{(1-\delta)}})$. Equivalently, 
\begin{equation}
1-Pr(\natural(X)) = O({\cal C}(k)),
\label{eq:concincexplicit}
\end{equation}
where ${\cal C}(k)$ is a function which satisfies the rate of decay condition ${\cal C}(k)= O(e^{-3k^{(1-\delta)}})$. Similarly, let $\natural_{-1}(X)$ denote the event 
\begin{equation} \mb{P}(X)>(1-p_k)k_M, 
\label{boundonPinv} \end{equation}
Then 
\begin{equation}
1-Pr(\natural_{-1}(X)) = O({\cal C}(k)),
\label{eq:concincexplicitinv}
\end{equation}
Also let $\natural\natural(X) = \natural(X) \cap \natural_{-1}(X)$. Then 
\begin{equation}
1-Pr(\natural\natural(X)) = O({\cal C}(k)),
\label{eq:concincexplicitjoint}
\end{equation}
Finally, we note that $\Gamma(x+a)/\Gamma(x) = x^a+o(x^a)$. Then for any $a<k$, $\expect[\mb{P}^{-a}(X)]$ exists and is given by
\begin{equation}
\expect[\mb{P}^{-a}(X)] = \frac{\Gamma(k-a)\Gamma(M+1)}{\Gamma(k)\Gamma(M+1-a)} = \Theta((k_M)^{-a}).
\label{eq:existence1}
\end{equation}

\subsubsection{Interior points}
\label{sec:intpnt}
Let ${\cal S'}$ to be any arbitrary subset of ${\cal S}_I$ (\ref{sbdefine}) satisfying the condition $Pr(\mb{Y} \notin {\cal S'}) = o(1)$ where ${\mb{Y}}$ is random variable with density $f$. This implies that given the event $\natural(X)$, the $k$-NN neighborhoods $\mb{S}_k(X)$ of points $X \in {\cal S'}$ will lie completely inside the domain ${\cal S}$. Therefore the density $f$ has continuous partial derivatives of order $2\nu$ in the $k$-NN ball neighborhood $\mb{S}_k(X)$ for each $X \in {\cal S'}$ (assumption $({\cal {A}}.2)$). We will now derive moments for the interior set of points $X \in {\cal S'}$. This excludes the set of points $X$ close to the boundary of the support whose $k$-NN neighborhoods $\mb{S}_k(X)$ intersect with the boundary of the support. We will deal with these points in Appendix B.

\subsubsection{Taylor series expansion of coverage probability}
Let $X \in {\cal S'}$. Given the event $\natural(X)$, the coverage function $\mb{P}(X)$ can be represented in terms of the volume of the $k$-NN ball $\mb{\mb{V}}_k(X)$ by expanding the density $f$ in a Taylor series about $X$ as follows. In particular, for some fixed $x \in {\cal S'}$, let $$p(u) = \int_{S_u(x)} f(z) dz.$$ Using $({\cal {A}}.2)$, we can write, by a Taylor series expansion of $f$ around $x$ using multi-index notation~\cite{multi}
\begin{eqnarray}
f(z) = \sum_{0 \leq |\alpha| \leq 2\nu} \frac{(z-x)^\alpha}{\alpha!} (\partial^\alpha f)(x) + o(||z-x||^{2\nu})
\end{eqnarray}
Assuming $S_u(x) \subset {\cal S}$, we can then write
\begin{eqnarray}
p(u) &=& \int_{S_u(x)} f(z) dz \nonumber \\
&=& \int_{S_u(x)} \left(\sum_{|0 \leq \alpha \leq 2\nu|} \frac{(z-x)^{\alpha}}{\alpha!} (\partial^\alpha f)(x) \right) dz + o(u^{d+2\nu}) \nonumber \\
&=& f(x)c_du^d + \sum_{i=1}^{\nu-1} c_i(x)c_d^{1+2i/d}u^{d+2i} + o(u^{d+2\nu}).
\label{eq:taylorP}
\end{eqnarray}
where $c_i(x)$ are functionals of the derivatives of $f$. Now, denote $v(u)= \int_{S_u(x)} dz$ to be the volume of $S_u(x)$. Let $u^{inv}(v)$ be the inverse function of $v(u)$. Note that this inverse is well-defined since $v(u)$ is monotonic in $u$. Since $S_u(x) \subset {\cal S}$, $v(u) = c_du^d$. This gives $u^{inv}(v) = (v/c_d)^{1/d}$. Define $$P(v) = \int_{S_{u^{inv}(v)}(x)} f(z) dz.$$ Using (\ref{eq:taylorP}), 
\begin{eqnarray}
P(v) &=& f(X)v + \sum_{i=1}^{\nu-1} c_i(X)v^{1+2i/d} + o(v^{1+2\nu/d}).
\label{eq:taylorPV}
\end{eqnarray}
Now denote $V(p) = P^{inv}(p)$ to be the inverse of $P(.)$. Note that this inverse is well-defined since $P(v)$ is monotonic in $v$. Dividing (\ref{eq:taylorPV}) by $vP(v)$ on both sides, we get
\begin{eqnarray}
\frac{1}{v} &=& \frac{f(X)}{P(v)} + \sum_{i=1}^{\nu-1} \frac{c_i(X)}{P(v)}v^{2i/d} + o(v^{2\nu/d}P^{-1}(v))
\label{incompleteeqn}
\end{eqnarray}
By repeatedly substituting the LHS of (\ref{incompleteeqn}) in the RHS of (\ref{incompleteeqn}), we can obtain (\ref{eq:taylorVP}):
\begin{eqnarray}
\frac{1}{V(p)} &=& \frac{f(X)}{p} + \sum_{i=1}^{\nu-1} \frac{{h}_i(X)}{p^{1-2i/d}} + o(p^{2\nu/d-1}), 
\label{eq:taylorVP}
\end{eqnarray}
From our derivation of (\ref{eq:taylorVP}) using (\ref{eq:taylorPV}), it is clear that ${h}_i(X)$ are of the form $${h}_i(X)  = \sum_{\{a_i\} = A; A \in {\cal A}} \frac{\prod_{i=1}^{\nu-1} c_i^{a_i}}{f^{a_0}(X)}$$ where $A$ is a $\nu$-tuple of positive real numbers ${a_0,..,a_{\nu-1}}$ and the cardinality of ${\cal A}$ is finite. By assumptions $({\cal {A}}.1)$ and $({\cal {A}}.2)$, this implies that the constants $h_i(X)$ are {\emph {bounded}}. Also, we note that $h(X) = h_1(X) = c(X)f^{-2/d}(X)$~\cite{fuk2}, where $c(X) := c_1(X)=\Gamma^{(2/d)}(\frac{d+2}{2})tr[\nabla^2(f(X))]$. This then implies that under the event $\natural(X)$
\begin{align}
\label{knnaprr}
\frac{1}{\mb{V}_k(X)} &= \frac{f(X)}{\mb{P}(X)} +  \sum_{t \in {\cal T}} \frac{h_t(X)}{\mb{P}^{1-t}(X)} + \mb{h_r}(X), \end{align}
where ${\cal T} = \{2/d,4/d,6/d..,2\nu/d\}$ and $\mb{h_r}(X) = o(\mb{P}^{2\nu/d-1}(X))$. Now, by $({\cal {A}}.2)$, we have $(k/M)^{2\nu/d} = o(1/M)$. This implies that ${2\nu/d} >1$. Under the event $\natural(X)$, we have $\mb{P}(X) \leq (p_k+1)k/M$, which, in conjunction with the condition ${2\nu/d} >1$ implies that 
\begin{eqnarray}
\mb{h_r}(X) &=& o(\mb{P}^{2\nu/d-1}(X)) = o((k/M)^{2\nu/d-1}) = o(1/k_MM).
\label{eq:hrapprox}
\end{eqnarray}
On the other hand, under the event, $\natural^c(X)$, $(p_k+1)k/M \leq \mb{P}(X) \leq 1$, which gives
\begin{eqnarray}
\mb{h_r}(X) &=& O(1).
\label{eq:hrapprox2}
\end{eqnarray}
\subsubsection{Approximation to the {$k$}-NN  density estimator}
Define the \emph{coverage} density estimate to be,
\begin{equation}
\hat{\mb{f}}_{c}(X) = f(X)\frac{k-1}{M}\frac{1}{\mb{P}(X)}. \nonumber
\end{equation}
The estimate $\hat{\mb{f}}_{c}(X)$ is clearly not implementable. Note also that the two estimates - $\hat{\mb{f}}_{c}(X)$ and $\hat{\mb{f}}_{k}(X)$ - are identical in the case of the uniform density. 
\begin{align}
\frac{1}{\mb{V}_k(X)} &= \frac{f(X)}{\mb{P}(X)} + \frac{h(X)}{\mb{P}^{1-2/d}(X)} + \mb{h_s}(X), 
\end{align}
where $\mb{h_s}(X) = o(1/\mb{P}^{1-2/d}(X))$. This gives, 
\begin{eqnarray}
 \hat{\mb{f}}_{k}(X) &=& \hat{\mb{f}}_{c}(X) + \left(\frac{k-1}{M} \right)\frac{h(X)}{\mb{P}^{1-2/d}(X)} + \frac{k-1}{M}\mb{h_s}(X). 
 \label{densitycoverageapprox}
\end{eqnarray} whenever $\natural(X)$ is true.

\subsubsection{Bounds on $k$-NN density estimates}

Let $X$ be a Lebesgue point of $f$, i.e., an $X$ for which $$\lim_{r \to 0} \frac{\int_{S_r(X)} f(y) dy}{\int_{S_r(x)} dy}  = f(X).$$ Because $f$ is an density, we know that almost all $X \in {\cal S}$ satisfy the above property. Now, fix $\epsilon \in (0,1)$ and find $\delta > 0$ such that 
$$\sup_{0<r\leq \delta} \frac{\int_{S_r(X)} f(y) dy}{\int_{S_r(x)} dy} - f(X) \leq \epsilon f(X).$$ 
This in turn implies that, for $\mb{P}(X) \leq P(\delta)$, 
\begin{eqnarray}
\frac{\mb{P}(X)}{(1+\epsilon)f(X)} \leq \mb{V}_k(X) \leq \frac{\mb{P}(X)}{(1-\epsilon)f(X)}
\label{coverageineqnatural}
\label{coverageineq}
\end{eqnarray}
and in turn implies
\begin{eqnarray}
(1-\epsilon) \hat{\mb{f}}_{c}(X) \leq &\hat{\mb{f}}_{k}(X)& \leq (1+\epsilon)\hat{\mb{f}}_{c}(X).
\label{densityineqnatural}
\end{eqnarray}
Also, because $\delta > 0$ is fixed, we note that the event $\mb{P}(X) \leq P(\delta)$ is a subset of $\natural(X)$ and therefore (\ref{coverageineq}) holds under $\natural(X)$. 

Under the event $\natural^c(X)$, we can bound $\mb{V}_k(X)$ from above by $c_d{\cal D}^d$. Also, since $\mb{V}_k(X)$ is monotone in $\mb{P}(X)$, under the event $\natural^c(X)$, we can bound $\mb{V}_k(X)$ from below by ${(1+p_k)(k-1)}/{M(1-\epsilon)f(X)}$ and therefore by ${(k-1)}/{M(1-\epsilon)f(X)}$. Written explicitly, 
\begin{eqnarray}
\frac{(k-1)}{M(1-\epsilon)f(X)} \leq \mb{V}_k(X) \leq c_d{\cal D}^d
\label{coverageineqnaturalc}
\end{eqnarray}
and in turn implies
\begin{eqnarray}
(k-1)/(Mc_d{\cal D}^d) \leq &\hat{\mb{f}}_{k}(X)& \leq (1-\epsilon)f(X).
\label{densityineqnaturalc}
\end{eqnarray}
Finally, note that $k_M/\mb{P}(X)$ is bounded above by $O(1)$ under the event $\natural(X)$. This implies that for any $a<k$,
\begin{eqnarray}
\expect[\natural^c(X)]k^a_M\mb{P}^{-a}(X) \leq O(1)Pr(\natural^c(X)) = O({\cal C}(k)).
\label{eq:existence2}
\end{eqnarray}


\subsection{Bias of the {$k$}-NN  density estimates} 
\label{sec:bias}
Let $X \in {\cal S'}$. We can analyze the bias of $k$-NN density estimates as follows by using (\ref{densitycoverageapprox})
\begin{eqnarray}
\expect[1_{\natural(X)}\hat{\mb{f}}_{k}(X)] &=& \expect[1_{\natural(X)}\hat{\mb{f}}_{c}(X)] + \expect\left[1_{\natural(X)}\left(\frac{k-1}{M} \right)\frac{h(X)}{\mb{P}^{1-2/d}(X)}\right] + \expect\left[1_{\natural(X)}\frac{k-1}{M} \mb{h_s}(X) \right]  \nonumber \\
&=& \expect[1_{\natural(X)}\hat{\mb{f}}_{c}(X)] + \expect\left[1_{\natural(X)}\left(\frac{k-1}{M} \right)\frac{h(X)}{\mb{P}^{1-2/d}(X)}\right] + o\left(\expect\left[1_{\natural(X)}\frac{k-1}{M} \mb{P}^{2/d-1}(X) \right]\right)  \nonumber \\
&=& \expect[\hat{\mb{f}}_{c}(X)] + \expect\left[\left(\frac{k-1}{M} \right)\frac{h(X)}{\mb{P}^{1-2/d}(X)}\right] + o\left(\frac{k}{M}\right)^{2/d} + O({\cal C}(k)) \nonumber \\
&=& f(X) + h(X)\left(\frac{k}{M} \right)^{2/d} + o\left(\frac{k}{M}\right)^{2/d},
\end{eqnarray}
where we used the fact that under the event $\natural^c(X)$, $((k-1)/M)\mb{P}^{1-t}(X) = O(1)$ for any $t>=0$, which in turn gives $\expect[1_{\natural^c(X)} ((k-1)/M)\mb{P}^{1-t}(X)] = O(Pr(\natural^c(X))) = O({\cal C}(k))$. This implies that
\begin{eqnarray}
\expect[\hat{\mb{f}}_{k}(X)] - f(X) &=& \expect[1_{\natural(X)}\hat{\mb{f}}_{k}(X)] + \expect[1_{\natural^c(X)}\hat{\mb{f}}_{k}(X)]- f(X) \nonumber \\
&=& h(X)\left(\frac{k}{M} \right)^{2/d} + o\left(\frac{k}{M}\right)^{2/d} + O({\cal C}(k)) + \expect[1_{\natural^c(X)}\hat{\mb{f}}_{k}(X)] \nonumber \\
&=& h(X)\left(\frac{k}{M} \right)^{2/d} + o\left(\frac{k}{M}\right)^{2/d} + O({\cal C}(k)),
\label{inbias}
\end{eqnarray}
where the last step follows because , by (\ref{densityineqnaturalc}), $1_{\natural^c(X)}\hat{\mb{f}}_{k}(X) = O(1)$. This expression is true for $k>=3$ by (\ref{eq:existence1}). 
%

Next, assuming that (\ref{eq:gcond}) holds, we evaluate $\expect[g(\hat{\mb{f}}_k(X),X)]$ in an identical fashion to the derivation of (\ref{inbias}).
\begin{eqnarray}
&&\expect[1_{\natural(X)}g(\hat{\mb{f}}_k(X),X)] = \expect \left[1_{\natural}(X) g\left( \hat{\mb{f}}_c(X) + k_Mh(X)(\mb{P}(X))^{2/d-1} + k_M\mb{h_s}(X),X \right) \right] \nonumber \\
&& = \expect \left[1_{\natural(X)} g\left( \hat{\mb{f}}_c(X) + k_M h(X)(\mb{P}(X))^{2/d-1} + k_M o((\mb{P}(X))^{2/d-1}),X \right) \right] \nonumber \\
&& = \expect \left[g \left( \hat{\mb{f}}_c(X) + k_M h(X)(\mb{P}(X))^{2/d-1} + k_M o((\mb{P}(X))^{2/d-1}),X \right) \right] + O({\cal C}(k))\nonumber \\
&&= \expect \left [g(\hat{\mb{f}}_c(X),X) + g'({\hat{\mb{f}}_c(X)},X)k_M h(X)(\mb{P}(X))^{2/d-1} + o(k_M \mb{P}(X))^{2/d-1})  \right] + O({\cal C}(k)) \nonumber \\
&&= g(f(X),X)g_1(k,M) + g_2(k,M) + g'(f(X),X)h(X)(k/M)^{2/d} + o((k/M)^{2/d}) + O({\cal C}(k)). \nonumber
\end{eqnarray}
This gives, 
\begin{eqnarray}
&&\expect[g(\hat{\mb{f}}_k(X),X)] = \expect[1_{\natural(X)}g(\hat{\mb{f}}_k(X),X)] + \expect[1_{\natural^c(X)}g(\hat{\mb{f}}_k(X),X)] \nonumber \\
&& = g(f(X),X)g_1(k,M) + g_2(k,M) + g'(f(X),X)h(X)(k/M)^{2/d} + o((k/M)^{2/d}) + O({\cal C}(k)).
\label{eq:bias:BC}
\end{eqnarray}

\subsection{Moments of error function}
\label{keyidea}
\label{similaritysubsection}


Let $\gamma_1(X)$, $\gamma_2(X)$ be arbitrary continuous functions satisfying the condition: $\sup_X[\gamma_i({X})]$ is finite, $i=1,2$. Also let $\gamma(X) = \gamma_1(X)$. Let $\mb{X}_{1},..,\mb{X}_{M},\mb{X},\mb{Y}$ denote $M+2$ i.i.d realizations of the density $f$. Let $q$, $r$ be arbitrary positive integers less than $k$. Define the error function $$\mb{e}_k(X) = {\hat{\mb{f}}_{k}(X)}-\expect{[\hat{\mb{f}}_{k}(X) \mid X]}.$$ Then, 

\begin{lemma}
\label{centsim}
\begin{eqnarray}
&& \expect{\left[1_{\{\mb{X} \in {\cal S'}\}}\gamma(\mb{X})\mb{e}^q_k(\mb{X})\right]} = O(k^{-q\delta/2}) + o(1/M) + O({\cal C}(k)). 
\label{eqcentsim}
\end{eqnarray}
\end{lemma}

\begin{lemma}
\label{crosssim}
\begin{eqnarray}
Cov{\left[1_{\{\mb{X} \in {\cal S'}\}}\gamma_1(\mb{X})\mb{e}^q_k(\mb{X}),1_{\{\mb{Y} \in {\cal S'}\}}\gamma_2(\mb{Y})\mb{e}^r_k(\mb{Y})\right]} &=& O\left(\frac{1}{k^{((q+r)\delta /2-1)}M}\right) + O(k^{2/d}_M/M) \nonumber \\ &+&  O(1/M^2) +O({\cal C}(k)). 
\end{eqnarray}
\label{eqcrosssim}
\end{lemma}

%

Define the operator ${\cal M}(\mb{Z}) = \mb{Z} - \expect[\mb{Z}]$. Let $\beta$ be any positive real number and define \begin{equation} \mb{E}_\beta(X) = k_M^{\beta}({\cal M}(\mb{P}^{-\beta}(X))). \label{Edefine} \end{equation} Define the terms 
\begin{equation} \mb{e}_c(X) = {\hat{\mb{f}}_{c}(X)}-\expect{[\hat{\mb{f}}_{c}(X) \mid X]}, \label{ecdefine}\end{equation}  \begin{equation} \mb{e}_t(X) = {\cal M} \left( \sum_{t \in {\cal T}} \frac{k_M h_t(X)}{ \mb{P}^{1-t}(X))} \right), \label{etdefine} \end{equation} \begin{equation} \mb{e}_r(X) = {\cal M}(k_M\mb{h_r}(X)). \label{erdefine} \end{equation} Note that \begin{equation} \mb{e}_c(X) = f(X)\mb{E}_1(X) \label{ecE} \end{equation} and  \begin{equation} \mb{e}_t(X) = (\sum_{t \in {\cal T}} k^t_M h_t(X) (\mb{E}_{1-t}(X))) \label{etE}. \end{equation}

Define the event $\{{X} \in {\cal S'}\} \cap \{\natural({X})\}$ by $\dag({X})$. Note that under the event $\dag(X)$, $\mb{e}_k(X) = \mb{e}_c(X) + \mb{e}_t(X) + \mb{e}_r(X) =: \mb{e}_o(X)$. Also, under the event $\natural({X})$, $\mb{P}(X) \leq (1+p_k)k_M$, which implies that under the event $\natural(X)$, the following hold
\begin{equation}
\mb{E}_\beta(X) = O(1), \mb{e}_c(X) = O(1), \mb{e}_t(X) = O(1), \mb{e}_r(X) = O(1), \mb{e}_o(X) = O(1).
\label{Emax}
\end{equation}
Furthermore, by (\ref{densityineqnaturalc}), under the event $\natural(X)$, 
\begin{equation}
\mb{e}_k(X) = O(1).
\label{E2max}
\end{equation}
\begin{proof} of Lemma~\ref{centsim}. 
Since $\mb{P}(X)$ is a beta random variable, the probability density function of $\mb{P}(X)$ is given by
\begin{equation}
{f(p_X)} = {\frac{M!}{(k-1)!(M-k)!}} p_X^{k-1} (1-p_X)^{M-k}. \nonumber
\label{mar}
\end{equation}
By (\ref{eq:existence1}), $\expect [{\mb{P}^{-\beta}(X)}] = \Theta((k/M)^{-\beta})$ if $\beta <k$. We will first show that $\expect[\mb{E}^q_\beta(X)] = O(1)$ if $q\beta <k$. This in turn implies that, by (\ref{ecE}) and (\ref{etE}),  $\expect[\mb{e}^q_c(X)] = O(1)$ and $\expect[\mb{e}^q_t(X)] = O(1)$ for any $q<k$. 
\begin{eqnarray}
\expect[\mb{E}^q_\beta(X)] &=& \expect \left[ k_M^{q\beta}(\mb{P}^{-\beta}(X) - \expect[\mb{P}^{-\beta}(X)])^q  \right]\nonumber \\
&=& k_M^{q\beta} \sum_{i=1}^{q} {q \choose i} (-1)^{q-i} \expect[\mb{P}^{-i\beta}(X)]\expect[\mb{P}^{-(q-i)\beta}(X)] \nonumber \\
&=& k_M^{q\beta} \sum_{i=1}^{q} {q \choose i} (-1)^{q-i} \Theta((k/M)^{-i\beta}) \Theta((k/M)^{-(q-i)\beta}) \nonumber \\
&=& \sum_{i=1}^{q} {q \choose i} (-1)^{q-i} \Theta(1)  = O(1).
\label{Emax3}
\end{eqnarray}
By (\ref{eq:concincexplicitjoint}) and (\ref{Emax3}), $$\expect[1_{\natural\natural^c({X})}\mb{E}^q_\beta({X})] = O({{\cal C}(k)}).$$ By the definition of $\natural\natural(X)$, 
\begin{equation}
1_{\natural\natural({X})}\mb{E}^q_\beta({X}) = O\left({k^{-(\delta q/2)}}\right),
\label{Ebetabound}
\end{equation}
and therefore
$$\expect[1_{\natural\natural({X})}\mb{E}^q_\beta({X})] = O\left({k^{-(\delta q/2)}}\right).$$ This gives,
\begin{equation}
\label{prodform}
\expect[\mb{E}^q_\beta(X)] = O(k^{-\delta q/2}) + O({{\cal C}(k)}).
\end{equation}
From this analysis on $\mb{E}_\beta(X)$, it trivially follows from (\ref{ecE}) that 
\begin{equation}
\expect[\mb{e}_c^l(X)] = O(k^{-\delta l/2}) + O({\cal C}(k)).
\label{eq:momc}
\end{equation}
Also observe that by (\ref{eq:hrapprox}) and (\ref{eq:hrapprox2}), 
\begin{equation}
\expect[\mb{e}_r^l(X)] = \expect[1_{\natural({X})}\mb{e}_r^l(X)] + \expect[1_{\natural^c({X})}\mb{e}_r^l(X)] = o(1/M^l) + O({{\cal C}(k)}).
\label{eq:momr}
\end{equation}
We will now bound $\mb{e}_t^l(X)$. 
Let $L = {\sum_{t \in {\cal T}} l_tt}$. Now, using (\ref{etE}), $\mb{e}_t^l(X)$ can be expressed as a \emph{sum} of terms of the form $(k/M)^L{l \choose l_1,..,l_t}\prod_{t \in {\cal T}} ( h^l_t(X) \mb{E}_t^{l_t}(X))$ where $\sum_{t} l_t = l$. Now, we can bound each of these summands using (\ref{Ebetabound}) as follows:
\begin{eqnarray}
(k/M)^l\expect[\prod_{t \in {\cal T}}\mb{E}_t^{l_t}(X)] &=& (k/M)^L\expect[1_{\natural\natural({X})}\prod_{t \in {\cal T}}\mb{E}_t^{l_t}(X)] + (k/M)^L\expect[1_{\natural\natural^c({X})}\prod_{t \in {\cal T}}\mb{E}_t^{l_t}(X)] \nonumber \\
&=& (k/M)^L\prod_{t \in {\cal T}} O(k^{-l_t \delta/2}) + O({{\cal C}(k)})\nonumber \\
&=& (k/M)^L O(k^{-l \delta/2}) + O({{\cal C}(k)})\nonumber \\
&=& o(k^{-l \delta/2}) + O({{\cal C}(k)}).
\end{eqnarray}
This implies that
\begin{equation}
\expect[\mb{e}_t^l(X)] = o(k^{-l \delta/2}) + O({{\cal C}(k)}).
\label{eq:momt}
\end{equation}

 
Note that $\mb{e}_o^q(X)$ will contain terms of the form $(\mb{e}_c(X) + \mb{e}_t(X))^l(\mb{e}_r(X))^{q-l}$. If $l<q$, the expectation of this term can be bounded as follows
\begin{eqnarray}
&& |\expect[(\mb{e}_c(X) + \mb{e}_t(X))^l(\mb{e}_r(X))^{q-l}]| \nonumber \\ 
&& \leq \sqrt{\expect[(\mb{e}_c(X) + \mb{e}_t(X))^{2l}]\expect[(\mb{e}_r(X))^{2(q-l)}]} \nonumber \\
&& = \sqrt{O(1)^{2l}(o(1/M))^{2(q-l)}} \nonumber \\
&& = O(1) \times (o(1/M))^{q-l}  = o(1/M).
\label{eq:momo1}
\end{eqnarray}
Let us concentrate on the case $l=q$. In this case, $\mb{e}_k^q(X)$ will contain terms of the form $(\mb{e}_c(X))^m(\mb{e}_t(X))^{q-m}$. For $m<q$,  
\begin{eqnarray}
&& |\expect[(\mb{e}_c(X))^m(\mb{e}_t(X))^{q-m}]| \nonumber \\
&& \leq \sqrt{\expect[(\mb{e}_c(X))^{2l}]\expect[(\mb{e}_t(X))^{2(q-l)}]} \nonumber \\
&& = \left(O(k^{-m\delta/2}) \times o(k^{-(q-m)\delta/2})\right) + {\cal C}(k) = o(k^{-q\delta/2})+ O({\cal C}(k)).
\label{eq:momo2}
\end{eqnarray}

This therefore implies that, by (\ref{eq:momc}), (\ref{eq:momr}), (\ref{eq:momt}), (\ref{eq:momo1}) and (\ref{eq:momo2}),
\begin{eqnarray}
\expect[\mb{e}_o^q(X)] &=& \expect[\mb{e}_c^q(X)] + o(k^{-q\delta/2}) + {\cal C}(k) \nonumber \\
&=& O(k^{-q\delta/2}) + o(k^{-q\delta/2}) + o(1/M) + {\cal C}(k) \nonumber \\
&=& O(k^{-q\delta/2}) + o(1/M) + {\cal C}(k).
\end{eqnarray}

This finally implies that
\begin{align}
\expect{\left[1_{\{\mb{X} \in {\cal S'}\}}\gamma(\mb{X})\mb{e}^q_k(\mb{X})\right]}  &= \expect{\left[1_{\dag(\mb{X})}\gamma(\mb{X})\mb{e}^q_k(\mb{X})\right]} +  O({{\cal C}(k)})\hspace{0.2in} \text(by (\ref{E2max}))\nonumber \\
&= \expect{\left[1_{\dag(\mb{X})}\gamma(\mb{X})\mb{e}^q_o(\mb{X})\right]} + O({{\cal C}(k)}) \nonumber \\
&= \expect{\left[1_{\{\mb{X} \in {\cal S'}\}}\gamma(\mb{X})\mb{e}^q_o(\mb{X})\right]} + O({{\cal C}(k)}) \hspace{0.2in} \text(by (\ref{Emax})) \nonumber \\
& = O(k^{-q\delta/2}) + o(1/M) + O({{\cal C}(k)}).
\label{derive1}
\end{align}
This concludes the proof.

\end{proof}
Before proving Lemma~\ref{crosssim}, we seek to answer the following question: for which set of pair of points $\{X,Y\}$ are the $k$-NN balls disjoint?
\subsubsection{Intersecting and disjoint balls} Define $\Psi_{{\es}} := \{X,Y\} \in {\cal S'}:||X-Y|| \geq  R_{{\es}}(X) + R_{{\es}}(Y)$ where $R_{{\es}}(X)$ and $R_{{\es}}(Y)$ are the ball radii of the spherical regions $S_u(X)$ and $S_u(Y)$, such that $\int_{S_u(X)} f(z) dz = \int_{S_u(Y)} f(z) dz = (1+p_k)k_M$. We will now show that for $\{X,Y\}\in \Psi_{{\es}}$, the $k$-NN balls will be disjoint with exponentially high probability. Let $\mb{d^{(k)}_X}$ and $\mb{d^{(k)}_Y}$ denote the $k$-NN distances from $X$ and $Y$ and let $\mb{\Upsilon}$ denote the event that the $k$-NN balls intersect. For $\{X,Y\}\in \Psi_\es$, 
\begin{eqnarray}
Pr(\mb{\Upsilon}) &=& Pr(\mb{d^{(k)}_X} + \mb{d^{(k)}_Y} \geq ||X-Y||) \nonumber \\
&\leq& Pr(\mb{d^{(k)}_X} + \mb{d^{(k)}_Y} \geq R_{{\es}}(X) + R_{{\es}}(Y)). \nonumber \\
&\leq& Pr(\mb{d^{(k)}_X} \geq R_{{\es}}(X)) + Pr(\mb{d^{(k)}_Y} \geq R_{{\es}}(Y)) \nonumber \\
&=& Pr(\mb{P}(X) \geq (p_k+1)((k-1)/M)) \nonumber \\
&& + Pr(\mb{P}(Y) \geq (p_k+1)((k-1)/M)) \nonumber \\
&=& 2{\cal C}(k), \nonumber
\end{eqnarray}
where the last inequality follows from the concentration inequality~(\ref{concincc}). We conclude that for $\{X,Y\}\in \Psi_{{\es}}$, the probability of intersection of $k$-NN balls centered at $X$ and $Y$ decays exponentially in $p_k^2{k}$. Stated in a different way, we have shown that for a given pair of points $\{X,Y\}$, if the ${{\es}}$ balls around these points are disjoint, then the $k$-NN balls will be disjoint with exponentially high probability. Let ${\Delta_\es}({X},{Y})$ denote the event $\{X,Y\} \in \Psi_\es^c$. From the definition of the region $\Psi_\es$, we have $Pr(\{\mb{X},\mb{Y}\} \in \Psi_\es^c) = O(k/M)$.

Let $\{X,Y\} \in \Psi_\es$ and let $q,r$ be non-negative integers satisfying $q+r>1$. The event that the $k$-NN balls intersect is given by $\mb{\Upsilon} := \{\mb{d^{(k)}_X}+\mb{d^{(k)}_Y} > ||X-Y||\}$. The joint probability distribution of $\mb{P}(X)$ and $\mb{P}(Y)$ when the $k$-NN balls do not intersect $=: \mb{\Upsilon^c}$ is given by
\begin{equation} 
{f_{\mb{\Upsilon^c}}(p_X,p_Y)} = M!\frac{(p_Xp_Y)^{k-1}}{(k-1)!^2}\frac{(1-p_X-p_Y)^{M-2k}}{(M-2k)!}. \nonumber
\label{joi}
\end{equation}
Define
\begin{equation}
i(p_X,p_Y) = \frac{\Gamma(t)\Gamma(u)\Gamma(v)}{\Gamma(t+u+v)}p_X^{t-1}p_Y^{u-1}(1-p_X-p_Y)^{v-1}, \nonumber
\end{equation}
and note that
\begin{equation}
\int_{p_X=0}^1 \int_{p_Y=0}^1 1_{\{p_X+p_Y \leq 1\}} i(p_X,p_Y) dp_X dp_Y = 1. \nonumber
\end{equation}

\begin{figure}[ht!]
  \begin{center}
    \includegraphics[width=3.5in]{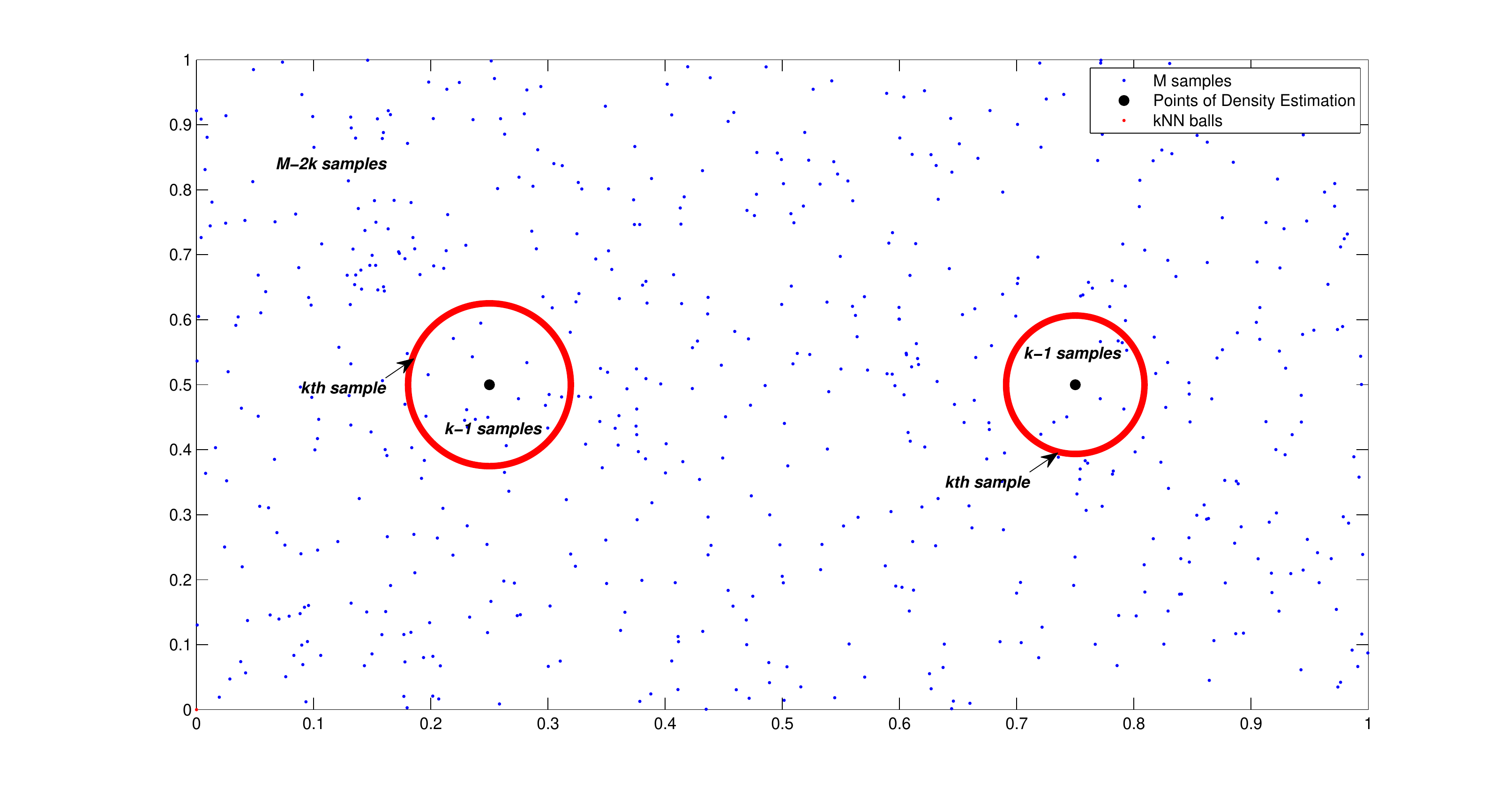}
  \end{center}

  \caption{\small Distribution of samples when $k$-NN balls are disjoint.}
  \label{fig-label8888}
\end{figure}
Figure \ref{fig-label8888} shows the distribution of the $M$ samples when the $k$-NN balls are disjoint. 
Now note that $i(p_X,p_Y)$ corresponds to the density function ${f_{\mb{\Upsilon^c}}(p_X,p_Y)}$ for the choices $t=k$, $u = k$ and $v = M-2k+1$. Furthermore, for $\{X,Y\} \in \Psi_\es$, the set ${\cal Q} := \{p_X,p_Y\}: p_X, p_Y \leq (1+p_k)(k-1)/M$ is a subset of the region ${\cal T} := \{p_X,p_Y\}: 0 \leq p_X,p_Y \leq 1$; $p_X+p_Y \leq 1$. Note that $\expect[1_{\cal Q}] = 1-{\cal C}(k)$. This implies that expectations over the region ${\cal R} := \{p_X,p_Y\}: 0 \leq p_X,p_Y \leq 1$; should be of the same order as the expectations over ${\cal T}$ with differences of order ${\cal C}(k)$. In particular, for $t,u <k$,
\begin{equation}
 \expect[\mb{P}^{-t}(X)\mb{P}^{-u}(Y)] = \expect[1_{\cal T}\mb{P}^{-t}(X)\mb{P}^{-u}(Y)] + {\cal C}(k). \nonumber
\end{equation}
From the joint distribution representation, it follows that
\begin{equation}
 \frac{\expect[1_{\cal T}\mb{P}^{-t}(X)\mb{P}^{-u}(Y)]}{\expect[\mb{P}^{-t}(X)]\expect[\mb{P}^{-u}(Y)]} = \frac{\Gamma(M-t)\Gamma(M-u)}{\Gamma(M-t-u)\Gamma(M)} = -\frac{tu}{M} + O(1/M^2). 
 \label{eq:indicov}
\end{equation}
Now observe that 
\begin{eqnarray}
\label{sumform}
&&(k_M)^{t+u}Cov(\mb{P}^{-t}(X),\mb{P}^{-u}(Y)) \nonumber \\
&& = (k_M)^{t+u}[\expect[\mb{P}^{-t}(X)\mb{P}^{-u}(Y)] - \expect[\mb{P}^{-t}(X)]\expect[\mb{P}^{-u}(Y)]] \nonumber \\
&& = (k_M)^{t+u}\expect[\mb{P}^{-t}(X)]\expect[\mb{P}^{-u}(Y)] \left[\frac{\expect[\mb{P}^{-t}(X)\mb{P}^{-u}(Y)]}{\expect[\mb{P}^{-t}(X)]\expect[\mb{P}^{-u}(Y)]}-1\right] \nonumber \\
&& = (k_M)^{t+u} \Theta(k_M^{-t})\Theta(k_M^{-u}) \left[1-\frac{tu}{M}+o(1/M^2)-1\right] \hspace{0.3in} \text{(by (\ref{eq:existence1}) and (\ref{eq:indicov}))}\nonumber \\
&& = -\left(\frac{tu}{M}\right) + O(1/M^2) .
\end{eqnarray}

Then, the covariance between the powers of the error function $\mb{E}_\beta$, for $qt,ru <k$ is given by 
\begin{eqnarray}
\label{sumext}
Cov(\mb{E}^q_t(X),\mb{E}^r_u(Y)) &=& k_M^{(t q + ur)} Cov \left(\left[{\mb{P}^{-t}(X)}-\expect \left[{\mb{P}^{-t}(X)} \right]\right]^q,\left[{\mb{P}^{-u}(Y)}-\expect \left[{\mb{P}^{-u}(Y)} \right]\right]^r \right) \nonumber \\
&=& \sum_{a=1}^{q} \sum_{b=1}^{r} \binom{q}{a} \binom{r}{b} [(-1)^{a+b}+o(1)] k_M^{(ta + ub)} Cov(\mb{P}^{-t a}(X),\mb{P}^{-u b}(Y)) \nonumber \\
&=& -tu \sum_{a=1}^{q} \sum_{b=1}^{r} \binom{q}{a} \binom{r}{b}  \frac{(-1)^aa(-1)^bb}{M}+O\left(\frac{1}{M^2}\right) \nonumber \\
&=& 1_{\{q=1,r=1\}}\left(\frac{-tu}{M}\right) + O(1/M^2).
\end{eqnarray}

\begin{proof} of Lemma~\ref{crosssim}. 
Let $\mb{X}_{1},..,\mb{X}_{M},\mb{X},\mb{Y}$ denote $M+2$ i.i.d realizations of the density $f$. Then, identical to the derivation of (\ref{derive1}) in the proof of Lemma~\ref{centsim}, 
\begin{eqnarray}
&&Cov{\left[1_{\{\mb{X} \in {\cal S'}\}}\gamma_1(\mb{X})\mb{e}^q_k(\mb{X}),1_{\{\mb{Y} \in {\cal S'}\}}\gamma_2(\mb{Y})\mb{e}^r_k(\mb{Y})\right]} \nonumber \\
&& = Cov{\left[1_{\{\mb{X} \in {\cal S'}\}}\gamma_1(\mb{X})\mb{e}^q_o(\mb{X}),1_{\{\mb{Y} \in {\cal S'}\}}\gamma_2(\mb{Y})\mb{e}^r_o(\mb{Y})\right]} + O({\cal C}(k)). \nonumber
\end{eqnarray}
Using the exact same arguments as in proof of Lemma~A.1, it can be shown that the contribution of terms $\mb{e}_r(\mb{X})$,$\mb{e}_r(\mb{Y})$ to the R.H.S. of the above equation is $o(1/M)$. Define $\sharp(\mb{X},\mb{Y}) := \gamma_1(\mb{X}) \gamma_2(\mb{Y})  Cov_{\{\mb{X},\mb{Y}\}}[(\mb{e}_c({X}) + \mb{e}_t({X}))^q, (\mb{e}_c({Y}) + \mb{e}_t({Y}))^r]$. Thus,
\begin{eqnarray}
&&Cov{\left[1_{\{\mb{X} \in {\cal S'}\}}\gamma_1(\mb{X})\mb{e}^q_k(\mb{X}),1_{\{\mb{Y} \in {\cal S'}\}}\gamma_2(\mb{Y})\mb{e}^r_k(\mb{Y})\right]} \nonumber \\
&& = \expect[1_{\{\mb{X},\mb{Y} \in {\cal S'}\}}\sharp(\mb{X},\mb{Y})] + O({\cal C}(k)) \nonumber \\
&& = \expect[\mb{1_{{\Delta_\es}^c({X},{Y})}} \sharp(\mb{X},\mb{Y})] + \expect[\mb{1_{{\Delta_\es}({X},{Y})}} \sharp(\mb{X},\mb{Y})] + O({\cal C}(k))  \nonumber \\
&& = I + II + O({\cal C}(k)). \nonumber
\end{eqnarray}

\paragraph*{For $\{X,Y\} \in \Psi_\es^c$} The covariance term $Cov_{\{\mb{X},\mb{Y}\}}[(\mb{e}_c({X}) + \mb{e}_t({X}))^q, (\mb{e}_c({Y}) + \mb{e}_t({Y}))^r]$ can be shown to be $O(k^{-(q+r)\delta/2})$ for $q,r<k$ by using Cauchy-Schwarz and (\ref{eq:momo1}), (\ref{eq:momo2}) as follows.
\begin{eqnarray}
|Cov[(\mb{e}_c({X}) + \mb{e}_t({X}))^q, (\mb{e}_c({Y}) + \mb{e}_t({Y}))^r]| &\leq & \sqrt{\var[(\mb{e}_c({X}) + \mb{e}_t({X}))^q]\var[(\mb{e}_c({Y}) + \mb{e}_t({Y}))^r]} \nonumber \\
&\leq & \sqrt{\expect[(\mb{e}_c({X}) + \mb{e}_t({X}))^{2q}]\expect[(\mb{e}_c({Y}) + \mb{e}_t({Y}))^{2r}]} \nonumber \\
&=& \sqrt{O(k^{-(2q) \delta/2})O(k^{-(2r) \delta/2})} \nonumber \\
&=& O(k^{-(q+r) \delta/2}).
\end{eqnarray}
This implies that
\begin{eqnarray}
&&II = \expect[\mb{1_{{\Delta}_\es({X},{Y})}}\sharp(\mb{X},\mb{Y})] = \expect{\left[\mb{1_{{\Delta}_\es({X},{Y})}}  O(k^{-(q+r) \delta/2})\right]} = O\left(\frac{1}{k^{((q+r) \delta/2-1)}M} \right),\nonumber
\end{eqnarray}
where the last but one step follows since the probability $Pr(\{\mb{X},\mb{Y}\} \in \Psi_\es^c) = O(k/M)$.

\paragraph*{For $\{X,Y\} \in \Psi_\es$} Now note that $(\mb{e}_c({X}) + \mb{e}_t({X}))^q$ will contain terms of the form $(\mb{e}_c(X))^m(\mb{e}_t(X))^{q-m}$. For $m<q$, the term $(\mb{e}_c(X))^m(\mb{e}_t(X))^{q-m}$ will be a sum of terms of the form $(k/M)^{(m+u)} \mb{P}^{-(m+v)}(X)$ for arbitrary $v < q-m$ with $u-v>=2/d$. By (\ref{sumform}), the covariance term $Cov[(\mb{e}_c(X))^m(\mb{e}_t(X))^{q-m}, (\mb{e}_c(Y))^n(\mb{e}_t(Y))^{r-m}]$ will be therefore be $O(k^{2/d}_M/M)$ if either $m<q$ or $n<r$.

On the other hand, if $m=q$ and $n=r$, $Cov[(\mb{e}_c(X))^q, (\mb{e}_c(Y))^r] = 1_{\{q=1,r=1\}}O(1/M) + O(1/M^2)$ by noting that the error $\mb{e}_c(X) = f(X)\mb{E}_1(X)$ and subsequently invoking (\ref{sumext}). Therefore
\begin{eqnarray}
&&I = \expect[\mb{1_{{\Delta}^c_\es({X},{Y})}}\sharp(\mb{X},\mb{Y})] \nonumber \\
&& = \expect{\left[\mb{1_{{\Delta}^c_\es({X},{Y})}}\left(1_{\{q=1,r=1\}} O(1/M) + O(k^{2/d}_M/M) + O(1/M^2)\right)\right]} \nonumber \\
&& = 1_{\{q=1,r=1\}} O(1/M) + O(k^{2/d}_M/M) + O(1/M^2), \nonumber
\end{eqnarray}
where the last step follows from the fact that probability $Pr(\{\mb{X},\mb{Y}\} \in \Psi_\es) = 1 - O(k/M) = O(1)$.
\end{proof}


\subsection{Specific cases}
We now focus on evaluating the specific cases $$\expect{\left[1_{\{\mb{X} \in {\cal S'}\}}\gamma(\mb{X})\mb{e}^2_k(\mb{X})\right]}$$ and $$Cov{\left[1_{\{\mb{X} \in {\cal S'}\}}\gamma_1(\mb{X})\mb{e}_k(\mb{X}),1_{\{\mb{Y} \in {\cal S'}\}}\gamma_2(\mb{Y})\mb{e}_k(\mb{Y})\right]}, $$ for $k>2$.

\subsubsection{Evaluation of $\expect{\left[1_{\{\mb{X} \in {\cal S'}\}}\gamma(\mb{X})\mb{e}^2_k(\mb{X})\right]}$}
$\mb{P}(X)$ has a beta distribution with parameters $k, M-k+1$. Therefore for $k>2$
\begin{eqnarray}
\expect[\mb{E}^2_\beta(X)] &=& \expect \left[ k_M^{2\beta}(\mb{P}^{-\beta}(X) - \expect[\mb{P}^{-\beta}(X)])^2  \right]\nonumber \\
&=& k_M^{2\beta} \expect[\mb{P}^{-2\beta}(X)] - \left(\expect[\mb{P}^{-\beta}(X)] \right)^2 \nonumber \\
&=& k_M^{2\beta} \left(\frac{\Gamma(k-2\beta)\Gamma(M+1)}{\Gamma(k)\Gamma(M+1-2\beta)} - \left(\frac{\Gamma(k-\beta)\Gamma(M+1)}{\Gamma(k)\Gamma(M+1-\beta)}\right)^2 \right) \nonumber \\
&=& O(1/k)
\label{eq:centralmomq=2}
\end{eqnarray}
where the last step follows by noting that for any $a>0$, $$\frac{\Gamma(x)}{\Gamma(x+a)}  = x^{-a}(1+o(1/x)).$$ From (~\ref{derive1}), 
\begin{align}
\expect{\left[1_{\{\mb{X} \in {\cal S'}\}}\gamma(\mb{X})\mb{e}^2_k(\mb{X})\right]}  &= 
\expect{\left[1_{\{\mb{X} \in {\cal S'}\}}\gamma(\mb{X})\mb{e}^2_o(\mb{X})\right]} + O({\cal C}(k)).
\end{align}
Note that $\mb{e}_o^2(X) = (\mb{e}_c(X) + \mb{e}_t(X) + \mb{e}_r(X))^2$ is a sum of terms of the form $(\mb{e}_c(X))^{2-l-m}(\mb{e}_t(X))^l(\mb{e}_r(X))^{m}$. Also,
\begin{eqnarray}
\expect[\mb{e}^2_c(X)] &=& f^2(X)\expect \left[ k_M^{2}(\mb{P}^{-1}(X) - \expect[\mb{P}^{-1}(X)])^2  \right]\nonumber \\
&=& f^2(X)k_M^{2} \expect[\mb{P}^{-2}(X)] - \left(\expect[\mb{P}^{-1}(X)] \right)^2 \nonumber \\
&=& f^2(X)k_M^{2\beta} \left(\frac{\Gamma(k-2)\Gamma(M+1)}{\Gamma(k)\Gamma(M+1-2)} - \left(\frac{\Gamma(k-1)\Gamma(M+1)}{\Gamma(k)\Gamma(M)}\right)^2 \right) \nonumber \\
&=& \frac{1}{k} + o\left(\frac{1}{k}\right).
\label{eq:centralmomq=22}
\end{eqnarray}
 Using (\ref{eq:centralmomq=2}), identical to the derivation of (\ref{eq:momo1}) and (\ref{eq:momo2}), it is clear that if  $l+m>0$, $\expect[(\mb{e}_c(X))^{2-l-m}(\mb{e}_t(X))^l(\mb{e}_r(X))^{m}] = o(k^{-1})+ o(1/M) + O({\cal C}(k))$. This implies that
\begin{align}
\expect{\left[1_{\{\mb{X} \in {\cal S'}\}}\gamma(\mb{X})\mb{e}^2_k(\mb{X})\right]}  &= 
\expect{\left[1_{\{\mb{X} \in {\cal S'}\}}\gamma(\mb{X})\mb{e}^2_o(\mb{X})\right]} + O({\cal C}(k)) \nonumber \\
&= {f^2(X)} \left(\frac{1}{k}\right) + o \left(\frac{1}{k}\right).
\label{eqcentsim2}
\end{align}

\subsubsection{Evaluation of $Cov{\left[1_{\{\mb{X} \in {\cal S'}\}}\gamma_1(\mb{X})\mb{e}_k(\mb{X}),1_{\{\mb{Y} \in {\cal S'}\}}\gamma_2(\mb{Y})\mb{e}_k(\mb{Y})\right]} $} We separately analyze disjoint balls and intersecting balls as follows:
\begin{eqnarray}
 &&Cov{\left[1_{\{\mb{X} \in {\cal S'}\}}\gamma_1(\mb{X})\mb{e}_k(\mb{X}), 1_{\{\mb{Y} \in {\cal S'}\}}\gamma_2(\mb{Y})\mb{e}_k(\mb{Y})\right]} \nonumber \\
&&= \expect[\left[1_{\{\mb{X} \in {\cal S'}\}}1_{\{\mb{Y} \in {\cal S'}\}}\gamma_1(\mb{X})\gamma_2(\mb{Y})\mb{e}_k(\mb{X}) \mb{e}_k(\mb{Y})\right]] \nonumber \\
&&= \expect[\left[1_{\{\mb{X} \in {\cal S'}\}}1_{\{\mb{Y} \in {\cal S'}\}}\gamma_1(\mb{X})\gamma_2(\mb{Y})\mb{e}_o(\mb{X}) \mb{e}_o(\mb{Y})\right]] + O({\cal C}(k)) \nonumber \\
&&= \expect[\left[1_{\{\mb{X} \in {\cal S'}\}}1_{\{\mb{Y} \in {\cal S'}\}}\gamma_1(\mb{X})\gamma_2(\mb{Y})(\mb{e}_c(\mb{X})+\mb{e}_t(\mb{X})+\mb{e}_r(\mb{X}))(\mb{e}_c(\mb{Y})+\mb{e}_t(\mb{Y})+\mb{e}_r(\mb{Y}))\right]] + O({\cal C}(k)) \nonumber \\
&&= \expect[\left[1_{\{\mb{X} \in {\cal S'}\}}1_{\{\mb{Y} \in {\cal S'}\}}\gamma_1(\mb{X})\gamma_2(\mb{Y})(\mb{e}_c(\mb{X})+\mb{e}_t(\mb{X}))(\mb{e}_c(\mb{Y})+\mb{e}_t(\mb{Y}))\right]] + O({\cal C}(k)) + o(1/M) \nonumber \\
&& = \expect[\mb{1_{{\Delta_\es}^c({X},{Y})}}\gamma_1(\mb{X}) \gamma_2(\mb{Y})  \expect_{\{\mb{X},\mb{Y}\}}[(\mb{e}_c ({X})+\mb{e}_t ({X}))(\mb{e}_c ({Y})+\mb{e}_t ({Y}))]] \nonumber \\
&& + \expect[\mb{1_{{\Delta_\es}({X},{Y})}}\gamma_1(\mb{X}) \gamma_2(\mb{Y})  \expect_{\{\mb{X},\mb{Y}\}}[(\mb{e}_c ({X})+\mb{e}_t ({X}))(\mb{e}_c ({Y})+\mb{e}_t ({Y}))]] \nonumber \\
&& + O({\cal C}(k)) + o(1/M) \nonumber \\
&& = I + II + O({\cal C}(k)) + o(1/M). \nonumber
\end{eqnarray}

\paragraph*{For $\{X,Y\} \in \Psi_\es$} 
$$\expect[(\mb{e}_c(X))(\mb{e}_c(Y))] = Cov[(\mb{e}_c(X)), (\mb{e}_c(Y))] = \frac{-f(X)f(Y)}{M} + O(1/M^2)$$ by noting that the error $\mb{e}_c(X) = \mb{E}_1(X)/f(X)$ and subsequently invoking (\ref{sumext}) in conjunction with the condition $k>2$. Similarly, using (\ref{ecE}), (\ref{etE}) and (\ref{sumext}), 
$$\expect[(\mb{e}_c(X))(\mb{e}_t(Y))] = O(k^{2/d}_M/M) + O(1/M^2),$$
$$\expect[(\mb{e}_t(X))(\mb{e}_c(Y))] = O(k^{2/d}_M/M) + O(1/M^2),$$
$$\expect[(\mb{e}_t(X))(\mb{e}_t(Y))] = O(k^{4/d}_M/M) + O(1/M^2).$$
This implies that
\begin{eqnarray}
\label{eq:disj}
&&I = \expect[\mb{1_{{\Delta}^c_\es({X},{Y})}}\expect_{\{\mb{X},\mb{Y}\}}[(\mb{e}_c ({X})+\mb{e}_t ({X}))(\mb{e}_c ({Y})+\mb{e}_t ({Y}))]] \nonumber \\
&& = \expect{\left[\mb{1_{{\Delta}^c_\es({X},{Y})}}\left({-f(X)f(Y)}(1/M) + O(k^{2/d}_M/M) + O(1/M^2)\right)\right]} \nonumber \\
&&= \expect[1_{\{\mb{X} \in {\cal S'}\}}1_{\{\mb{Y} \in {\cal S'}\}}\gamma_1(\mb{X})\gamma_2(\mb{Y})(f(\mb{X})f(\mb{Y}))]\left(-1/M + O(k^{2/d}_M/M) + O(1/M^2)\right) \nonumber \\
&&= -\expect[1_{\{\mb{X} \in {\cal S'}\}}\gamma_1(\mb{X})f(\mb{X})]\expect[1_{\{\mb{Y} \in {\cal S'}\}}\gamma_2(\mb{Y})f(\mb{Y})]\frac{1}{M} + O(k^{2/d}_M/M) + O(1/M^2). 
\end{eqnarray}
where the last but one step follows from the fact that probability $Pr(\{\mb{X},\mb{Y}\} \in \Psi_\es) = 1 - O(k/M) = O(1)$.

\paragraph*{For $\{X,Y\} \in \Psi_\es^c$}
First observe that by Cauchy Schwarz, and by (\ref{eq:centralmomq=2})
$|\expect[\mb{E}_t(X)\mb{E}_u(X)]| \leq \sqrt{\expect[\mb{E}^2_t(X)]\expect[\mb{E}^2_u(X)]} = O(1/k)$. This implies that
\begin{eqnarray}
\expect[(\mb{e}_c ({X})+\mb{e}_t ({X}))(\mb{e}_c ({Y})+\mb{e}_t ({Y}))] = \expect[\mb{e}_c ({X})\mb{e}_c ({Y})] + O(k^{2/d}_M/k).
\end{eqnarray}
In subsection~\ref{specificeval}, we will show Lemma~\ref{intersectcoveps}, which states that 
\begin{eqnarray}
 &&\expect[\mb{1_{{\Delta_\es}({X},{Y})}}\gamma_1(\mb{X})\gamma_2(\mb{Y})\mb{e_c}(\mb{X})\mb{e_c}(\mb{Y})] \nonumber \\
&& = {\expect[1_{\{\mb{X} \in {\cal S'}\}}\gamma_1(\mb{X})\gamma_2(\mb{X})f^2(\mb{X})]}\left( \frac{1}{M} + o \left(\frac{1}{M}\right)\right)  \nonumber
\end{eqnarray}

This implies that
\begin{eqnarray}
\label{eq:disj2}
&&II = \expect[\mb{1_{{\Delta}_\es({X},{Y})}}\expect_{\{\mb{X},\mb{Y}\}}[(\mb{e}_c ({X})+\mb{e}_t ({X}))(\mb{e}_c ({Y})+\mb{e}_t ({Y}))]] \nonumber \\
&& = \expect[\mb{1_{{\Delta}_\es({X},{Y})}}\expect_{\{\mb{X},\mb{Y}\}}[\mb{e}_c ({X})\mb{e}_c ({Y})] + O(k^{2/d}_M/k)]   \nonumber \\
&&= \expect[\mb{1_{{\Delta_\es}({X},{Y})}}\gamma_1(\mb{X})\gamma_2(\mb{Y})\mb{e_c}(\mb{X})\mb{e_c}(\mb{Y})] + \expect{\left[\mb{1_{{\Delta}_\es({X},{Y})}}\left(O(k^{2/d}_M/k)\right)\right]} \nonumber \\
&& = {\expect[1_{\{\mb{X} \in {\cal S'}\}}\gamma_1(\mb{X})\gamma_2(\mb{X})/f^2(\mb{X})]}\left( \frac{1}{M} + O(k^{2/d}_M/M)+ o \left(\frac{1}{M}\right)\right)
\end{eqnarray}
where the last step follows from recognizing that $Pr(\{\mb{X},\mb{Y}\} \in \Psi_\es^c) = O(k/M)$ and $O(k/M) \times 1/k = O(1/M)$. This implies that
\begin{eqnarray}
 &&Cov{\left[1_{\{\mb{X} \in {\cal S'}\}}\gamma_1(\mb{X})\mb{e}_k(\mb{X}), 1_{\{\mb{Y} \in {\cal S'}\}}\gamma_2(\mb{Y})\mb{e}_k(\mb{Y})\right]} \nonumber \\
&& = I + II + O({\cal C}(k)) + o(1/M) \nonumber \\
&&= Cov[1_{\{\mb{X} \in {\cal S'}\}}\gamma_1(\mb{X})/f(\mb{X}),1_{\{\mb{Y} \in {\cal S'}\}}\gamma_2(\mb{Y})/f(\mb{Y})]\left( \frac{1}{M} \right)  + o(1/M) + O({\cal C}(k)).
\label{eqcrosssim2}
\end{eqnarray}

\subsection{Summary}
Noting that $\delta>2/3$, the equations (\ref{eqcentsim}), (\ref{eqcrosssim}), (\ref{eqcentsim2}), (\ref{eqcrosssim2}) imply that for positive integers $q,r<k$, 
\label{summarybeforeBC}
\begin{align}
\label{bknncentIapp}
& \expect{\left[1_{\{\mb{X} \in {\cal S'}\}}\gamma(\mb{X})\mb{e}_k^q(\mb{X})\right]}  = 1_{\{q=2\}}\expect{\left[1_{\{\mb{X} \in {\cal S'}\}}\gamma(\mb{X})f^2(\mb{X})\right]} \left(\frac{1}{k}\right) + o \left(\frac{1}{k}\right) +O({\cal C}(k)),
\end{align}
\begin{align}
\label{bknncrossIapp}
 &Cov{\left[1_{\{\mb{X} \in {\cal S'}\}}\gamma_1(\mb{X})\mb{e}_k^q(\mb{X}), 1_{\{\mb{Y} \in {\cal S'}\}}\gamma_2(\mb{Y})\mb{e}_k^r(\mb{Y})\right]} \nonumber \\
& = 1_{\{q,r=1\} }Cov[1_{\{\mb{X} \in {\cal S'}\}}\gamma_1(\mb{X})f(\mb{X}),1_{\{\mb{Y} \in {\cal S'}\}}\gamma_2(\mb{Y})f(\mb{Y})]\left( \frac{1}{M} + o(1/M) \right)  \nonumber \\
& + 1_{\{q+r>2\} }\left(O\left(\frac{1}{k^{((q+r)\delta /2-1)}M}\right) + O(k^{2/d}_M/M) + O(1/M^2)\right)  +O({\cal C}(k)).
\end{align}

\subsection{Evaluation of $\expect[\mb{e}_c ({X})\mb{e}_c ({Y})]$ for $\{X,Y\}\in \Psi_{{\es}}^c$}
\label{specificeval}

For $\{X,Y\}\in \Psi_{{\es}}^c$, it will be shown that the cross-correlations $\expect[\mb{e}_c ({X})\mb{e}_c ({Y})]$ of the coverage density estimator and an oracle uniform kernel density estimator (defined below) are identical up to leading terms (without explicitly evaluating the cross-correlation between the coverage density estimates) and then derive the correlation of the oracle density estimator to obtain corresponding results for the coverage estimate. 

\paragraph*{Oracle $\es$ ball density estimate}
In order to estimate cross moments for the $k$-NN density estimator, the \emph{$\es$ ball} {density estimator} is introduced. The $\epsilon$-ball density estimator is a kernel density estimator that uses a uniform kernel with bandwidth which depends on the unknown density $f$. Let the volume of the kernel be $V_\epsilon(X)$ and the corresponding kernel region be $S_\epsilon(X) = \{Y \in {\cal S}:c_d||X-Y||^d \leq V_\epsilon(X)\}$. The volume is chosen such that the coverage $Q_\es(X) = \int_{S_\epsilon(X)}f(z) dz$ is set to $(1+p_k)k/M$. Let $\mb{l_\es}(X)$ denote the number of points among $\{\mb{X}_1,..,\mb{X}_M\}$ falling in $S_\es(X)$: $\mb{l_\epsilon(X)} = \Sigma_{i=1}^{M} 1_{\mb{X}_i \in S_\epsilon(X)}$. The \emph{$\epsilon$ ball} density estimator is defined as
\begin{equation}
  \mb{\hat{f}_\es}(X) = \frac{\mb{l_\epsilon(X)}}{MV_\epsilon(X)}.
\end{equation}
Also define the error $\mb{{e}}_\es(X)$ as $\mb{{e}}_\es(X) = \hat{\mb{f}}_\es(X)-\expect[\hat{\mb{f}}_\es(X)]$.
It is then possible to prove the following lemma using results on the volumes of intersections of hyper spheres (refer Appendix A for details).
\begin{lemma}
\label{covariancelemmaeps}
Let $\gamma_1(X)$, $\gamma_2(X)$ be arbitrary continuous functions. Let $\mb{X}_{1},..,\mb{X}_{M},\mb{X},\mb{Y}$ denote $M+2$ i.i.d realizations of the density $f$. Then,
\begin{eqnarray}
&&\expect{\left[\mb{1_{{\Delta_\es}({X},{Y})}}\gamma_1(\mb{X})\mb{e_\es}(\mb{X})\gamma_2(\mb{Y})\mb{e_\es}(\mb{Y})\right]} \nonumber \\
&& = {\expect[1_{\{\mb{X} \in {\cal S'}\}}\gamma_1(\mb{X})\gamma_2(\mb{X})f^2(\mb{X})]}\left( \frac{1}{M} + o \left(\frac{1}{M}\right)\right). \nonumber
\end{eqnarray}

\end{lemma}
Next, the cross-correlations of the coverage density estimator and the ${\es}$ ball density estimator are shown to be asymptotically equal. In particular, 
\begin{lemma}
\label{epscov}
\begin{equation}
\expect[\mb{{e}}_c(X)\mb{{e}}_c(Y)] = \expect[\mb{{e}}{_{{\es}}}(X)\mb{{e}}{_{{\es}}}(Y)] + o(1/k). \nonumber
\end{equation}
\end{lemma}


\begin{proof}

We begin by establishing the conditional density and expectation of $\hat{\mb{f}}_\mb{{\es}}(X)$ given $\hat{\mb{f}}_\mb{c}(X)$. We drop the dependence on $X$ and denote $\mb{l_{{\es}}} = \Sigma_{i=1}^{M} 1_{\{X_i \in S_{{\es}}(X)\}}$, the $k$-NN coverage by $\mb{\mb{\mb{P}}}$ and the ${{\es}}$ ball coverage $Q_{\es}(X)$ by $Q$. Let $\mb{q}=Q/\mb{\mb{\mb{P}}}$ and $\mb{r} = (Q-\mb{\mb{\mb{P}}})/(1-\mb{\mb{\mb{P}}})$. The following expressions for conditional densities and expectations are derived in \cite{rst}
\begin{eqnarray}
&&\mb{P}r\{\mb{l_{{\es}}}=l | \mb{P};\mb{P} > Q\} \nonumber \\
&& = \left\{ \begin{array}{rl}
 \binom{k-1}{l} \mb{q}^l (1-\mb{q})^{k-1-l} &\mbox{ $l=0,1,\ldots,k-1$} \\
  0 &\mbox{ $l=k,k+1,\ldots,M$}
       \end{array} \right. \nonumber
\end{eqnarray}
\begin{eqnarray}
&&\mb{P}r\{\mb{l_{{\es}}}=l | \mb{P};\mb{P} \leq Q\} \nonumber \\
&& = \left\{ \begin{array}{rl}
 0 &\mbox{ $l=0,1,\ldots,k-1$} \\
  \binom{M-k}{l-k} \mb{r}^{l-k} (1-\mb{r})^{M-l} &\mbox{ $l=k,k+1,\ldots,M$}
       \end{array} \right. \nonumber
\end{eqnarray}
which implies 
\begin{eqnarray}
&&\expect[\mb{l_{{\es}}}=l | \mb{P};\mb{P} > Q] = (k-1)Q/\mb{P} \nonumber \\ 
&&\expect[\mb{l_{{\es}}}=l | \mb{P};\mb{P} \leq Q] = \left(\frac{1-Q}{1-\mb{P}}\right)(k-M) + M \nonumber
\end{eqnarray}
Using the above expressions for conditional expectations, the following marginal expectation are obtained. Denote the density of the coverage $\mb{P}$ by $f_{k,M}(p)$. Also let $\hat{\mb{P}}$ be the coverage corresponding to the $k-2$ nearest neighbor in a total field of $M-3$ points. Then
\begin{eqnarray}
\expect[\mb{\tilde{e}}_c(X) \mb{\tilde{e}}{_{{\es}}}(X)] &=& \expect[\hat{\mb{f}}_\mb{{\es}}(X)\hat{\mb{f}}_\mb{c}(X)]-\expect[\hat{\mb{f}}_\mb{c}(X)]\expect[\hat{\mb{f}}_\mb{{\es}}(X)] \nonumber \\
&=& \expect \left[\left(\left(\frac{1-Q}{\mb{P}(1-\mb{P})}\right)(k-M) + M/\mb{P}\right)1_{\mb{P} \leq Q}\right] \nonumber \\
&& + \frac{f^2(X)(k-1)}{kM} \expect\left[\left((k-1)Q/\mb{P}^2\right)1_{\mb{P} > Q}\right] -\frac{f^2(X)}{k}MQ. \nonumber \\
&=& \frac{f^2(X)}{k} \frac{(M-1)(M-2)}{(k-2)(M-k)} \times \nonumber \\
&& \expect[({1-Q}\hat{\mb{P}})(k-M) + M\hat{\mb{P}}(1-\hat{\mb{P}})] - \frac{f^2(X)}{k}MQ \nonumber \\
&&  + \expect[((k-1)Q(1-\hat{\mb{P}}) - ({1-Q}\hat{\mb{P}})(k-M) + M\hat{\mb{P}}(1-\hat{\mb{P}}))(1_{\hat{\mb{P}} > Q})] \nonumber \\
&=& C \times (I - II + III). \nonumber 
\end{eqnarray}

It can be shown that $C \times (I-II) = \frac{f^2(X)}{k}(1-Q)$ using the fact that $\hat{\mb{P}}$ has a beta distribution. Note that from the definition of $Q = ((1+p_k)(k-1)/M)$, from the concentration inequality we have that $\expect[1_{\hat{\mb{P}} > Q}] = {\cal C}(M)$. The remainder ($C \times III$) can be simplified and bounded using the Cauchy-Schwarz inequality and the concentration inequality to show $ C \times III = o(1/M)$.

Therefore, 
\begin{eqnarray}
\label{eq:epscov}
\expect[\mb{{e}}_c(X) \mb{{e}}{_{{\es}}}(X)] &=& \frac{f^2(X)}{k}(1-Q) + {\cal C}(M). \nonumber \\
&=& \frac{f^2(X)}{k} - \frac{f^2(X)}{M} + o\left(\frac{1}{M}\right) \nonumber \\
&=& f^2(X)\left(\frac{1}{k} + o\left(\frac{1}{k}\right)\right).
\end{eqnarray}

Now denote $\mb{E}(X) = ({\mb{{e}}{_c}(X)-\mb{{e}}{_{{\es}}}(X)})$. Note that $\expect[\mb{E}^2(X)] = \expect[\mb{{e}}{_c}(X)^2]-2E[\mb{{e}}{_c}(X) \mb{{e}}{_{{\es}}}(X)] + \expect[\mb{{e}}{_{{\es}}}(X)^2]$. Since $E[\mb{{e}}{_c}(X)^2] = {f^2(X)} \frac{1}{k} + o(1/k)$ and $E[\mb{{e}}{_{{\es}}}(X)^2] = {f^2(X)}( 1/k + o(1/k) )$ it follows from (\ref{eq:epscov}) that $\expect[E(X)] = o(1/k)$. This result means $\mb{{e}}{_c}(X)$ and $\mb{{e}}{_{{\es}}}(X)$ are almost perfectly correlated. Next express the covariance between the coverage density estimates in terms of the covariance between the ${{\es}}$ ball estimates as follows:
\begin{eqnarray}
&&\expect[\mb{{e}}{_c}(X)\mb{{e}}{_c}(Y)] \nonumber \\
&& = \expect[(\mb{{e}}{_{{\es}}}(X) + \mb{E}(X)) (\mb{{e}}{_{{\es}}}(Y) + \mb{E}(Y))] \nonumber \\
&& = \expect[\mb{{e}}{_{{\es}}}(X)\mb{{e}}{_{{\es}}}(Y)] + \expect[\mb{{e}}{_{{\es}}}(X) (\mb{E}(Y))] \nonumber \\
&& + \expect[\mb{{e}}{_{{\es}}}(Y) (\mb{E}(X))] + \expect[(\mb{E}(X))(\mb{E}(Y))] \nonumber \\
&&= I + II + III +IV. \nonumber
\end{eqnarray}

Using Cauchy-Schwarz, a bound on each of the terms $II$, $III$ and $IV$ is obtained in terms of $\expect[\mb{E}(X)]$: $|II| \leq \sqrt{\expect[\mb{E}(Y)]\expect[\mb{{e}}{_{{\es}}}^2(X)]}$, $|III| \leq \sqrt{\expect[\mb{E}(X)]\expect[\mb{{e}}{_{{\es}}}^2(Y)]}$ and $|IV| \leq \sqrt{\expect[\mb{E}(X)]\expect[\mb{E}(Y)]}$. Note that the above application of Cauchy-Schwarz {\em decouples} the problem of joint expectation of density estimates located at two \emph{different} points $X$and $Y$ to a problem of estimating the error $\mb{E}$ between two different density estimates at the \emph{same} point(s). Therefore all the three terms $II$, $III$ and $IV$ are $o(1/k)$. This concludes the proof of Lemma~\ref{epscov}.
\end{proof}

For Lemma~\ref{epscov} to be useful, $\expect[\mb{{e}}{_{{\es}}}(X)\mb{{e}}{_{{\es}}}(Y)]$ must be orders of magnitude larger than the error $o(1/k)$, which is indeed the case for $\{X,Y\}\in {\Psi_\es}^c$ since $\expect[\mb{{e}}{_{{\es}}}(X)\mb{{e}}{_{{\es}}}(Y)] = O(1/k)$ (Lemma A.2, Appendix .1) for such $X$ and $Y$. This lemma can be used along with previously established results on co-variance of $\es$-ball density estimates (Lemma~\ref{covariancelemmaeps}) to obtain the following result:
\begin{lemma}
\label{intersectcoveps}
Let $\gamma_1(X)$, $\gamma_2(X)$ be arbitrary continuous functions. Let $\mb{X}_{1},..,\mb{X}_{M},\mb{X},\mb{Y}$ denote $M+2$ i.i.d realizations of the density $f$. Then,
\begin{eqnarray}
 &&\expect[\mb{1_{{\Delta_\es}({X},{Y})}}\gamma_1(\mb{X})\gamma_2(\mb{Y})\mb{e_c}(\mb{X})\mb{e_c}(\mb{Y})] \nonumber \\
&& = {\expect[1_{\{\mb{X} \in {\cal S'}\}}\gamma_1(\mb{X})\gamma_2(\mb{X})f^2(\mb{X})]}\left( \frac{1}{M} + o \left(\frac{1}{M}\right)\right)  \nonumber
\end{eqnarray}
\end{lemma}

\begin{proof}

%

\begin{eqnarray}
&&\expect[\mb{1_{{\Delta_\es}({X},{Y})}}\gamma_1(\mb{X})\gamma_2(\mb{Y})\expect_{\mb{X},\mb{Y}}[\mb{e_c}({X})\mb{e_c}({Y})]] \nonumber \\
&& =\expect[\mb{1_{{\Delta_\es}({X},{Y})}}\gamma_1(\mb{X})\gamma_2(\mb{Y})\mb{{e}_{{\es}}}(\mb{X})\mb{{e}_{{\es}}}(\mb{Y})] + o(1/k) \nonumber \\
&& = {\expect[1_{\{\mb{X} \in {\cal S'}\}}\gamma_1(\mb{X})\gamma_2(\mb{X})f^2(\mb{X})]}\left( \frac{1}{M} + o \left(\frac{1}{M}\right)\right). \nonumber
\end{eqnarray}
In the second to last step, $o(1/M)$ is obtained for the second term by recognizing that $Pr(\{\mb{X},\mb{Y}\} \in \Psi_\es^c) = O(k/M)$ and $O(k/M) \times o(1/k) = o(1/M)$. 
\end{proof}

\section{Boundary correction for density estimates}
\label{bcorrection}
\label{bouncorrec}

In the previous section, moment results were established for the standard $k$-NN density estimate $\hat{\mb{f}}_k(X)$ for points $X$ in any deterministic set ${\cal S'}$ with respect to the samples ${\cal X}_M = \{\mb{X}_{N+1},..,\mb{X}_{N+M}\}$ satisfying the condition $Pr(\mb{X} \notin {\cal S'}) = o(1)$ and ${\cal S'} \subset {\cal S}_I$, where $\mb{X}$ is an realization from density $f$. In this section, these moment results are extended to boundary corrected $k$-NN density estimate $\tilde{\mb{f}}_k(X)$ for all $X \in {\cal S}$ as follows.

Specify the set ${\cal S'}$ to be ${\cal S'} = {\cal S}_I$ as defined in (\ref{sbdefine}). Exclusively using the set ${\cal X}_N = \{\mb{X}_{1},..,\mb{X}_{N}\}$, a set of interior points $\cal{I}_N \subset \cal{X}_N$ are determined such that $\cal{I}_N \subset \cal{S'}$ with high probability $1-O(N{\cal C}(k))$. Define the set of boundary points $\cal{B}_N = \cal{X}_N - \cal{I}_N$. For points $X \in {\cal I}_N$, the boundary corrected $k$-NN density estimate $\tilde{\mb{f}}_k(X)$ is defined  to be the standard $k$-NN estimate $\hat{\mb{f}}_k(X)$, and we invoke the moment properties of the standard $k$-NN density estimate $\hat{\mb{f}}_k(X)$ derived in the previous section. For points $X \in {\cal B_N}$, the density estimate $\tilde{\mb{f}}_k(X)$ is defined as $\hat{\mb{f}}_k(Y_n)$ for points $Y_n \in {\cal I}_N$, and we invoke the moment properties of the standard $k$-NN density estimate $\hat{\mb{f}}_k(X)$ derived in the previous section. 

\subsection{Bias in the $k$-NN density estimator near boundary}
If a probability density function has bounded support, the $k$-NN balls centered at points close to the boundary are often truncated at the boundary. Let 
\begin{equation}
\mb{\alpha}_k(X) = \frac{\int_{\mb{S}_k(X) \cap {\cal S}}{dZ}}{\int_{\mb{S}_k(X)}{dZ}} \nonumber
\end{equation}
be the fraction of the volume of the $k$-NN ball inside the boundary of the support. Also define $\mb{V}_{k,M}(X)$ to be the $k$-NN ball volume in a sample of size $M$. For interior points $X \in {\cal S'}$, $\mb{\alpha}_k(X) = 1$, while for boundary points $X \in {\cal S-S'}$, $\mb{\alpha}_k(X)$ is closer to $0$ when the points are closer to the boundary. For boundary points we then have
\begin{equation}
\label{outbias}
\expect{[\hat{\mb{f}}_{k}(X)]}-f(X) = (1-\mb{\alpha}_k(X))f(X) + o(1) .
\end{equation}
Therefore the bias is much higher at the boundary of the support ($O(1)$) as compared to its interior ($O((k/M)^{2/d})$) ~(\ref{inbias}). Furthermore, the bias at the support boundary does not decay to $0$ as $k/M \to 0$. 

In the next section, we detect interior points $\cal{I}_N$ which lie in ${\cal S'}$ with high probability $O(N{\cal C}(k))$. The results on bias, variance and cross-moments derived in the previous Appendix for points $X \in {\cal S'}$ therefore carry over to the points $\cal{I}_N$. A density estimate at points $\cal{B}_N$ is then proposed that will reduce the bias of density estimates close to the boundary.

\subsection{Boundary point detection}
\label{bpd}
Define $V_{k,M}(X):=\frac{k}{M\alpha_k(X)f(X)}$. Let $p(k,M)$ be any positive function satisfying $p(k,M) = \Theta((k/M)^{2/d}) + ({\sqrt{6}}/k^{\delta/2})$. From the concentration inequality~(\ref{concincc}) and Taylor series expansion of the coverage function~(\ref{knnaprr}), for small values of $k/M$, we have
\begin{eqnarray}
\label{concincvol1}
1 - Pr\left(\left|\frac{\mb{V}_{k,M}(X)}{V_{k,M}(X)}-1\right| \leq p(k,M) \right) = O({\cal C}(k)). \nonumber
\end{eqnarray}
To determine $\cal{I}_N$ and $\cal{B}_N$, we first construct a $K$-NN graph on the samples ${\cal X}_N$ where $K = \lfloor k \times (N/M) \rfloor.$ For any ${X} \in \cal {X}_N$, from the concentration inequality~(\ref{concincc})
\begin{eqnarray}
\label{concincvol2}
1 - Pr\left(\left|\frac{\mb{V}_{K,N}(X)}{V_{K,N}(X)}-1\right| \leq p(K,N) \right) = O({\cal C}(K)) = O({\cal C}(k)),
\end{eqnarray}
where ${\cal C}(K) = O({\cal C}(k))$ because by $({\cal {A}}.0)$, $K = \theta(k)$. This implies that, with high probability, the radius of the $K$-NN ball at $X$ concentrates around $(V_{K,N}(X)/c_d)^{1/d}$. By this concentration inequality~(\ref{concincvol2}), this choice of $K$ guarantees that the size of the $k$-NN ball in the partitioned sample is the same as the the size of the $K$-NN ball in the pooled sample with high probability $1-{\cal C}(k)$. By the union bound and (\ref{concincvol2}), the probability that $$\left|\frac{\mb{V}_{K,N}(X)}{V_{K,N}(X)}-1\right| \leq p(K,N) $$ is satisfied by every $X_i \in {\cal X}_N$ is lower bounded by $1-O(N{\cal C}(k))$.

Using the $K$-NN graph, for each sample $\mb{X} \in \cal{X}_N$, we compute the number of points in $\cal{X}_N$ that have $\mb{X}$ as a $l$-th nearest neighbor ($l$-NN), $l=\{1,\ldots,K\}$. Denote this count as $count(\mb{X})$.  Let $Y$ be the $l$-nearest neighbor of $X$, $l=\{1,\ldots,K\}$. Then $Y$ can be represented as $Y = X + R_K(X)u$ where $u$ is an arbitrary vector with $||u|| \leq 1$. 

For $X$ to be one of the $K$-NN of $Y$ it is necessary that $R_K(Y) \geq ||Y-X||$ or equivalently, $R_K(Y)/R_K(X) \geq ||u||$. Using the concentration inequality~(\ref{concincvol2}) for $R_K(X)$ and $R_K(Y)$, a sufficient condition for this is
\begin{equation}
\label{eqabove}
\frac{\alpha_K(X)f(X)}{\alpha_K(Y)f(Y)}(1-2p(K,N)) \geq ||u||. 
\end{equation}
Because $f$ is differentiable and has a finite support, $f$ is Lipschitz continuous. Denote the Lipschitz constant by $\mathbb{L}$. Then, we have $|f(Y) - f(X)| \leq \mathbb{L}({K}/{c_dN\epsilon_0})^{1/d}$. Define $q(K,N) = (\mathbb{L}/\epsilon_0)({K}/{c_dN\epsilon_0})^{1/d} + 2{\sqrt{6}}/k^{\delta/2}$. Then (\ref{eqabove}) is satisfied if 
\begin{equation}
\frac{\alpha_K(X)}{\alpha_K(Y)}(1-q(K,N)) \geq ||u||. \nonumber
\end{equation}
For points $X \in \cal{S'}$, $\alpha_K(X) = 1$ with probability $1-{\cal C}(k)$. This implies that $X$ will be one of the $K$-NN of $Y$ if $||u|| \leq 1-q(K,N)$. This implies that, with probability $1-O(N{\cal C}(k))$, $count(\mb{X}) \geq K(1-q(K,N))$ whenever $X \in \cal{S'}$. On the other hand, for $X \in \cal{S-S'}$, $\alpha_K(X) < 1$ with probability $1-{\cal C}(k)$. It is also clear that for small values of $K/N$, $\alpha_K(X) < \alpha_K(Y)$ for at least $K/2$ $l$-NN $Y$ of $X$. This then implies that $count(\mb{X}) < K(1-q(K,N))$ for $X \in \cal{S-S'}$ with probability $1-O(N{\cal C}(k))$. We therefore can apply the threshold $K(1-q(K,N))$ to detect interior points $\cal{I}_N = \cal{X}_N \cap {\cal S'}$ and boundary points ${\cal B}_N = \cal{X}_N - \cal{I}_N = \cal{X}_N \cap ({\cal S-S'})$ with high probability $1-O(N{\cal C}(k))$.  Algorithm 1, shown below, codifies this into a precise procedure. 
\begin{algorithm}                      
\caption{Detect boundary points $\cal{B}_N$}          
\label{alg1H}                           
\begin{algorithmic}                    
\STATE 1. Construct $K$-NN tree on $\cal{X}_N$
\STATE 2. Compute $count(\mb{X})$ for each $\mb{X} \in \cal{X}_N$
\STATE 3. Detect boundary points $\cal{B}_N$:
\FOR {each $\mb{X} \in \cal{X}_N$}
 \IF {$count(\mb{X}) < (1-q(K,N))K$} 
        \STATE ${\cal B}_N \gets \mb{X}$
\ELSE
        \STATE ${\cal I}_N \gets \mb{X}$
\ENDIF 
\ENDFOR
\end{algorithmic}
\end{algorithm}

\subsection{Boundary corrected density estimator}

Here the boundary corrected $k$-NN density estimator is defined and its asymptotic rates are computed. The proposed density estimator corrects the $k$-NN ball volumes  for points that are  close to the boundary. To estimate the density at a boundary point ${\mb{X} \in \cal{B}_N}$, we find a point $\mb{Y} \in {\cal I_N}$ that is close to $\mb{X}$. Because of the proximity of $\mb{X}$ and $\mb{Y}$, $f(\mb{X}) \approx f(\mb{Y})$. We can then estimate the density at $\mb{Y}$ instead and use this as an estimate of $f(\mb{Y})$.  This informal argument is made more precise in what follows.

Consider the corrected density estimator $\tilde{\mb{{f}}}_k$ defined in (\ref{correcdec3}). This estimator has bias of order $O((k/M)^{1/d})$, which can be shown as follows. Let $\mb{X}$ denote $\mb{X}_i$ for some fixed $i \in \{1,..,N\}$. Also, let $\mb{X}_{-1} = \text{arg} \min_{x \in {\cal S}'} d(x,{\mb{X}})$. 

Given ${\cal X}_N$, if $X \in {\cal I}_N$, then by (\ref{inbias}), 
\begin{eqnarray}
\expect[\tilde{\mb{f}}_{k}({X})] &=& \expect[\hat{\mb{f}}_{k}({X})] = f({X}) + O((k/M)^{2/d}) + O({\cal C}(k)). \nonumber
\end{eqnarray}

Next consider the alternative case $X \in {\cal B}_N$. Let ${X}_n \in {\cal I}_N$ be the closest interior point to ${X}$. Define $h = {X} - {X}_n$. $h$ can be rewritten as $h = h_1 + h_2$, where $h_1 = {X} - {X_{-1}}$ and $h_2 = X_{-1} - X_n$. Since $X \in {\cal B}_N$ implies that $X \in {\cal S-S'}$ with probability $1-O(N{\cal C}(k))$, consequently $||h_1|| = ||{X} - {X_{-1}}|| = O((k/M)^{1/d})$ with probability $1-O(N{\cal C}(k))$. 

Again with probability $1-O(N{\cal C}(k))$, $X_n \in {\cal S'}$. Let ${\cal C}_N = \cup_{Y \in {\cal S'}} \text{arg}{\min}_{x \in {\cal I}_N} d(x,Y)$. By construction of ${\cal C}_N$, $X_n \in {\cal C}_N$. Consequently, by (\ref{concincvol2}), $||h_2|| = ||{X}_{-1} - {X_{n}}|| = O((1/N)^{1/d}) = o((k/M)^{1/d})$. 

Because $||h_1|| = ||{X} - {X_{-1}}|| = O((k/M)^{1/d})$ and $||h_2|| = ||{X}_{-1} - {X_{n}}|| = o((k/M)^{1/d})$ with probability $1-O(N{\cal C}(k))$, consequently with probability $1-O(N{\cal C}(k))$, $||h|| = O((k/M)^{1/d})$.  Now,
\begin{equation}
f({X}) = f({X}_{n})+ O(||h||) \nonumber. 
\end{equation}
If ${X}_{n}$ is located in the interior {${\cal S'}$}, by (\ref{inbias}), 
\begin{eqnarray}
\expect[\hat{\mb{f}}_{k}({X}_{n})] &=& f({X}_{n}) + O((k/M)^{2/d}) + O({\cal C}(k)), 
\label{outbiasCC}
\end{eqnarray}
and therefore
\begin{eqnarray}
\expect[\tilde{\mb{f}}_k({X})] &=& \expect[\hat{{f}}_{k}(\mb{X}_{n})] + O(N{\cal C}(k))\nonumber \\
&=& f({X}_{n}) + O((k/M)^{2/d}) + O(N{\cal C}(k))\nonumber \\
&=& f({X}) + O(||h||) + O((k/M)^{2/d})+ O(N{\cal C}(k)) \nonumber \\
&=& f({X}) + O((k/M)^{1/d})+ O(N{\cal C}(k)),
\label{outbiasC}
\end{eqnarray}
where the $O(N{\cal C}(k))$ accounts for error in the case of the event that $X_{n(i)} \notin {\cal S'}$. This implies that the corrected density estimate has lower bias as compared to the standard $k$-NN density estimate (compare to (\ref{inbias}) and (\ref{outbias})). In particular, boundary compensation has reduced the bias of the estimator at points near the boundary from $O(1)$ to $O((k/M)^{1/d})+ O(N{\cal C}(k))$.

\subsection{Properties of boundary corrected density estimator}
By section~\ref{bpd}, $\cal{I}_N \in {\cal S'}$ with probability $1-N{\cal C}(k)$. The results on bias, variance and cross-moments of the standard $k$-NN density estimator $\mb{\hat{f}}_k$ derived in the previous Appendix for points $X \in {\cal S'}$ therefore carry over to the corrected density estimator $\tilde{\mb{{f}}}_k$ for points $\cal{I}_N$ with error of order $O(N{\cal C}(k))$.

In the definition of the corrected estimator $\tilde{\mb{{f}}}_k$ in (\ref{correcdec3}), $\hat{\mb{f}}_{k}(\mb{X}_{n(i)})$ is the standard $k$-NN density estimates and $\mb{X}_{n(i)} \in {\cal S'}$ . It therefore follows that the variance and other central and cross moments of the corrected density estimator $\tilde{\mb{f}}_k$ will continue to decay at the same rate as the standard $k$-NN density estimator in the interior, as given by (\ref{bknncentIapp}) and (\ref{bknncrossIapp}). 

Given these identical rates and that the probability of a point being in the boundary region ${\cal S-\cal S'}$ is $O((k/M)^{1/d}) = o(1)$, the contribution of the boundary region to the overall variance and other cross moments of the boundary corrected density estimator $\tilde{\mb{{f}}}_k$ are asymptotically negligible compared to the contribution from the interior. As a result we can now generalize the results from Appendix A on the central moments and cross moments to include the boundary regions as follows. Denote ${\tilde{\mb{f}}_k({X})}-\expect_{X}[\tilde{\mb{f}}_k({X}) \mid X]$ by $\mb{e}({X})$. 

\subsubsection{Central and cross moments}
For positive integers $q,r<k$
\begin{align}
\label{bknncent}
& \expect{\left[\gamma(\mb{X})\mb{e}^q(\mb{X})\right]} = 1_{\{q=2\}}\expect{\left[\gamma(\mb{X})f^2(\mb{X})\right]} \left(\frac{1}{k}\right) + o \left(\frac{1}{k}\right)+ O(N{\cal C}(k)),
\end{align}
\begin{align}
\label{bknncross}
 &Cov{\left[\gamma_1(\mb{X})\mb{e}^q(\mb{X}), \gamma_2(\mb{Y})\mb{e}^r(\mb{Y})\right]} \nonumber \\
& = 1_{\{q,r=1\} }Cov[1_{\{\mb{X} \in {\cal S'}\}}\gamma_1(\mb{X})f(\mb{X}),1_{\{\mb{Y} \in {\cal S'}\}}\gamma_2(\mb{Y})f(\mb{Y})]\left( \frac{1}{M} + o(1/M) \right)  \nonumber \\
& + 1_{\{q+r>2\} }\left(O\left(\frac{1}{k^{((q+r)\delta /2-1)}M}\right) + O(k^{2/d}_M/M) + O(1/M^2)\right)  +O(N{\cal C}(k)).
\end{align}
Next, we derive the following result on the bias of boundary corrected estimators.
\subsubsection{Bias}
For $k>2$, 
\begin{eqnarray}
&&\expect[\gamma(\expect[\tilde{\mb{f}}_k(\mb{X}) \mid \mb{X}])-\gamma(f(\mb{X})))] = \expect \left[ \expect \left[ (\gamma(\tilde{\mb{f}}_k(\mb{X}))-\gamma(f(\mb{X}))) \mid {\cal X}_N \right] \right]\nonumber \\
&& = \expect \left[\expect \left[ 1_{\{X \in {\cal I}_N\}}(\gamma(\expect[\tilde{\mb{f}}_k({X})])-\gamma(f({X}))) \mid {\cal X}_N \right]\right] + \expect \left[\expect \left[  1_{\{X \in {\cal B}_N\}} (\gamma(\expect[\tilde{\mb{f}}_k({X})])-\gamma(f({X}))) \mid {\cal X}_N\right] \right] \nonumber \\
&& = I+II.
\end{eqnarray}
From (\ref{inbias}), and $Pr(\mb{X} \in {\cal B}_N) = O((k/M)^{1/d})$, we have 
\begin{eqnarray}
I = \expect \left[ \gamma'(f(\mb{X}))h(\mb{X}) \right] \left({\frac{k}{M}}\right)^{2/d} + o\left(\frac{k}{M}\right)^{2/d}+ O(N{\cal C}(k)).
\end{eqnarray}
Next, we will now derive $II$. 
\begin{eqnarray}
II &=&\expect \left[\expect \left[  1_{\{X \in {\cal B}_N\}} (\gamma(\expect[\tilde{\mb{f}}_k({X})])-\gamma(f({X}))) \mid {\cal X}_N\right] \right] \nonumber \\
&& = \expect \left[\expect \left[  1_{\{X \in {\cal B}_N\}} (\gamma(f(X_n))-\gamma(f({X}))) +  O\left(\frac{k}{M}\right)^{2/d} \mid {\cal X}_N\right] \right] + O(N{\cal C}(k)),
\end{eqnarray}
where the last step follows by (\ref{outbiasCC}). Let us concentrate on the inner expectation now. By section~\ref{bpd}, we know that with probability $1-O(N{\cal C}(k))$, if $X \in {\cal B}_N$, then $X \in {\cal S-S'}$ and if $X_n \in {\cal I}_N$, then $X_n \in {\cal S'}$. Furthermore, $||X-X_{-1}|| = O(k/M)^{1/d}$ and $||X_{-1}-X_n|| = o(k/M)^{1/d}$ with probability $1-O(N{\cal C}(k))$. This implies that
\begin{eqnarray}
&&\expect \left[  1_{\{X \in {\cal B}_N\}} (\gamma(f(X_n))-\gamma(f({X}))) +  O\left(\frac{k}{M}\right)^{2/d} \mid {\cal X}_N\right] \nonumber \\
&& = \expect \left[  1_{\{X \in {\cal S-S'}\}} (\gamma(f(X_{-1}))-\gamma(f({X}))) \mid {\cal X}_N\right] + o\left(\frac{k}{M}\right)^{1/d} + O(N{\cal C}(k)). \nonumber
\end{eqnarray}
Since $Pr(\mb{X} \in {\cal S-S'}) = O((k/M)^{1/d})$, this in turn implies that
\begin{eqnarray}
II &=&\expect \left[\expect \left[  1_{\{X \in {\cal B}_N\}} (\gamma(\expect[\tilde{\mb{f}}_k({X})])-\gamma(f({X}))) \mid {\cal X}_N\right] \right] \nonumber \\
&& = \expect [  1_{\{\mb{X} \in {\cal S-S'}\}} (\gamma(f(\mb{X}_{-1}))-\gamma(f(\mb{X})))]  + o\left(\frac{k}{M}\right)^{2/d}+ O(N{\cal C}(k)). 
\label{IIorder}
\end{eqnarray}
We therefore finally get, 
\begin{eqnarray}
\label{bknnbiastotal}
\label{bknnbias2}
&&\expect[\gamma(\expect[\tilde{\mb{f}}_k(\mb{X}) \mid \mb{X}])-\gamma(f(\mb{X})))] = I+II \nonumber \\
&& = \expect \left[ \gamma'(f(\mb{X}))h(\mb{X}) \right] \left({\frac{k}{M}}\right)^{2/d} + \expect [ 1_{\{\mb{X} \in {\cal S-S'}\}} (\gamma(f(\mb{X}_{-1}))-\gamma(f(\mb{X})))]  + o\left(\frac{k}{M}\right)^{2/d}+ O(N{\cal C}(k)).
\end{eqnarray}
Note that $||\mb{X} - \mb{X}_{-1}|| = O((k/M)^{1/d})$ with probability $1-O(N{\cal C}(k))$. This therefore implies that $$c_3 = \expect [ 1_{\{\mb{X} \in {\cal S-S'}\}} (\gamma(f(\mb{X}_{-1}))-\gamma(f(\mb{X})))] = O((k/M)^{1/d}) \times O((k/M)^{1/d}) + O(N{\cal C}(k)) = O((k/M)^{2/d}) + O(N{\cal C}(k)).$$ 

\subsubsection{Optimality of boundary correction}
Comparing (\ref{bknnbias2}), (\ref{bknncent}) and (\ref{bknncross}) with (\ref{inbias}), (\ref{bknncentIapp}) and (\ref{bknncrossIapp}) respectively,  oracle rates of convergence of bias, and central and cross moments for the boundary corrected density estimate are attained. The oracle rates are defined as the rates of MSE convergence attainable by the {\emph {oracle}} density estimate that knows the boundary of $\mathcal S$ $$\tilde{\mb{f}}_{k,o} = \frac{k-1}{M\mb{\mb{V}}_{k,o}(X)},$$ where $\mb{V}_{k,o}(X)$ is the volume of the region $\mb{S}_k(X) \cap {\cal S}$. It follows that the boundary compensated BPI estimator is adaptive in the sense that it's asymptotic MSE rate of convergence is identical to that of a $k$-NN plug-in estimator that knows the true boundary. Equivalent corrections exist for the uniform kernel density estimator and will be left to the reader.

\section{Proof of theorems on bias and variance}

\newcommand{\ttt}{\expect_\mb{Z}[\tilde{\mb{f}}_k(\mb{Z})]}
\newcommand{\ttti}{\expect_{\mb{X}_i}[\tilde{\mb{f}}_k(\mb{X}_i)]}
\newcommand{\tttj}{\expect_{\mb{X}_j}[\tilde{\mb{f}}_k(\mb{X}_j)]}
\newcommand{\tttt}{\expect_{\mb{X}_2}[\tilde{\mb{f}}_k(\mb{X}_2)]}

\begin{lemma}
\label{boundonexpec}

Assume that $U(x,y)$ is any arbitrary functional which satisfies 
$$ (i) \sup_{x \in (\epsilon_0,\epsilon_1) }|U(x,y)|  = G_0 < \infty, $$
$$(ii) \sup_{x \in (q_l,q_u) }|U(x,y)|{\cal C}(k)  = G_1 < \infty,$$ 
$$(iii) \expect[\sup_{x \in (p_l,p_u)}|U(x/\mb{p},y)|]  = G_2 < \infty.$$
Let $\mb{Z}$ denote $\mb{X}_i$ for some fixed $i \in \{1,..,N\}$. Let $\zeta_{\mb{Z}}$ be any random variable which almost surely lies in the range $(f(\mb{Z}),{\tilde{\mb{f}}_k(\mb{Z})})$. Then, $$\expect[|U(\zeta_{\mb{Z}},{\mb{Z}})|] < \infty.$$
\end{lemma}

\begin{proof}

We will show that the conditional expectation $\expect[|U(\zeta_{Z},{Z})| \mid{\cal X}_N] < \infty.$ Because $0<\epsilon_0<f(X)<\epsilon_\infty<\infty$ by $({\cal {A}}.1)$, it immediately follows that $$ \expect[|U(\zeta_{\mb{Z}},{\mb{Z}})|] = \expect[\expect[|U(\zeta_{Z},{Z})| \mid {\cal X}_N]] < \infty.$$

For fixed ${\cal X}_N$, $Z \in {\cal I}_N$ or $Z \in {\cal B}_N$. These two cases are handled seperately. 

\paragraph*{Case 1: $Z \in {\cal I}_N$}

In this case, $\tilde{\mb{f}}_k({Z}) = \hat{\mb{f}}_k({Z})$. By (\ref{densityineqnatural}) and $({\cal {A}}.1)$, we know that if $\natural({Z})$ holds, $p_l/{\mb{P}(Z)} < \hat{\mb{f}}_k({Z}) < p_u/{\mb{P}(Z)}$. On the other hand, if $\natural^c({Z})$ holds, by (\ref{densityineqnaturalc}) and $({\cal {A}}.1)$, $q_l < \hat{\mb{f}}_k({Z}) < q_u$. This therefore implies that if $\natural({Z})$ holds, $\min\{\epsilon_0,p_l/{\mb{P}(Z)}\} < \zeta_{Z} < \max\{\epsilon_\infty,p_u/{\mb{P}(Z)}\}$ and if $\natural^c({Z})$ holds, $\min\{\epsilon_0,q_l\} < \zeta_{Z} < \max\{\epsilon_\infty,q_u\}$. Then, 
\begin{eqnarray}
\expect[|U(\zeta_{Z},{Z})| \mid {\cal X}_N] &=& \expect[1_{\natural(Z)}|U(\zeta_{Z},{Z})| \mid {\cal X}_N] + \expect[1_{\natural^c(Z)}|U(\zeta_{Z},{Z})| \mid {\cal X}_N] \nonumber \\
&\leq& G_0 + \expect[1_{\natural(Z)}\sup_{x \in (p_l,p_u)} |U(x/\mb{P}(Z),Z)|] + \max\{G_0,G_1/{\cal C}(k)\}(1-Pr(\natural(Z))) \nonumber \\
&\leq& G_0 + \expect[\sup_{x \in (p_l,p_u)} |U(x/\mb{P}(Z),Z)|] + \max\{G_0,G_1/{\cal C}(k)\}(1-Pr(\natural(Z))) \nonumber \\
&=& G_0 + G_2 + \max\{G_1/{\cal C}(M),G_0\}{\cal C}(k) \nonumber \\
&=& G_0 + G_2 + \max\{G_1,G_0{\cal C}(k)\} < \infty
\end{eqnarray}
where the final step follows from the fact that ${\cal C}(k) = o(1)$.

\paragraph*{Case 2: $Z \in {\cal B}_N$}
If $Z \in {\cal B}_N$, let $Y_n$ be the nearest neighbor of $Z$ in the set ${\cal I}_N$. Then, 
\begin{eqnarray}
\label{correcdec5}
\mb{\tilde{f}}_k(Z) = \mb{\hat{f}}_k(Y_n)
\end{eqnarray}
This implies that we can now condition on the event $\natural(Y_n)$, and follow the exact procedure as in case 1 to obtain
\begin{eqnarray}
\expect[|U(\zeta_{Z},{Z})| \mid {\cal X}_N] &=& \expect[1_{\natural(Y_n)}|U(\zeta_{Z},{Z})| \mid {\cal X}_N] + \expect[1_{\natural^c(Y_n)}|U(1/\zeta_{Z},{Z})| \mid {\cal X}_N] \nonumber \\
&\leq& G_0 + G_2 + \max\{G_1,G_0{\cal C}(k)\} < \infty
\end{eqnarray}
where the final step follows from the fact that ${\cal C}(k) = o(1)$. This concludes the proof.



\end{proof}

\textbf{Proof of Theorem~\ref{knnbiasH}}.
\begin{proof}

Using the continuity of $g'''(x,y)$, construct the following third order Taylor series of $g(\tilde{\mb{f}}_k(\mb{Z}),\mb{Z})$ around the conditional expected value $\expect_Z[\tilde{\mb{f}}_k(\mb{Z})] = \expect[\tilde{\mb{f}}_k(\mb{Z}) \mid \mb{Z}]$. 
\begin{eqnarray}
&&g({\tilde{\mb{f}}_k(\mb{Z})},\mb{Z}) = g(\ttt,\mb{Z})+g'(\ttt,\mb{Z})\mb{e}(\mb{Z}) \nonumber \\
&& + \frac{1}{2}g''(\ttt,\mb{Z})\mb{e}^2(\mb{Z}) + \frac{1}{6}g^{(3)}(\zeta_\mb{Z},\mb{Z})\mb{e}^3(\mb{Z}),  \nonumber
\end{eqnarray}
where $\zeta_\mb{Z} \in (\ttt,{\tilde{\mb{f}}_k(\mb{Z})})$ is defined by the mean value theorem. This gives
\begin{align}
 &\expect{[({g}(\tilde{\mb{f}}_k(\mb{Z}),\mb{Z}) - {g}(\ttt,\mb{Z}))]} \nonumber \\
&= \expect{\left[\frac{1}{2}g''(\ttt,\mb{Z})\mb{e}^2(\mb{Z})\right]} + \expect{\left[\frac{1}{6}g^{(3)}(\zeta_\mb{Z},\mb{Z})\mb{e}^3(\mb{Z})\right]} \nonumber 
\end{align}
 Let $\Delta(\mb{Z}) = \frac{1}{6}g^{(3)}(\zeta_\mb{Z},\mb{Z})$. Direct application of Lemma~\ref{boundonexpec} in conjunction with assumptions $({\cal {A}}.5)$ , $({\cal {A}}.6)$ implies that $\expect[\Delta^2(\mb{Z})] = O(1)$. By Cauchy-Schwarz and  assumption $({\cal {A}}.4)$ applied to (\ref{bknncent}) for the choice $q=6$, 
\begin{eqnarray}
&&\left| \expect{\left[\frac{1}{6}\Delta(\mb{Z})\mb{e}^3(\mb{Z})\right]} \right| \leq \sqrt{\expect{\left[\frac{1}{36}\Delta^2(\mb{Z})\right] \expect \left[\mb{e}^6(\mb{Z})\right]}} = o\left(\frac{1}{k}\right) + O(N{\cal C}(k)). \nonumber
\end{eqnarray}

By observing that the density estimates $\{\tilde{\mb{f}}_k(\mb{X}_i)\}, i=1,\ldots,N$ are identical, we therefore have
\begin{eqnarray}
&&\expect[\hat{\mb{G}}_N(\mb{\tilde{f}}_k)] - G(f) = \expect{[{g}(\tilde{\mb{f}}_k(\mb{Z}),\mb{Z}) - {g}({{f}(\mb{Z})},\mb{Z})]} \nonumber \\
&& = \expect{[{g}(\ttt,\mb{Z}) - {g}({{f}(\mb{Z})},\mb{Z})]} + \expect{\left[\frac{1}{2}g''(\ttt,\mb{Z})\mb{e}^2(\mb{Z})\right]} + o(1/k) + O(N{\cal C}(k)). \nonumber 
\end{eqnarray}
By (\ref{bknnbias2}) and (\ref{bknncent}) for the choice $q=2$, in conjunction with assumption  $({\cal {A}}.4)$,this implies that
\begin{eqnarray}
\expect[\hat{\mb{G}}_N(\mb{\tilde{f}}_k)] - G(f) &=& \expect{[g'(f(\mb{Z}),\mb{Z})h(\mb{Z})]}\left({\frac{k}{M}}\right)^{2/d} + \expect[1_{\{\mb{Z} \in {{\cal S-S}_I}\}} (g(f(\mb{Z}_{-1}),\mb{Z}_{-1}) - g(f(\mb{Z}),\mb{Z}))] \nonumber \\
&& +  \expect{[f^2(\mb{Z})g''(\ttt,\mb{Z})/2]}\left(\frac{1}{k}\right) + O(N{\cal C}(k)) + o\left(\frac{1}{k} + \left(\frac{k}{M}\right)^{2/d}\right) \nonumber \\
&=&\expect{[g'(f(\mb{Z}),\mb{Z})h(\mb{Z})]}\left({\frac{k}{M}}\right)^{2/d} + \expect[1_{\{\mb{Z} \in {{\cal S-S}_I}\}} (g(f(\mb{Z}_{-1}),\mb{Z}_{-1}) - g(f(\mb{Z}),\mb{Z}))] \nonumber \\ 
&& + \expect{[f^2(\mb{Z})g''(f(\mb{Z}),\mb{Z})/2]}\left(\frac{1}{k}\right) + O(N{\cal C}(k)) + o\left(\frac{1}{k} + \left(\frac{k}{M}\right)^{2/d}\right) \nonumber \\
&=&c_1\left({\frac{k}{M}}\right)^{2/d} + c_2\left(\frac{1}{k}\right) + c_3 + O(N{\cal C}(k)) + o\left(\frac{1}{k} + \left(\frac{k}{M}\right)^{2/d}\right), \nonumber
\end{eqnarray}
where the last but one step follows because, by (\ref{inbias}) and (\ref{outbiasC}), we know $\ttt = f(\mb{Z}) + o(1)$. This in turn implies $\expect{[f^2(\mb{Z})g''(\ttt,\mb{Z})/2]} = \expect{[f^2(\mb{Y})g''(f(\mb{Y}),\mb{Y})/2]}$. Finally, by assumption $({\cal {A}}.5)$ and $({\cal {A}}.2)$, the leading constants $c_1$ and $c_2$ are bounded. We have also shown in equation (\ref{IIorder}) that $c_3 = O((k/M)^{2/d})$. This concludes the proof.

\end{proof}

\textbf{Proof of Theorem~\ref{knnbiasHRS}}
\begin{proof}
Let $\mb{X}$ denote $\mb{X}_i$ for some fixed $i \in \{1,..,N\}$. Also, let $\mb{X}_{-1} = \text{arg} \min_{x \in {\cal S}_I} d(x,{\mb{X}})$. Using  (\ref{eq:bias:BC}), we can derive the following in an identical manner to (\ref{bknnbiastotal}):
\begin{eqnarray}
\mathbb{B}(\hat{\mb{G}}_{N,BC}(\mb{\tilde{f}}_k)) &=&  \expect[\hat{\mb{G}}_{N,BC}(\mb{\tilde{f}}_k)] - \int g(f(x),x) f(x) dx \nonumber \\
&=& (\expect[g(\tilde{\mb{f}}_k(\mb{Z}),\mb{Z})]-g_2(k,M))/g_1(k,M) - \int g(f(x),x) f(x) dx \nonumber \\
&=& \expect[\expect[(g(\tilde{\mb{f}}_k(\mb{Z}),\mb{X})-g_2(k,M))/g_1(k,M) \mid {\cal X}_N]] - \int g(f(x),x) f(x) dx \nonumber \\
&=& \expect[\expect[(g(\tilde{\mb{f}}_k(\mb{X}),\mb{X})-g_2(k,M))/g_1(k,M) \mid {\cal X}_N], X \in {\cal I}_N ] \nonumber \\ &&+ \expect[\expect[(g(\tilde{\mb{f}}_k(\mb{X}),\mb{X})-g_2(k,M))/g_1(k,M) \mid {\cal X}_N], X \in {\cal B}_N ] \nonumber \\ && - \int g(f(x),x) f(x) dx  \nonumber \\
&=& \expect[g(f(\mb{X}),\mb{X}) + \frac{g'(f(\mb{X}),\mb{X})h(\mb{X})}{g_1(k,M)}(k/M)^{2/d} \nonumber \\ && + \frac{1_{\{\mb{X} \in {\cal S-S'}\}}}{g_1(k,M)} (g(f(\mb{X}_{-1}),\mb{X}_{-1})-g(f(\mb{X}),\mb{X})) \nonumber \\
&& + o((k/M)^{2/d}) + O(N{\cal C}(k))] - \int g(f(x),x) f(x) dx \nonumber \\
&=& \frac{c_{1}}{g_1(k,M)}\left({\frac{k}{M}}\right)^{2/d} + \frac{c_{3}}{g_1(k,M)} + o\left(\left(\frac{k}{M}\right)^{2/d}\right) + O(N{\cal C}(k)). \nonumber
\end{eqnarray}
Because we assume the logarithmic growth condition $k = O((\log(M))^{2/(1-\delta)})$, it follows that $O(N{\cal C}(k)) = O(N/M^3) = o(1/T)$. Also, by (\ref{eq:gcond}), $g_1(k,M) = 1+o(1)$. This implies that 
\begin{eqnarray}
\mathbb{B}(\hat{\mb{G}}_{N,BC}(\mb{\tilde{f}}_k)) &=&  {c_{1}}\left({\frac{k}{M}}\right)^{2/d} + {c_{3}} + o\left(\left(\frac{k}{M}\right)^{2/d}\right). 
\label{BCbiasproof}
\end{eqnarray}



\end{proof}

\textbf{Proof of Theorem~\ref{knnvarH} and Theorem~\ref{knnvarHRS}}.
\begin{proof}
By the continuity of $g^{(\lambda)}(x,y)$, we can construct the following Taylor series of $g(\tilde{\mb{f}}_k(\mb{Z}),\mb{Z})$ around the conditional expected value $\expect_Z[\tilde{\mb{f}}_k(\mb{Z})]$.
\begin{eqnarray}
g(\tilde{\mb{f}}_k(\mb{Z}),\mb{Z}) &=& g(\ttt,\mb{Z})+ {g'}(\ttt,\mb{Z})\mb{e}(\mb{Z}) \nonumber \\ 
&+& \left(\sum_{i=2}^{\lambda-1}\frac{g^{(i)}(\ttt,\mb{Z})}{i!}\mb{e}^i(\mb{Z})\right) + \frac{g^{(\lambda)}(\xi_\mb{Z},\mb{Z})}{\lambda!}\mb{e}^\lambda(\mb{Z}),  \nonumber
\end{eqnarray}
where $\xi_\mb{Z} \in (g(\expect_Z[\tilde{\mb{f}}_k(\mb{Z})],g(\tilde{\mb{f}}_k(\mb{Z})))$. Denote $(g^{\lambda}(\xi_\mb{Z},\mb{Z}))/\lambda!$ by $\Psi(\mb{Z})$. Further define the operator ${\cal M}(\mb{Z}) = \mb{Z} - \expect[\mb{Z}]$
and
\begin{eqnarray}
p_i &=& {\cal M}(g(\ttti,\mb{X_i})), \nonumber \\
q_i &=& {\cal M}({g'}(\ttti,\mb{X_i})\mb{e}(\mb{X_i})), \nonumber \\
r_i &=& {\cal M}\left(\sum_{i=2}^{\lambda}\frac{g^{(i)}(\ttti,\mb{X_i})}{i!}\mb{e}^i(\mb{X_i})\right) \nonumber \\
s_i &=& {\cal M}\left(\Psi(\mb{X_i})\mb{e}^{\lambda}(\mb{X_i})\right) \nonumber
\end{eqnarray}

The variance of the estimator $\hat{\mb{G}}_N(\mb{\tilde{f}}_k)$ is given by
\begin{eqnarray}
&&\var[\hat{\mb{G}}_N(\mb{\tilde{f}}_k)] = \expect{[({\mb{\hat{G}}}(f)-\expect{[{\mb{\hat{G}}}(f)]})^2]} \nonumber \\
&& = \frac{1}{N}\expect{\left[(p{_1} + q{_1} + r{_1} + s_1)^2\right]} \nonumber \\
&& + \frac{N-1}{N}\expect{\left[(p{_1} + q{_1} + r{_1} + s_1)(p{_2} + q{_2} + r{_2} + s_2)\right]}. \nonumber
\end{eqnarray}
Because $\mb{X}_1$, $\mb{X}_2$ are independent, we have $\expect{\left[(p{_1})(p{_2} + q{_2} + r{_2} + s_2)\right]} = 0$. Furthermore,
\begin{eqnarray}
\expect{\left[(p{_1} + q{_1} + r{_1} + s_1)^2\right]} &=& \expect{[p{_1}^2]} + o(1) = \var[g(\expect_\mb{Z}[\hat{\mb{f}}_\mb{}(\mb{Z})],\mb{Z})] + o(1). \nonumber
\end{eqnarray}
From assumption $({\cal {A}}.4)$ applied to (\ref{bknncent}) and (\ref{bknncross}), in conjunction with assumption $({\cal {A}}.3)$, it follows that
\begin{itemize}
\item $\expect{[p{_1}^2]} =  \var[g(\ttt,\mb{Z})]$
\item $\expect{\left[q{_1}q_{2}\right]} = \var[g'(\ttt,\mb{Z}){{f}}_\mb{}(\mb{Z})]\left(\frac{1}{M}\right) + o\left(\frac{1}{M}\right) + O(N{\cal C}(k))$
\item $\expect{\left[q{_1}r_{2}\right]} = \sum_{i=2}^{\lambda-1} O\left(\frac{1}{k^{((1+i)\delta /2 -1)}M}\right) + O\left(\frac{\lambda(k^{2/d}_M + 1/M)}{M}\right) + O(N{\cal C}(k)) = o\left(\frac{1}{M}\right)  + O(N{\cal C}(k))$
\item $\expect{\left[r{_1}r_{2}\right]} = \sum_{i_1=2}^{\lambda-1} \sum_{i_2=2}^{\lambda-1}O\left(\frac{1}{k^{((i_1+i_2)\delta /2 -1)}M}\right) + O\left(\frac{\lambda^2(k^{2/d}_M + 1/M)}{M}\right) + O(N{\cal C}(k)) = o\left(\frac{1}{M}\right)  + O(N{\cal C}(k))$
\end{itemize}

Since $q_1$ and $s_2$ are $0$ mean random variables
\begin{eqnarray}
&&\expect{\left[q_1s{_2}\right]} = \expect \left[q_1 \Psi(\mb{X_2})(\hat{\mb{f}}_\mb{}(\mb{X_2})-\tttt)^{\lambda} \right] \nonumber \\ 
&& = \expect \left[q_1 \Psi(\mb{X_2})(\hat{\mb{f}}_\mb{}(\mb{X_2})-\tttt)^{\lambda} \right] \nonumber \\ 
&& \leq  \sqrt{\expect \left[ \Psi^2(\mb{X_2})\right]\expect\left[q^2_1(\hat{\mb{f}}_\mb{}(\mb{X_2})-\tttt)^{2\lambda} \right]} \nonumber \\
&& = \sqrt{\expect \left[ \Psi^2(\mb{Z})\right]}\left(o\left(\frac{1}{k^{\lambda}}\right) + O(N{\cal C}(k))\right)\nonumber
\end{eqnarray}
Direct application of Lemma~\ref{boundonexpec} in conjunction with assumptions $({\cal {A}}.5)$, $({\cal {A}}.6)$ implies that $\expect \left[\Psi^2(\mb{Z})\right] = O(1)$. Note that from assumption $({\cal {A}}.3)$, $o\left(\frac{1}{k^{\lambda}}\right) = o(1/M)$ . In a similar manner, it can be shown that $\expect{\left[r{_1}s_{2}\right]} = o\left(\frac{1}{M}\right) + O(N{\cal C}(k))$ and $\expect{\left[s{_1}s_{2}\right]} = o\left(\frac{1}{M}\right) + O(N{\cal C}(k))$. Finally, by (\ref{inbias}) and (\ref{outbiasC}), we know $\ttt = \expect{[\tilde{\mb{f}}_k(\mb{Z})]}=f(\mb{Z}) + o(1)$. This implies that 
\begin{eqnarray}
\var[\hat{\mb{G}}_N(\mb{\tilde{f}}_k)] &=& \frac{1}{N}\expect{\left[p{_1}^2\right]} + \frac{(N-1)}{N}\expect{\left[q{_1}q_{2}\right]} \nonumber + O(N{\cal C}(k))+ o\left(\frac{1}{M}+\frac{1}{N}\right) \nonumber \\
&=& \var[g(\ttt,\mb{Z})]\left(\frac{1}{N}\right)+ \var[g'(\ttt,\mb{Z}){{f}}_\mb{}(\mb{Z})]\left(\frac{1}{M}\right) + O(N{\cal C}(k)) + o\left(\frac{1}{M} + \frac{1}{N}\right) \nonumber \\
&=& \var[g({{f}}_\mb{}(\mb{Z}),\mb{Z})]\left(\frac{1}{N}\right)+ \var[g'({{f}}_\mb{}(\mb{Z}),\mb{Z}){{f}}_\mb{}(\mb{Z})]\left(\frac{1}{M}\right) + O(N{\cal C}(k)) + o\left(\frac{1}{M} + \frac{1}{N}\right) \nonumber \\
&=& c_4\left(\frac{1}{N}\right)+ c_5\left(\frac{1}{M}\right) + O(N{\cal C}(k)) + o\left(\frac{1}{M} + \frac{1}{N}\right), \nonumber
\end{eqnarray}
where the last but one step follows because, by (\ref{inbias}) and (\ref{outbiasC}), we know $\ttt = f(\mb{Z}) + o(1)$. This in turn implies $\var[g(\ttt,\mb{Z})] = \var[g({{f}}_\mb{}(\mb{Z}),\mb{Z})]$ and $ \var[g'(\ttt,\mb{Z}){{f}}_\mb{}(\mb{Z})] = \var[g'({{f}}_\mb{}(\mb{Z}),\mb{Z}){{f}}_\mb{}(\mb{Z})]$. Finally, by assumptions $({\cal {A}}.5)$ and $({\cal {A}}.2)$, the leading constants $c_4$ and $c_5$ are bounded. This concludes the proof of Theorem~\ref{knnvarH}.

Under the logarithmic growth condition $k = O((\log(M))^{2/(1-\delta)})$, $g_2(k,M) = o(1)$ and $g_1(k,M) = 1+o(1)$ by assumption (\ref{eq:gcond}). Theorem~\ref{knnvarHRS} follows by observing that $\hat{\mb{G}}_{N,BC}(\mb{\tilde{f}}_k) = (\hat{\mb{G}}_N(\mb{\tilde{f}}_k)-g_1(k,M))/g_2(k,M)$ 
\end{proof}

\textbf{Bias of Baryshnikov's estimator: Proof of equation (\ref{eq:baryshbias})}
\begin{proof}
We will first prove that 
\begin{equation}
\mathbb{B}(\tilde{\mb{G}}_N(\mb{\hat{f}}_{k})) = \Theta((k/M)^{1/d} + 1/k),
\label{eq:baryshbiaskn}
\end{equation}
Because the standard $k$-NN density estimate $\mb{\hat{f}}_{kS}(\mb{X}_i)$ is identical to the partitioned $k$-NN density estimate $\mb{\hat{f}}_{k}(\mb{X}_i)$ defined on the partition $\{\mb{X}_i\}$ and $\{\mb{X}_1,..,\mb{X}_T\} - \{\mb{X}_i\}$, it follows that
\begin{equation}
\mathbb{B}(\tilde{\mb{G}}_N(\mb{\hat{f}}_{kS})) = \Theta((k/T)^{1/d} + 1/k).
\label{eq:baryshbiask}
\end{equation}

From the definition of set ${\cal S'}$ in section~\ref{sec:intpnt}, we can choose the set ${\cal S'}$, such that $Pr(\mb{Z} \notin {\cal S'}) = O((k/M)^{1/d})$.
\begin{eqnarray}
&&\expect[\hat{\mb{G}}_N(\mb{\hat{f}}_k)] - G(f) = \expect{[{g}(\hat{\mb{f}}_k(\mb{Z}),\mb{Z}) - {g}({{f}(\mb{Z})},\mb{Z})]} \nonumber \\
&& = \expect{[1_{\{\mb{Z} \in {\cal S'}\}}{g}(\hat{\mb{f}}_k(\mb{Z}),\mb{Z}) - {g}({{f}(\mb{Z})},\mb{Z})]} + \expect{[1_{\{\mb{Z} \in {\cal S-S'}\}}{g}(\hat{\mb{f}}_k(\mb{Z}),\mb{Z}) - {g}({{f}(\mb{Z})},\mb{Z})]} \nonumber \\
&& = I+II
\end{eqnarray}

Using the exact same method as in the Proof of Theorem~\ref{knnbiasH}, using (\ref{inbias}) and (\ref{bknncentIapp}), and the fact that $Pr(\mb{Z} \notin {\cal S'}) = O((k/M)^{1/d}) = o(1)$, we have
\begin{eqnarray}
I = \expect{[g'(f(\mb{Z}),\mb{Z})h(\mb{Z})]}\left({\frac{k}{M}}\right)^{2/d} + \expect{[f^2(\mb{Z})g''(f(\mb{Z}),\mb{Z})/2]}\left(\frac{1}{k}\right) + O({\cal C}(k)) + o\left(\frac{1}{k} + \left(\frac{k}{M}\right)^{2/d}\right), \nonumber
\end{eqnarray}

Because we assume that $g$ satisfies assumption $({\cal {A}}.6)$, from the proof of Lemma~\ref{boundonexpec}, for ${Z} \in {\cal S-S'}$, we have $\expect{[{g}(\hat{\mb{f}}_k({Z}),{Z}) - {g}({{f}({Z})},{Z})]} = O(1)$. This implies that, 
\begin{eqnarray}
II &=& \expect{[1_{\{\mb{Z} \in {\cal S-S'}\}}{g}(\hat{\mb{f}}_k(\mb{Z}),\mb{Z}) - {g}({{f}(\mb{Z})},\mb{Z})]} \nonumber \\
&=&\expect\left[\expect{[{g}(\hat{\mb{f}}_k({Z}),{Z}) - {g}({{f}({Z})},{Z})]}\mid 1_{\{\mb{Z} \in {\cal S-S'}\}} \right] \times Pr(\mb{Z} \notin {\cal S'}) \nonumber \\
&=& O(1) \times O((k/M)^{1/d}) = O((k/M)^{1/d}).
\end{eqnarray}
This concludes the proof.

\end{proof}


\section{Asymptotic normality}
Define the random variables $\{\mb{Y}_{M,i}; i = 1,\ldots,N\}$ for any fixed $M$ 
\begin{equation}
\mb{Y}_{M,i} = \frac{g({\mb{\tilde{f}}_k(\mb{X}_i)},\mb{X}_i) - \expect[g({\mb{\tilde{f}}_k(\mb{X}_i)},\mb{X}_i)]}{\sqrt{\var[g({\mb{\tilde{f}}_k(\mb{X}_i)},\mb{X}_i)]}}, \nonumber
\end{equation} and define the sum $\mb{S_{N,M}}$ 
\begin{equation}
\mb{S}_{N,M} = \frac{1}{\sqrt{N}}\sum_{i=1}^N \mb{Y}_{M,i}, \nonumber
\end{equation}
where the indices $N$ and $M$ explicitly stress the dependence of the sum $\mb{S}_{N,M}$ on the number of random variables $N+M$. Observe that the random variables $\{\mb{Y}_{M,i}; i = 1,\ldots,N\}$ belong to an $0$ mean, unit variance, interchangeable process~\cite{chernoff} for all values of $M$. To establish the CLT for $\mb{S}_{N,M}$, we will exploit the fact the random variables $\{\mb{Y}_{M,i}; i = 1,\ldots,N\}$ are interchangeable by appealing to DeFinetti's theorem, which we describe below.

\subsection{De Finetti's Theorem}
Let $\mathcal{F}$ be the class of one dimensional distribution functions and for each pair of real numbers $x$ and $y$ define $\mathcal{F}(x,y) = \{F \in \mathcal{F}|F(x) \leq y\}$. Let $\mathcal{B}$ be the Borel field of subsets of $\mathcal{F}$  generated by the class of sets $\mathcal{F}(x,y)$. Then De Finetti's theorem asserts that for any interchangeable process $\{\mb{Z}_i\}$ there exists a probability measure $\mu$ defined on $\mathcal{B}$ such that
\begin{equation}
Pr\{\mb{B}\} = \int_\mathcal{F} Pr_F\{\mb{B}\} d\mu(F), \label{deFi} 
\end{equation}
for any Borel measurable set defined on the sample space of the sequence $\{\mb{Z}_i\}$. Here $Pr\{\mb{B}\}$ is the probability of the event $\mb{B}$ and $Pr_F\{\mb{B}\}$ is the probability of the event $B$ under the assumption that component random variables $\mb{X}_i$ of the interchangeable process are independent and identically distributed with distribution $F$.

\subsection{Necessary and Sufficient conditions for CLT}
For each $F \in \mathcal{F}$ define $m(F)$ and $\sigma^2(F)$ as $m(F) = \int_{-\infty}^{\infty} x dF(x)$, $\sigma(F) = {\int_{-\infty}^{\infty} x^2 dF(x)} - 1$ and for all real numbers $m$ and non-negative real numbers $\sigma^2$ let $\mathcal{F}_{m,\sigma^2}$ be the set of $F \in \mathcal{F}$ for which $m(F)=m$ and $\sigma^2(F) = \sigma^2$.

Let $\{\mb{Z}_i; i = 1, 2, \ldots\}$ be an interchangeable stochastic process with $0$ mean and variance $1$. Blum~{\it etal}~\cite{chernoff} showed that the random variable $\mb{S}_N = \frac{1}{\sqrt{N}}\sum_{i=1}^N \mb{Z}_i$ converges in distribution to $N(0,1)$ if and only if $\mu(\mathcal{F}_{0,0}) = 1$. Furthermore, they show that the condition $\mu(\mathcal{F}_{0,0}) = 1$ is equivalent to the condition that $Cov(\mb{Z}_1,\mb{Z}_2) = 0$ and $Cov(\mb{Z}_1^2,\mb{Z}_2^2) = 0$. We will {\emph extend} Blum~{\it etal}'s results to interchangeable processes where $Cov(\mb{Z}_1,\mb{Z}_2) = o(1)$ and $Cov(\mb{Z}_1^2,\mb{Z}_2^2) = o(1)$.

In particular, we will show that $Cov(\mb{Y}_{M,1},\mb{Y}_{M,2})$ and $Cov(\mb{Y}_{M,1}^2,\mb{Y}_{M,2}^2)$ are $O(1/M)$. Subsequently we will show that the random variable $\mb{S}_{N,M} = \frac{1}{\sqrt{N}}\sum_{i=1}^N \mb{Y}_{M,i}$ converges in distribution to $N(0,1)$ and conclude that Theorem~\ref{knncltH} holds. 

\subsection{CLT for Asymptotically Uncorrelated processes}
Let $\mb{X}$ be a random variable with density $f$. In the proof of Theorem~\ref{knnvarH}, we showed that
\begin{eqnarray}
Cov(\mb{Y}_{M,i},\mb{Y}_{M,j}) &=& \frac{Cov(g({\mb{\tilde{f}}_k(\mb{X}_i)},\mb{X}_i),g({\mb{\tilde{f}}_k(\mb{X}_j)},\mb{X}_j))}{\sqrt{\var[g({\mb{\tilde{f}}_k(\mb{X}_i)},\mb{X}_i)]\var[g({\mb{\tilde{f}}_k(\mb{X}_j)},\mb{X}_j)]}} \nonumber \\
&=& \frac{Cov(p_i+q_i+r_i+s_i,p_j+q_j+r_j+s_j)}{\sqrt{\var[g({\mb{\tilde{f}}_k(\mb{X}_i)},\mb{X}_i)]\var[g({\mb{\tilde{f}}_k(\mb{X}_j)},\mb{X}_j)]}} \nonumber \\
&=& \frac{Cov(p_i+q_i+r_i+s_i,p_j+q_j+r_j+s_j)}{\sqrt{\var[g({\mb{\tilde{f}}_k(\mb{X}_i)},\mb{X}_i)]\var[g({\mb{\tilde{f}}_k(\mb{X}_j)},\mb{X}_j)]}} \nonumber \\
\nonumber \\
&=& \frac{\var(g'(f(\mb{X}),\mb{X})f(\mb{X}))}{\var[g({{{f}}(\mb{X}_i)},\mb{X}_i)]}\left(\frac{1}{M}\right) + o\left(\frac{1}{M}\right) + O(N{\cal C}(k)) \nonumber \\
&=& \frac{\var(g'(f(\mb{X}),\mb{X})f(\mb{X}))}{\var[g({{{f}}(\mb{X}_i)},\mb{X}_i)]}\left(\frac{1}{M}\right) + o\left(\frac{1}{M}\right),
\label{eq:covG}
\end{eqnarray}
where the last but one step follows by observing that $N{\cal C}(k)/M \to 0$ under the logarithmic growth condition $k = O((\log(M))^{2/(1-\delta)})$. Define the function $d(x,y) = g(x,y)(g(x,y)-c)$, where the constant $c = \expect[g({\mb{\tilde{f}}_k(\mb{X})},\mb{X})]$. Then, similar to the derivation of (\ref{eq:covG}), we have, 
\begin{eqnarray}
Cov(\mb{Y}_{M,i}^2,\mb{Y}_{M,j}^2) &=& \frac{Cov(d({\mb{\tilde{f}}_k(\mb{X}_i)},\mb{X}_i),d({\mb{\tilde{f}}_k(\mb{X}_j)},\mb{X}_j))}{\sqrt{\var[d({\mb{\tilde{f}}_k(\mb{X}_i)},\mb{X}_i)]\var[d({\mb{\tilde{f}}_k(\mb{X}_j)},\mb{X}_j)]}} \nonumber \\
&=& \frac{\var(d'(f(\mb{X}),\mb{X})f(\mb{X}))}{\var[d({{{f}}(\mb{X}_i)},\mb{X}_i)]}\left(\frac{1}{M}\right) + o\left(\frac{1}{M}\right).
\end{eqnarray}
\textbf{Proof of Theorem~\ref{knncltH} and Theorem~\ref{knncltHRS}}.
\begin{proof}

Let $\delta_{\mu}(M)$ and $\delta_\sigma(M)$ be a strictly positive functions parameterized by $M$ such that $\delta_{\mu}(M) = o(1) ; \frac{1}{M\delta_{\mu}(M)} = o(1)$, $\delta_{\sigma}(M) = o(1) ; \frac{1}{M\delta_{\sigma}(M)} = o(1)$. Denote the set of $F \in \mathcal{F}$ with $\mathcal{F}_{m,\delta,M}:=\{m^2(F)\geq \delta_\mu(M)\}$; $\mathcal{F}_{\sigma,\delta,M}: =\{\sigma^2(F)\geq \delta_\sigma(M)\}$; $\mathcal{F}_{m,\delta,M}^{*}:=\{m^2(F)\in (0,\delta_\mu(M))\}$ and $\mathcal{F}_{\sigma,\delta,M}^{*}:=\{\sigma^2(F)\in (0,\delta_\sigma(M))\}$. Denote the measures of these sets by $\mu_{m,\delta,M}$, $\mu_{\sigma,\delta,M}$, $\mu_{m,\delta,M}^*$ and $\mu_{\sigma,\delta,M}^*$ respectively. We have from (\ref{deFi}) that
\begin{eqnarray}
\int_\mathcal{F} m^2(F) d\mu(F) &=& Cov(\mb{Y}_{M,i},\mb{Y}_{M,j}) \nonumber \\
\int_\mathcal{F} \sigma^2(F) d\mu(F) &=& \int_\mathcal{F} [\expect_F[\mb{Z}^2-1]]^2 d\mu(F) = Cov(\mb{Y}_{M,i}^2,\mb{Y}_{M,j}^2).
\label{eq:covrelation}
\end{eqnarray}
Applying the Chebyshev inequality, we get
\begin{eqnarray}
\label{chebineq}
\delta_\mu(M)\mu_{m,\delta,M} \leq Cov(\mb{Y}_{M,i},\mb{Y}_{M,j}), \nonumber \\
\delta_\sigma(M)\mu_{\sigma,\delta,M} \leq Cov(\mb{Y}_{M,i}^2,\mb{Y}_{M,j}^2) \nonumber.
\end{eqnarray}
Because the covariances decay at $O(1/M)$, $\mu_{m,\delta,M}$ and $\mu_{\sigma,\delta,M} \to 0$ as $M \to \infty$. From the definition of $\mathcal{F}_{m,\delta,M}^*$ and $\mathcal{F}_{\sigma,\delta,M}^*$, we also have that $\mu_{m,\delta,M}^*$ and $\mu_{\sigma,\delta,M}^* \to 0$ as $M \to \infty$. We also have
\begin{equation}
1 - (\mu_{m,\delta,M} + \mu_{\sigma,\delta,M} + \mu_{m,\delta,M}^* + \mu_{\sigma,\delta,M}^*) \leq \mu(\mathcal{F}_{0,0}) \leq 1, \nonumber
\end{equation}
and therefore \begin{eqnarray}
\lim_{M \to \infty} \mu(\mathcal{F}_{0,0}) = 1.
\label{mulimit}
\end{eqnarray} 

We will now show that $\tilde{\mb{G}}_N(\mb{\tilde{f}}_k) = ({\hat{\mb{G}}_N(\mb{\tilde{f}}_k)-\expect[\hat{\mb{G}}_N(\mb{\tilde{f}}_k)]})/({{\sqrt{\var[\hat{\mb{G}}_N(\mb{\tilde{f}}_k)]}}})$ converges weakly to $\mathbb{N}(0,1)$. Denote $g({\mb{\tilde{f}}_k(\mb{X}_i)},\mb{X}_i)$ by $\mb{g}_i$. Observe that
\begin{eqnarray}
&& \lim_{\Delta \to 0} Pr\{\tilde{\mb{G}}_N(\mb{\tilde{f}}_k) \leq \alpha\} = \lim_{\Delta \to 0} \int_\mathcal{F} Pr_F\{\tilde{\mb{G}}_N(\mb{\tilde{f}}_k)  \leq \alpha\} d\mu(F) \nonumber \\
&& = \lim_{\Delta \to 0} \int_{\mathcal{F}_{0,0}} Pr_F\{\tilde{\mb{G}}_N(\mb{\tilde{f}}_k) \leq \alpha\} d\mu(F) + \lim_{\Delta \to 0} \int_{\mathcal{F}} 1_{\{F \in \mathcal{F} - \mathcal{F}_{0,0}\}} Pr_F\{\tilde{\mb{G}}_N(\mb{\tilde{f}}_k) \leq \alpha\} d\mu(F) \nonumber \\
&& = \lim_{\Delta \to 0} \int_{\mathcal{F}_{0,0}} Pr_F\{\tilde{\mb{G}}_N(\mb{\tilde{f}}_k) \leq \alpha\} d\mu(F) +  \int_{\mathcal{F}} \lim_{\Delta \to 0} \left(1_{\{F \in \mathcal{F} - \mathcal{F}_{0,0}\}} Pr_F\{\tilde{\mb{G}}_N(\mb{\tilde{f}}_k) \leq \alpha\}\right) d\mu(F) \\
&& = \lim_{\Delta \to 0} \int_{\mathcal{F}_{0,0}} Pr_F\{\tilde{\mb{G}}_N(\mb{\tilde{f}}_k) \leq \alpha\} d\mu(F) \\
&& = \lim_{\Delta \to 0}  \int_{\mathcal{F}_{0,0}} Pr_F \left\{ \frac{1}{N}\sum_{i=1}^{N} \left( \frac{g({\mb{\tilde{f}}_k(\mb{X}_i)},\mb{X}_i) - \expect[g({\mb{\tilde{f}}_k(\mb{X}_i)},\mb{X}_i)]}{\sqrt{\var[\hat{\mb{G}}_{N}(\mb{\tilde{f}}_k)]}} \right) \leq \alpha\ \right\} d\mu(F) \nonumber \\
&& = \lim_{\Delta \to 0} \int_{\mathcal{F}_{0,0}} Pr_F \left\{ \frac{1}{N}\sum_{i=1}^{N} \left( \frac{g({\mb{\tilde{f}}_k(\mb{X}_i)},\mb{X}_i) - \expect[g({\mb{\tilde{f}}_k(\mb{X}_i)},\mb{X}_i)]}{\sqrt{\var[\mb{g}_i]/N+((N-1)/N)Cov[\mb{g}_i,\mb{g}_j]}} \right) \leq \alpha\ \right\} \int_{\mathcal{F}_{0,0}} d\mu(F) \nonumber \\
&& = \lim_{\Delta \to 0} \int_{\mathcal{F}_{0,0}} Pr_F \left\{ \frac{1}{N}\sum_{i=1}^{N} \left( \frac{g({\mb{\tilde{f}}_k(\mb{X}_i)},\mb{X}_i) - \expect[g({\mb{\tilde{f}}_k(\mb{X}_i)},\mb{X}_i)]}{\sqrt{\var[\mb{g}_i]/N+((N-1)/N)\sqrt{\var[\mb{g}_i]\var[\mb{g}_j]}Cov[\mb{Y}_{M,i},\mb{Y}_{M,j}]}} \right) \leq \alpha\  \right\} d\mu(F) \nonumber \\
&& = \lim_{\Delta \to 0} \int_{\mathcal{F}_{0,0}} Pr_F \left\{ \frac{1}{N}\sum_{i=1}^{N} \left( \frac{g({\mb{\tilde{f}}_k(\mb{X}_i)},\mb{X}_i) - \expect[g({\mb{\tilde{f}}_k(\mb{X}_i)},\mb{X}_i)]}{\sqrt{\var[\mb{g}_i]/N}} \right) \leq \alpha\ \right\} d\mu(F) \\
&& = \lim_{\Delta \to 0} \int_{\mathcal{F}_{0,0}} Pr_F \left\{ \frac{1}{\sqrt{N}}\sum_{i=1}^N \mb{Y}_{M,i} \leq \alpha\ \right\} d\mu(F) \nonumber \\
&& = \int_{\mathcal{F}} \lim_{\Delta \to 0} \left(1_{\{F \in \mathcal{F}_{0,0}\}} Pr_F \left\{\frac{1}{\sqrt{N}}\sum_{i=1}^N \mb{Y}_{M,i} \leq \alpha\ \right\}\right) d\mu(F) \nonumber \\
&& = \int_{\mathcal{F}} \phi(\alpha) d\mu(F) = \phi(\alpha),
\end{eqnarray}
where $\phi(.)$ is the distribution function of a Gaussian random variable with mean $0$ and variance $1$. Step (D.6) follows from the Dominated Convergence theorem. By (\ref{mulimit}), $\lim_{\Delta \to 0} 1_{\{F \in \mathcal{F} - \mathcal{F}_{0,0}\}} = 0$ almost surely. This gives Step (D.7). Step (D.8) is obtained by observing that, by (\ref{eq:covrelation}),  $Cov[\mb{Y}_{M,i},\mb{Y}_{M,j}] = 0$ when $F \in \mathcal{F}_{0,0}$. The last step (D.9) follows from the CLT for sums of $0$ mean, unit variance, i.i.d random variables and (\ref{mulimit}). This concludes the proof of Theorem~\ref{knncltH}.

To show Theorem~\ref{knncltHRS}, observe that under the logarithmic growth condition $k = O((\log(M))^{2/(1-\delta)})$, $g_2(k,M) = o(1)$ and $g_1(k,M) = 1+o(1)$ by assumption (\ref{eq:gcond}). Since $\hat{\mb{G}}_{N,BC}(\mb{\tilde{f}}_k) = (\hat{\mb{G}}_N(\mb{\tilde{f}}_k)-g_1(k,M))/g_2(k,M)$, it follows that the asymptotic distribution of  $$\frac{\hat{\mb{G}}_{N,BC}(\mb{\tilde{f}}_k)-\expect[\hat{\mb{G}}_{N,BC}(\mb{\tilde{f}}_k)]}{{\sqrt{\var[\hat{\mb{G}}_{N,BC}(\mb{\tilde{f}}_k)]}}}$$ is equal to the asymptotic distribution of $\tilde{\mb{G}}_N(\mb{\tilde{f}}_k) = ({\hat{\mb{G}}_N(\mb{\tilde{f}}_k)-\expect[\hat{\mb{G}}_N(\mb{\tilde{f}}_k)]})/({{\sqrt{\var[\hat{\mb{G}}_N(\mb{\tilde{f}}_k)]}}})$.

\end{proof}
\subsection{Berry-Esseen bounds}

We now establish Berry-Esseen bounds for the case where $\frac{N}{M} \to 0$. In particular, we assume that there exists a $\delta : 0<\delta<1$, such that $N = O(M^\delta)$. We also assume that the interchangeable process has finite absolute third order moment $E(|\mb{Z}_{M,i}|^3) = \rho_M < \infty$ $\vee M$.

\subsubsection{Details}

Define the subset $\tilde{\digamma}$ of $\digamma$ as follows: 
$\tilde{\digamma} =\digamma - \{\digamma_{m,\delta,M} \bigcup \digamma_{\sigma,\delta,M}\}$. 

We recognize that for $F \in \tilde{\digamma}$, we have 

\begin{eqnarray}
-\sqrt{\delta_\mu(M)} \leq m(F) \leq \sqrt{\delta_\mu(M)}, \nonumber \\
-\sqrt{\delta_\sigma(M)} \leq \sigma(F) \leq \sqrt{\delta_\sigma(M)}. \nonumber 
\end{eqnarray} 

The mean and variance of $Y_{M,i}$ under the distribution $F$ are given by $m(F)$ and $\sigma(F) + \rho -m^2(F)$ respectively.

As in the previous section, let 
$\phi$ be the distribution function of a Gaussian random variable with $0$ mean and $\rho$ variance.

\subsubsection*{Lower bound}

\begin{eqnarray}
&& Pr\{\mb{S}_{N,M} \leq \alpha\} = \int_\digamma Pr_F\{\mb{S}_{N,M} \leq \alpha\} d\mu(F) \nonumber \\
&& \geq \int_{\tilde{\digamma}} Pr_F\{\mb{S}_{N,M} \leq \alpha\} d\mu(F) \nonumber \\
&& \geq \int_{\tilde{\digamma}} \left[\phi\left(\frac{\alpha-\sqrt{N}m(F)}{1+(\sigma(F)-m^2(F))/\rho}\right) -\frac{C \kappa(F)}{(\sigma(F)+\rho-m^2(F))^3\,\sqrt{N}}\right] d\mu(F) \nonumber \\
&& \geq \phi\left(\frac{\alpha-\sqrt{N\delta_\mu(M)}}{1+(\sqrt{\delta_\sigma(M)})/\rho}\right)\mu(\tilde{\digamma}) - 
\int_{\tilde{\digamma}} \frac{C \kappa(F)}{(\rho-\sqrt{\delta_\sigma(M)}-\delta_\mu(M))^3\,\sqrt{N}} d\mu(F) \nonumber \\
&& \geq \phi\left(\frac{\alpha-\sqrt{N\delta_\mu(M)}}{1+(\sqrt{\delta_\sigma(M)})/\rho}\right)\mu(\tilde{\digamma}) - 
\frac{C \kappa}{(\rho-\sqrt{\delta_\sigma(M)}-\delta_\mu(M))^3\,\sqrt{N}}.  \nonumber 
\end{eqnarray}

\subsubsection*{Upper bound}

Denote  $\mu{(\tilde{\digamma}^c)} := \tilde{\mu}$. We note that 
$\tilde{\mu} \leq \mu_{m,\delta,M} + \mu_{\sigma,\delta,M}$.

\begin{eqnarray}
&& Pr\{\mb{S}_{N,M} \leq \alpha\} = \int_\digamma Pr_F\{\mb{S}_{N,M} \leq \alpha\} d\mu(F) \nonumber \\
&& \leq \int_{\tilde{\digamma}} Pr_F\{\mb{S}_{N,M} \leq \alpha\} d\mu(F) + \tilde{\mu} \nonumber\\
&& \leq \int_{\tilde{\digamma}} \left[\phi\left(\frac{\alpha-\sqrt{N}m(F)}{1+(\sigma(F)-m^2(F))/\rho}\right) + \frac{C \kappa(F)}{(\sigma(F)+\rho-m^2(F))^3\,\sqrt{N}}\right] d\mu(F) + \tilde{\mu} \nonumber \\
&& \leq \phi\left(\frac{\alpha+\sqrt{N\delta_\mu(M)}}{1-(\sqrt{\delta_\sigma(M)}+\delta_\mu(M))/\rho}\right)\mu(\tilde{\digamma}) +
\int_{\tilde{\digamma}} \frac{C \kappa(F)}{(\rho+\sqrt{\delta_\sigma(M)})^3\,\sqrt{N}} d\mu(F) + \tilde{\mu} \nonumber \\
&& \leq \phi\left(\frac{\alpha-\sqrt{N\delta_\mu(M)}}{1-(\sqrt{\delta_\sigma(M)}+\delta_\mu(M))/\rho}\right)\mu(\tilde{\digamma}) + 
\frac{C \kappa}{(\rho+\sqrt{\delta_\sigma(M)})^3\,\sqrt{N}} + \mu_{m,\delta,M} + \mu_{\sigma,\delta,M} \nonumber \\
&& \leq \phi\left(\frac{\alpha-\sqrt{N\delta_\mu(M)}}{1-(\sqrt{\delta_\sigma(M)}+\delta_\mu(M))/\rho}\right)\mu(\tilde{\digamma}) + 
\frac{C \kappa}{(\rho+\sqrt{\delta_\sigma(M)})^3\,\sqrt{N}} + \frac{1}{M\delta_\mu(M)} + \frac{1}{M\delta_\sigma(M)}. \nonumber
\end{eqnarray}

We have shown that the appropriately normalized sum $S_{N,M}$ converges in distribution to a normal random variable. Also for the case where $N$ grows slower than $M$, we have established Berry-Esseen type bounds on the error.


\section{Uniform kernel based plug-in estimator}

In this section, we will state the main results concerning uniform kernel plug-in estimators. The proofs for these results rely on the properties of the uniform kernel density estimates established in Appendix A and proofs for equivalent results for the $k$-NN plug-in estimators. Let $\hat{\mb{f}}_u$ denote the boundary corrected uniform kernel density estimate. Denote the uniform kernel plug-in estimator by

\begin{eqnarray}
{\mb{\hat{G}_u}}(f) &=& \left(\frac{1}{N}\sum_{i=1}^N g({\hat{\mb{f}_u}(\mb{X}_i)},\mb{X}_i)\right).
\end{eqnarray}
Let $\mb{Y}$ denote a random variable with density function $f$.

\subsection{Results}

\begin{corollary}
\label{Biasunif}

Suppose that the density $f$, the functional $g$ and the density estimate $\hat{\mb{f}}_u$ satisfy the necessary conditions listed above. The bias of the plug-in estimator $\hat{\mb{G}}_u(f)$ is then given by
\begin{eqnarray}
B_u(f) &=& c_{1}\left({\frac{k}{M}}\right)^{2/d} + c_{2}\left(\frac{1}{k}\right) + o\left(\frac{1}{k} + \left(\frac{k}{M}\right)^{2/d}\right), \nonumber
\end{eqnarray}
where $c_1 = \expect{[g'(f(\mb{Y}),\mb{Y})c(\mb{Y})]}$, $c_2 = \expect{[g''(f(\mb{Y}),\mb{Y})f(\mb{Y})/2]}$ are constants which depend on the underlying density $f$.
\end{corollary}

\begin{corollary}
\label{Varianceunif}
Suppose that the density $f$, the functional $g$ and the density estimate $\hat{\mb{f}}_u$ satisfy the necessary conditions listed above. The variance of the plug-in estimator $\hat{\mb{G}}_u(f)$ is given by

\begin{eqnarray}
\var_u(f) &=& c_4\left(\frac{1}{N}\right)+ c_5\left(\frac{1}{M}\right) + o\left(\frac{1}{M} + \frac{1}{N}\right), \nonumber
\end{eqnarray}
where $c_4=\var[g(f(\mb{Y}),\mb{Y})]$ and $c_5=\var[f(\mb{Y})g'(f(\mb{Y}),\mb{Y})]$ are constants which depend on the underlying density $f$.
\end{corollary}

\begin{corollary}
\label{CLTunif}
Suppose that the density $f$, the functional $g$ and the density estimate $\hat{\mb{f}}_u$ satisfy the necessary conditions listed above. Further suppose $\expect[|g(f)|^3]$ is finite. The asymptotic distribution of the plug-in estimator $\hat{\mb{G}}_u(f)$ is given by
\begin{equation}
\lim_{\Delta(k,N,M) \to 0} Pr\left(\frac{\hat{\mb{G}}_u(f)-\expect[\hat{\mb{G}}_u(f)]}{{\sqrt{\var[f(\mb{Y})g'(f(\mb{Y}),\mb{Y})]/N}}} \leq \alpha \right) = Pr(\mb{Z} \leq \alpha), \nonumber
\end{equation}
where $\mb{Z}$ is a standard normal random variable.
\end{corollary}


\bibliographystyle{plain}	
\bibliography{ref}

\begin{thebibliography}{10}

\bibitem{ahmad}
I.~Ahmad and {Pi-Erh Lin}.
\newblock A nonparametric estimation of the entropy for absolutely continuous
  distributions (corresp.).
\newblock {\em Information Theory, IEEE Transactions on}, 22(3):372 -- 375, may
  1976.

\bibitem{bar}
Yu. Baryshnikov, M.~D. Penrose, and J.E. Yukich.
\newblock {Gaussian limits for generalized spacings}.
\newblock {\em Ann. Appl. Probab.}, 19(1):158--185, 2009.

\bibitem{bickel}
P.~J. Bickel and Y.~Ritov.
\newblock Estimating integrated squared density derivatives: Sharp best order
  of convergence estimates.
\newblock {\em Sankhya: The Indian Journal of Statistics}, 50:381--393, October
  1988.

\bibitem{birge}
L.~Birge and P.~Massart.
\newblock Estimation of integral functions of a density.
\newblock {\em The Annals of Statistics}, 23(1):11--29, 1995.

\bibitem{chernoff}
J.R. Blum, H.~Chernoff, M.~Rosenblatt, and H.~Teicher.
\newblock Central limit theorems for interchangeable processes.
\newblock {\em Canadian Journal of Mathematics}, June 1957.

\bibitem{chen}
Y.~Chen, A.~Wiesel, and A.~O. Hero.
\newblock Robust shrinkage estimation of high-dimensional covariance matrices.
\newblock submitted to IEEE Trans. on Signal Process., preprint available in
  arXiv:1009.5331.

\bibitem{amin}
R.~C.~H. Cheng and N.~A.~K. Amin.
\newblock Estimating parameters in continuous univariate distributions with a
  shifted origin.
\newblock {\em Journal of the Royal Statistical Society. Series B
  (Methodological)}, 11:394--403, 1983.

\bibitem{chow}
C.~I. Chow and C.~N. Liu.
\newblock Approximating discrete probability distributions with dependence
  trees.
\newblock {\em IEEE Transactions on Information Theory}, 14:462--467, 1968.

\bibitem{Costa&etal:ICASSP04}
J.A. Costa, A.~Girotra, and A.O. Hero.
\newblock {Estimating local intrinsic dimension with k-nearest neighbor
  graphs}.
\newblock In {\em 2005 IEEE/SP 13th Workshop on Statistical Signal Processing},
  pages 417--422, 2005.

\bibitem{dude}
E.~J. Dudewicz and E.~C. van~der Meulen.
\newblock Entropy-based tests of uniformity.
\newblock {\em Journal of the American Statistical Association}, 76:967--974,
  1981.

\bibitem{egg}
P.~B. Eggermont and V.~N. LaRiccia.
\newblock Best asymptotic normality of the kernel density entropy estimator for
  smooth densities.
\newblock {\em Information Theory, IEEE Transactions on}, 45(4):1321 --1326,
  May 1999.

\bibitem{evans}
D.~Evans.
\newblock A law of large numbers for nearest neighbor statistics.
\newblock {\em Proceedings of the Royal Society A}, 464:3175--3192, 2008.

\bibitem{evjo}
D.~Evans, A.~Jones, and W.~M. Schmidt.
\newblock Asymptotic moments of nearest neighbor distance distributions.
\newblock {\em Proceedings of the Royal Society A}, 458:2839--2849, 2008.

\bibitem{Farahmand&etal:ICML07}
A.M. Farahmand, C.~Sepesvari, and J-Y Audibert.
\newblock Manifold-adaptive dimension estimation.
\newblock {\em Proc of 24th Intl Conf on Machine Learning}, pages 265--272,
  2007.

\bibitem{fuk2}
K.~Fukunaga and L.~D. Hostetler.
\newblock Optimization of k-nearest-neighbor density estimates.
\newblock {\em IEEE Transactions on Information Theory}, 1973.

\bibitem{gini}
E.~Gin\'e and D.M. Mason.
\newblock Uniform in bandwidth estimation of integral functionals of the
  density function.
\newblock {\em Scandinavian Journal of Statistics}, 35:739–761, 2008.

\bibitem{leo}
M.~Goria, N.~Leonenko, V.~Mergel, and P.~L.~Novi Inverardi.
\newblock A new class of random vector entropy estimators and its applications
  in testing statistical hypotheses.
\newblock {\em Nonparametric Statistics}, 2004.

\bibitem{hallmarr}
Peter Hall and J.~S. Marron.
\newblock Estimation of integrated squared density derivatives.
\newblock {\em Stat. Prob. Lett}, pages 109--115, 1987.

\bibitem{hero}
A.~O. Hero, J.~Costa, and B.~Ma.
\newblock Asymptotic relations between minimal graphs and alpha-entropy.
\newblock {\em Technical Report CSPL-334 Communications and Signal Processing
  Laboratory, The University of Michigan}, March 2003.

\bibitem{heroapp}
A.~O. Hero, B.~Ma, O.~Michel, and J.~Gorman.
\newblock Applications of entropic spanning graphs.
\newblock {\em Signal Processing Magazine, IEEE}, 19(5):85 -- 95, sep 2002.

\bibitem{edge}
Marc M.~Van Hulle.
\newblock Edgeworth approximation of multivariate differential entropy.
\newblock {\em Neural Computation}, 17(9):1903--1910, 2005.

\bibitem{ihler}
A.~T. Ihler, J.~W. {Fisher III}, and A.~S. Willsky.
\newblock Nonparametric hypothesis tests for statistical dependency.
\newblock {\em IEEE Transactions on Signal Processing}, 52(8):2234--2249,
  August 2004.

\bibitem{jainak}
A.K. Jain.
\newblock Image data compression: A review.
\newblock {\em Proceedings of the IEEE}, 69(3):349 -- 389, March 1981.

\bibitem{lakh}
A.~Lakhina, M.~Crovella, and C.~Diot.
\newblock Mining anomalies using traffic feature distributions.
\newblock In {\em In ACM SIGCOMM}, pages 217--228, 2005.

\bibitem{laurent}
B.~Laurent.
\newblock Efficient estimation of integral functionals of a density.
\newblock {\em The Annals of Statistics}, 24(2):659--681, 1996.

\bibitem{leo2}
N.~Leonenko, L.~Prozanto, and V.~Savani.
\newblock A class of r\'enyi information estimators for multidimensional
  densities.
\newblock {\em Annals of Statistics}, 36:2153--2182, 2008.

\bibitem{Leonenko&etal:AnnStat08}
N.~Leonenko, L.~Prozanto, and V.~Savani.
\newblock A class of r\'enyi information estimators for multidimensional
  densities.
\newblock {\em Annals of Statistics}, 36:2153--2182, 2008.

\bibitem{levbick}
E.~Levina and P.~J. Bickel.
\newblock Maximum likelihood estimation of intrinsic dimension.
\newblock In {\em Advances in Neural Information Processing Systems},
  Cambridge, MA, 2005.

\bibitem{litt}
E.~Liiti\"{a}inen, A.~Lendasse, and F.~Corona.
\newblock On the statistical estimation of r\'{e}nyi entropies.
\newblock In {\em Proceedings of {IEEE}/{MLSP} 2009 International Workshop on
  Machine Learning for Signal Processing, Grenoble (France)}, September 2-4
  2009.

\bibitem{quu}
D.~O. Loftsgaarden and C.~P. Quesenberry.
\newblock A nonparametric estimate of a multivariate density function.
\newblock {\em Ann. Math. Statist.}, 1965.

\bibitem{fuk}
Y.~P. Mack and M.~Rosenblatt.
\newblock Multivariate k-nearest neighbor density estimates.
\newblock {\em Journal of Multivariate Analysis}, 9(1):1 -- 15, 1979.

\bibitem{milfish}
E.~G. Miller and J.~W. {Fisher III}.
\newblock {\MakeUppercase{ICA}} using spacings estimates of entropy.
\newblock {\em Proc. 4th Intl. Symp. on ICA and BSS}, pages 1047--1052, 2003.

\bibitem{rst}
D.~S. Moore and J.~W. Yackel.
\newblock Consistency properties of nearest neighbor density function
  estimators.
\newblock {\em The Annals of Statistics}, 1977.

\bibitem{neem}
H.~Neemuchwala and A.~O. Hero.
\newblock Image registration in high dimensional feature space.
\newblock {\em Proc. of SPIE Conference on Electronic Imaging, San Jose},
  January 2005.

\bibitem{long2}
X.~Nguyen, M.~J. Wainwright, and M.~I. Jordan.
\newblock Estimating divergence functionals and the likelihood ratio by convex
  risk minimization.
\newblock {\em Information Theory, IEEE Transactions on}, 56(11):5847 --5861,
  November 2010.

\bibitem{pal}
D.~{P{\'a}l}, B.~{P{\'o}czos}, and C.~{Szepesv{\'a}ri}.
\newblock {Estimation of R$\backslash$'enyi Entropy and Mutual Information
  Based on Generalized Nearest-Neighbor Graphs}.
\newblock {\em ArXiv e-prints}, March 2010.

\bibitem{ranneby}
B.~Ranneby.
\newblock The maximum spacing method. an estimation method related to the
  maximum likelihood method.
\newblock {\em Scandinavian Journal of Statistics}, 11:93--112, 1984.

\bibitem{raykar}
V.~C. Raykar and R.~Duraiswami.
\newblock Fast optimal bandwidth selection for kernel density estimation.
\newblock In J.~Ghosh, D.~Lambert, D.~Skillicorn, and J.~Srivastava, editors,
  {\em Proceedings of the sixth SIAM International Conference on Data Mining},
  pages 524--528, 2006.

\bibitem{multi}
Xavier~Saint Raymond.
\newblock {\em Elementary Introduction to the Theory of Pseudodifferential
  Operators}.
\newblock CRC Press, 1991.

\bibitem{sing}
H.~Singh, N.~Misra, and V.~Hnizdo.
\newblock Nearest neighbor estimators of entropy.
\newblock {\em The Annals of Statistics}, 2005.

\bibitem{kks2}
K.~Sricharan, R.~Raich, and A.~O. Hero.
\newblock Global performance prediction for divergence-based image registration
  criteria.
\newblock In {\em Proc. IEEE Workshop on Statistical Signal Processing}, 2009.

\bibitem{kks}
K.~Sricharan, R.~Raich, and A.~O. Hero.
\newblock {Empirical estimation of entropy functionals with confidence}.
\newblock {\em ArXiv e-prints}, December 2010.

\bibitem{vanes}
B.~van Es.
\newblock Estimating functionals related to a density by class of statistics
  based on spacing.
\newblock {\em Scandinavian Journal of Statistics}, 1992.

\bibitem{vasi}
O.~Vasicek.
\newblock A test for normality based on sample entropy.
\newblock {\em Journal of the Royal Statistical Society. Series B
  (Methodological)}, 38:54--59, 1976.

\bibitem{wang}
Q.~Wang, S.~R. Kulkarni, and S.~Verd{\'u}.
\newblock {Divergence estimation of continuous distributions based on
  data-dependent partitions}.
\newblock {\em Information Theory, IEEE Transactions on}, 51(9):3064--3074,
  2005.

\end{thebibliography}

\end{document}